\documentclass[10pt,a4paper, reqno]{amsart}
\usepackage{mathrsfs}

\usepackage{amsmath,amsfonts,verbatim}
\usepackage{latexsym}
\usepackage{amssymb,leftidx}
\usepackage{extarrows}
\usepackage{overpic}
\usepackage{color}
\usepackage{epsfig}
\usepackage{subfigure}
\usepackage{tikz}
\usepackage{bbm}
\usepackage{tikz, caption}
\usepackage[font=small,labelfont=bf]{caption}

\usepackage[colorlinks=true]{hyperref}
\hypersetup{urlcolor=blue, citecolor=blue}

\usepackage{amssymb}
\usepackage{mathrsfs}
\usepackage{amscd}
\usepackage{cite}

\newcommand\R{\mathbb{R}}
\newcommand\C{\mathbb{C}}
\newcommand\Z{\mathbb{Z}}
\newcommand\N{\mathbb{N}}

\def\ovr{\overline}

\newcommand{\TsnCHself}[3]{L^{#1}_t( #3 , \mathcal{H}^{#2}_x(\R^3) )}

\newcommand{\Tsnself}[3]{L^{#1}_t( #3 , L^{#2}_x(\R^3) )}
\newcommand{\LpnR}[1]{L^{#1}_x(\R^3)}

\newcommand{\Tsnselfd}[3]{L^{#1}_t( #3 , L^{#2}_x(\R^d) )}
\newcommand{\LpnRd}[1]{L^{#1}_x(\R^d)}

\newcommand{\CHsnR}[1]{\mathcal{H}^{#1}_x(\R^3)}

\newcommand{\CHsnRd}[1]{\mathcal{H}^{#1}_x(\R^d)}

\newcommand{\Lag}{\langle}
\newcommand{\Rag}{\rangle}
\newcommand{\Lf}{\left}
\newcommand{\Rt}{\right}
\newcommand{\Vt}{\Vert}
\newcommand{\vt}{\vert}

\newcommand{\wt}[1]{\widetilde{#1}}

\newcommand{\Ind}[1]{\mathrm{d}#1}

\newcommand{\ntof}{\wt{\mu}_1,\wt{\mu}_2,\wt{\mu}_3,\wt{\mu}_4}
\newcommand{\nof}{\mu_1,\mu_2,\mu_3,\mu_4}
\newcommand{\thetaof}{\theta_1,\theta_2,\theta_3,\theta_4}

\usepackage{graphicx}
\usepackage{float}

\numberwithin{equation}{section}
\newtheorem{proposition}{Proposition}[section]
\newtheorem{definition}{Definition}[section]
\newtheorem{lemma}{Lemma}[section]
\newtheorem{theorem}{Theorem}[section]
\newtheorem{corollary}{Corollary}[section]
\newtheorem{remark}{Remark}[section]

\begin{document}

\title[Nonlinear Schr\"{o}dinger equations]{ \bf On Growth of Sobolev norms  for  cubic  Schr\"odinger equation with harmonic potential in dimensions $d=2,3$ }

\author{Yilin Song}
\address{Yilin Song
	\newline \indent The Graduate School of China Academy of Engineering Physics,
		Beijing 100088,\ P. R. China}
\email{songyilin21@gscaep.ac.cn}
	
\author{Ruixiao Zhang}
\address{Ruixiao Zhang
		\newline \indent The Graduate School of China Academy of Engineering Physics,
		Beijing 100088,\ P. R. China}
\email{zhangruixiao21@gscaep.ac.cn}
	
\author{Jiqiang Zheng}
\address{Jiqiang Zheng
		\newline \indent Institute of Applied Physics and Computational Mathematics,
		Beijing, 100088, China.
		\newline\indent
		National Key Laboratory of Computational Physics, Beijing 100088, China}
\email{zheng\_jiqiang@iapcm.ac.cn, zhengjiqiang@gmail.com}

\maketitle
	
\begin{abstract}

In this article, we study  the growth of higher-order Sobolev norms for solutions to the defocusing cubic nonlinear Schr\"odinger equation with harmonic potential in dimensions $d=2,3$,
\begin{align}\label{PNLS}
\begin{cases}\tag{PNLS}
i\partial_tu-Hu=\left|u\right|^{2}u,&(t,x)\in\R\times\R^d,\\
u(0,x)=u_0(x),
\end{cases}
\end{align}
where $H=-\Delta+|x|^2$. Motivated by Planchon-Tzvetkov-Visciglia [Rev. Mat. Iberoam., 39 (2023), 1405-1436], we first establish the bilinear Strichartz estimates, which removes the $\varepsilon$-loss of Burq-Poiret-Thomann [Preprint, arXiv: 2304.10979]. To show the polynomial growth of Sobolev norm, our proof relies on the upside-down $I$-method associated to the harmonic oscillator.  Due to the lack of Fourier transform or expansion, we need to carefully control the freqeuncy interaction of the type ``high-high-low-low". To overcome this difficulty, we establish the explicit interaction for  products of eigenfunctions. Our bound covers the result of Planchon-Tzvetkov-Visciglia [Rev. Mat. Iberoam., 39 (2023), 1405-1436] in dimension two and  is new in dimension three. 
		
\end{abstract}

\begin{center}
\begin{minipage}{1000mm}
{ \small {{\bf Key Words:}   Schr\"{o}dinger equation, harmonic oscillator, bilinear estimate, upside-down  I-method,\\ growth of Sobolev norms.}
				{}
			}
\end{minipage}
\end{center}



\section{Introduction }

In this paper, we study the growth of higher-order Sobolev norms for solutions to the  Cauchy problem of nonlinear Schr\"odinger equation(NLS) with harmonic oscillator of the form
\begin{align}\label{fml-NLS}
\begin{cases}
i\partial_tu-Hu=\left|u\right|^2u, & (t,x)\in\Bbb{R}\times\Bbb{R}^d,\\
u(x,0)=u_0(x)\in\mathcal{H}^s(\Bbb{R}^d),
\end{cases}
\end{align}
where the operator $H$ is defined as
\begin{align*}
H=-\Delta+\left|x\right|^2=\sum_{j=1}^{d}\Big(-\frac{\partial^2}{\partial x_j^2}+x_j^2\Big)
\end{align*}
and $\mathcal{H}^s$ the Sobolev space associated with the harmonic oscillator $H$, which reads as
\begin{align*}
\Vert u\Vert_{\mathcal{H}^s(\Bbb{R}^d)}=\Vert H^\frac{s}{2}u\Vert_{L^2(\Bbb{R}^d)}.
\end{align*}
In particular, $\mathcal{H}^1=H^1\cap L^2(\left|x\right|^2dx)$, where $H^1$ denotes the classical Sobolev space. Similarly, we can define the distorted Sobolev space $\mathcal{W}^{s,p}$ as
\begin{align*}
\Vert u\Vert_{\mathcal{W}^{s,p}(\R^d)}=\| H^\frac{s}{2}u\Vert_{L^p(\R^d)}.
\end{align*}
In \cite{Dzi}, Dziubanski and Glowacki proved the following  equivalence of Sobolev norms: for every $s\geq 0$, there exists $C>0$ such that
\begin{equation}\label{equivalence}
\frac 1{C} \Big(\big\|\vt\nabla\vt^s u\big\|_{L^2}^2 +\big \|\langle x\rangle ^s u\big\|_{L^2}^2\Big)  \leq
\|u\|_{{\mathcal H}^s}^2\leq  C\Big (\big\|\vt\nabla\vt^s u\big\|_{L^2}^2 + \big\|\langle x \rangle^s u\big\|_{L^2}^2\Big).
\end{equation}
	
From the above equivalence, the Hermite-Sobolev norm  can be understood as the weighted Sobolev norm. The Schr\"odinger equation with harmonic oscillator models  magnetic traps in the Bose-Einstein condensation. The trapping potential makes all the spectrum  discrete  on Euclidean space. We denote by $\{\tilde h_m\}_{m=0}^{\infty}$ an orthogonal basis of $L^2(\Bbb{R})$ satisfying that
\begin{align*}
H_0\tilde h_m(x)=(-\partial_x^2+x^2) \tilde{h}_m(x)=(2 m+1)\tilde{h}_m(x),\quad x\in\R.
\end{align*}
For general dimensions, the eigen-basis can be generated by tensor product of $\tilde h_m(x_j)$,
\begin{align*}
(-\Delta+|x|^2)h_m(x)=(2|m|+d)h_m(x),\quad h_m(x)=\prod_{j=1}^{d}\tilde{h}_m(x_j),\;x=(x_1,\cdots,x_d)\in\R^d.
\end{align*}
In recent years, growth of Sobolev norms for solution to the dispersive equation gathered  a lot of attention because of the strong connection with weak wave turbulence phenomenon. However, this phenomenon may not occur if the dispersive effect is strong. For the pure Euclidean space,  Dodson \cite{Dodson} proved the scattering theory for cubic NLS in $\R^2$. By using the persistence of regularity, one can verify that the higher order Sobolev norm will not grow up. For $d=3$, Ginibre-Velo \cite{GV} proved the scattering in $d=3$ in energy space and the higher Sobolev norm will  not grow up.  For general nonlinearities $F(u)=|u|^{p-1}u$, if $1+\frac{4}{d}<p\leq 1+\frac{4}{d-2}$, one can show the scattering in  $\dot{H}^1$, see \cite{KV-note} for more information.
	
The growth of Sobolev norm can be expected when considering the Schr\"odinger equation on compact manifold or with some trapping potential due to the lack of dispersive effect.  From the local well-posedness, one can easily obtain the exponential bound on the Sobolev norm. The first polynomial  upper bound of growth of Sobolev norm was proved by Bourgain \cite{Bourgain-JAM}  for the  cubic NLS on $\Bbb T^2$
\begin{align*}
i\partial_tu+\Delta_{\Bbb T^2}u=|u|^2u
\end{align*}
via the Fourier multiplier method:
\begin{align*}
\|u(t)\|_{H^s(\Bbb T^2)}\leq C|t|^{2(s-1)+}.
\end{align*}
For the Euclidean case, Staffilani \cite{Staffilani-Duke} established the polynomial bound $|t|^{(s-1)+}$. For the circle $\Bbb T^1$, Bourgain \cite{Bourgain-JAM} introduced the upside-down $I$-method and the normal form approach to obtain the polynomial bound for   quintic NLS
\begin{align*}
\|u(t)\|_{H^s(\Bbb T)}\lesssim(1+|t|)^{\frac{1}{2}(s-1)+}.
\end{align*}
Later, Colliander-Kwon-Oh \cite{CKO-JAM} generalized the result of \cite{Bourgain-JAM} to general nonlinearity. Recently, Berti-Planchon-Tzvetkov-Visciglia \cite{BPTV-JDE} further improve the result of \cite{CKO-JAM} by constructing the modified energy. For Hartree equation in $d=2$
\begin{align*}
i\partial_tu+\Delta u=(V*|u|^2)u,
\end{align*}
Sohinger \cite{Sohinger-IUMJ} established the following polynomial bound under some assumption on $V$:
\begin{gather*}
\|u(t)\|_{H^s(\R^2)}\lesssim(1+|t|)^{\frac{4}{7}(s-1)+},\quad x\in\R^2,\\
\|u(t)\|_{H^s(\Bbb T^2)}\lesssim(1+|t|)^{(s-1)+},\quad x\in\Bbb T^2.
\end{gather*}
His proof relies on the combination of upside-down $I$-method and resonant decomposition technique initially developed in \cite{CKSTT2004}. For one-dimensional Hartree equation, we refer to \cite{Sohinger-DCDS,Sohinger-IUMJ}.
	
Very recently, Takaoka \cite{Takaoka} used the upside-down $I$-method combined with the resonant decomposition technique to show the improved bound for the cubic NLS on $\R\times\Bbb T$ 
\begin{align*}
\|u(t)\|_{H^s(\R\times\Bbb T)}\leqslant C|t|^{\frac{1}{2}(s-1)+}.
\end{align*}

For the three-dimensional waveguide $\R^m\times\Bbb T^n$ with $m+n=3$, Zhao-Zheng\cite{Zhao-Zheng} showed that the bound
\begin{align*}
\|u(t)\|_{H^2(\R^m\times\Bbb T^n)}\leqslant C(1+|t|).
\end{align*}
For the compact manifold without boundary,  assuming that the following bilinear Strichartz estimate holds for $s\geq s_0$,
\begin{align*}
\big\|e^{it\Delta_g}\Delta_Nfe^{it\Delta_g}\Delta_Mg\big\|_{L_{t,x}^2([0,1]\times M)}\lesssim\min\{N,M\}^{s}\|f\|_{L^2(M)}\|g\|_{L^2(M)},
\end{align*}
Zhong \cite{Zhong} showed the polynomial bound 
\begin{align*}
\|u(t)\|_{H^s(M)}\lesssim(1+|t|)^{A(s-1)+},
\end{align*}
where 
\begin{align}
A^{-1}=\begin{cases}
(1-2s_0)-,&d=2,\\
\big(1-\frac{d-2}{2(d-2s_0)}\big)\big(1-\frac{d-2}{d-s_0-1}\big)-,&d=3.
\end{cases}
\end{align}
In \cite{Hani1}, for general compact manifold without boundary, Hani showed the bilinear Strichartz estimate holds for $s_0=\frac{1}{2}$ in $d=2$ and $s_0=\frac{d-1}{2}$ in $d\geq3$. Therefore, the above bound in $d=2$ becomes $|t|^{K(s-1)+}$ with $K=\infty-$. 
	
Later,  under the assumption on the following $L^4$ Strichartz estimate
\begin{align*}
\big\|e^{it\Delta_g}\Delta_Nf\big\|_{L^4([0,T]\times M)}\lesssim N^{s_0}\|f\|_{L^2(M)},
\end{align*} 
Planchon, Tzvetkov and Visciglia \cite{PTV-APDE} developed the modified energy method and  showed the improved bound  in $d=2$,
\begin{align*}
\|u(t)\|_{H^s(M)}\lesssim(1+|t|)^{\frac{1}{1-2s_0}(s-1)+}.
\end{align*}
In \cite{Burq1,Burq2}, for general compact manifold without boundary, $s_0\leq \frac{1}{4}$, which implies the polynomial bound is at most $|t|^{2(s-1)+}$.
	
For  quintic NLS on $\Bbb T^3$, the global well-posedness was established in Ionescu-Pausader \cite{Ionescu}. However, whether its Sobolev norm grows up is still an open question. When the initial data has small energy, Deng \cite{Deng} established the following bound
\begin{align*}
\|u(t)\|_{H^s(\Bbb T_\lambda^3)}\lesssim\max\big\{\|u_0\|_{H^s(\Bbb T_\lambda^3)},|t|^{300(s-1)}\big\}
\end{align*} 
where $\Bbb T_\lambda^3$ is the irrational tori. It remains open for the large energy case.  For linear Schr\"odinger equation with potential in compact manifold, we refer to \cite{Delort} for more details.
	
For the Schr\"odinger operator with purely trapping potential, i.e. $-\Delta+V(x)$, one can also expect the similar phenomenon holds. The typical model of potential is that $V(x)=|x|^2$.  Using the modified energy method, Planchon-Tzvetkov-Visciglia \cite{PTV-RMI} showed the polynomial growth in $d=2$,
\begin{align}\label{equ:PTVres}
\|u(t)\|_{\mathcal{H}^s(\R^2)}\lesssim C(1+|t|)^{\frac{2}{3}(s-1)+},
\end{align}
For the partial harmonic oscillator, i.e. $V(x)=\sum_{j=1}^dw_jx_j^2$  with $w_j\in\{0,1\}$, the partial dispersive effect will impact the growth of Sobolev norm. For example, Cheng-Guo-Guo-Liao-Shen \cite{CGgLS-APDE} showed the Sobolev norm for the $3d$ cubic NLS with one-dimensional  harmonic oscillator will not grow up. For the lower-dimension case $H=-\partial_{x}^2-\partial_y^2+y^2$ in $\R^2$,  Deng-Su-Zheng \cite{DSZ-CPAA} showed  the same bound \eqref{equ:PTVres}. For other nonlinearity, Killip-Visan-Zhang \cite{KVZ-CPDE} and Jao \cite{Jao-CPDE} proved the global well-posedness for the energy-critical case. It is also interesting to study the growth of Sobolev norm for the energy-critical NLS with harmonic potential. 
	
The lower bound of growth of Sobolev norms is also interesting problem. Colliander-Keel-Staffilani-Takaoka-Tao \cite{CKSTT-Invent} showed that there exists a time $T>0$ such that the solution $\|u(t)\|_{H^s}\geq K$ with initial data $\|u_0\|_{H^s}\leq\delta\ll1$, $s>1$.  Guardia-Kaloshin \cite{Guardia-Kaloshin} showed the polynomial control  on the time  $T $ that the solution growing up at time $T$.
For the waveguide mannifold $\R\times\Bbb T^d$, Hani-Pausader-Tzvetkov-Visciglia \cite{HPTV} constructed the growing-up solution with the rate $\log t$. There exists a global solution $u(t)$ to cubic NLS such that for  $\varepsilon>0$, $s\in\Bbb N$ and $s\geq 30$,   
\begin{align*}
\|u(0)\|_{H^s(\R\times\Bbb T^d)}<\varepsilon,\quad\limsup_{t\to\infty}\|u(t)\|_{H^s(\R\times\Bbb T^d)}=+\infty.
\end{align*} 
Furthermore, there exists a sequence $t_k\to\infty$ as $k\to\infty$, the following holds
\begin{align*}
\|u(t_k)\|_{H^s(\R\times\Bbb T^d)}\gtrsim\exp(c(\log\log t_k)^\frac12).
\end{align*}
 For the case of harmonic oscillator, Chabert \cite{Chabert-CPDE} constructed a smooth solution growing-up in $\mathcal{H}^1$-norm with rate $(\log t)^\frac{1}{2}$ to the following equation
\begin{align*}
i\partial_tu=Hu+|u|^2u+Vu,
\end{align*}
where $V(t,x)$ is smooth real potential such that $V$ decays to zero in $H^s$ norms as $t\to\infty$. We also mention that one can refer to  \cite{Guardia-Haus-Procesi,Haus-Procesi} for extension of growing up solution for high-order nonlinearities.

\subsection{Main Results}
Our main goal of this paper is to show the polynomial bound for the cubic Schr\"odinger equation with harmonic potential in dimensions $d\in\{2,3\}$. 

\begin{theorem}\label{thm1}
Let $s>1$, there exists a constant $C:=C(\|u_0\|_{\mathcal{H}^s(\R^d)})$ such that the following holds: the global solution $u(t)\in C(\R,\mathcal{H}^s)$ to \eqref{fml-NLS} satisfies
\begin{align*}
\|u(t)\|_{\mathcal{H}^s(\R^d)}\leqslant C(1+|t|)^{\tilde{s}_0(s-1)+},\quad \forall\;t\in\R,
\end{align*}
with
\begin{align*}
\tilde s_0=\begin{cases}
		\frac{2}{3},&d=2,\\
			1,&d=3.
			\end{cases}
\end{align*}
\end{theorem}

\begin{remark}
Based on the refined Strichartz estimate on irrational  tori in \cite{DGG-JFA}, Deng-Germain \cite{DGG-IMRN} showed the same bound for the growth rate of Sobolev norms via the upside-down $I$-method. Our result in $d=3$ match the same bound as in \cite{DGG-IMRN}. Moreover, we can show the similar bound holds for solution to the cubic Schr\"odinger equation on $\Bbb S^3$, which will appear in the future work.
\end{remark}

\begin{remark}
Using our method and the standard $I$-method, we can show the global well-posedness for \eqref{fml-NLS} in low regularity space $\mathcal{H}^s(\R^d)$  for some $s_{loc}<s<1$, where $s_{loc}$ is the regularity such that if the regularity of the initial data is below it, the solution will be ill-posed. We leave it for interesting readers.
\end{remark}

\subsection{Outline of the proof}
Our first result is a bilinear estimate for harmonic oscillator in dimension $d\ge2$ without $\varepsilon$ loss. Let $u_N,v_M$ be both frequency localized with $M\le N$. For further application, we consider a more general bilinear estimate
\begin{equation}\label{fml-Intro-bilinear}
\Vt P(\alpha) e^{itH}u_N(0) \cdot P(\beta) e^{itH}v_M(0)\Vt_{\Tsnselfd{2}{2}{[0,T]}} \lesssim \frac{M^{\frac{d-1}{2}}}{N^{\frac{1}{2}}} N^{\mathrm{Ord}(\alpha)}M^{\mathrm{Ord}(\beta)} \Vt u_N\Vt_{\LpnRd{2}} \Vt v_M\Vt_{\LpnRd{2}}.
\end{equation}
Here $P\in \{\nabla,x\}$ and $\alpha,\beta$ are the multi-index, which will be defined in Definition \ref{def-pre-Palp}. Also, $\operatorname{Ord}(\alpha)$ will be defined in  Definition \ref{def-pre-Palp}.  When $\alpha = \beta = 0$, $P(\alpha),P(\beta)$ degenerate to the identical operator.  
	
For $d=2$ and $\alpha=\beta=0$, in \cite{PTV-RMI},  Planchon,  Tzvetkov and  Visciglia proved bilinear estimates for harmonic oscillator in the sense that there is no derivatives loss.  Let $u_N:=e^{itH}u_N(0)$ and $v_M:=e^{itH}v_M(0)$,  they first utilized virial argument to show that
\begin{equation}\label{fml-Intro-bilinear-xy}
\begin{aligned}
			&\int \Lf( \int_{\vt x-y\vt<M^{-1}} M\vt u_N(x)\nabla_y \ovr{v}_M(y) + \ovr{v}_M(y) \nabla_x u_N(x)\vt^2 \Ind{x}\Ind{y} \Rt)\Ind{t}\\
			\lesssim& N \Vt u_N(0)\Vt^2_{L^2(\R^2)} \Vt v_M(0)\Vt^2_{L^2(\R^2)}.
\end{aligned}
\end{equation}
Combining this with a point-wise bound in \cite{Planchon-Jussieu}
\begin{align*}
\vt\phi(x)\vt^2 \lesssim \lambda^{2} \int_{\vt x-y\vt<(4\lambda)^{-1}} \vt\Delta \phi(y)\vt^2 \Ind{y} + \lambda^2\int_{\vt x-y\vt<(4\lambda)^{-1}} \vt\phi(y)\vt^2 \Ind{y}, 
\end{align*}
they proved that 
\begin{equation}\label{fml-Intro-H1}
\Vt e^{itH}u_N(0) \cdot e^{itH}v_M(0)\Vt^2_{L_t^2\dot H^1([0,T]\times\R^2)} \lesssim M^{2} N \Vt u_N\Vt^2_{L^2(\R^2)} \Vt v_M\Vt^2_{L^2(\R^2)}.
\end{equation}
Moreover, by the Strichartz inequality and Sobolev embedding, one has
\begin{equation}\label{fml-Intro-x2}
\int_{0}^T \Lf( \int_{\R^2} \vt x\vt^2 \vt u_N v_M\vt^2(x) \Ind{x} \Rt) \Ind{t} \lesssim M^2 N \Vt u_N\Vt^2_{L^2(\R^2)} \Vt v_M\Vt^2_{L^2(\R^2)}.
\end{equation}
Together with \eqref{fml-Intro-H1} and \eqref{fml-Intro-x2}, they got
\begin{equation}\label{fml-Intro-CH1}
\Vt e^{itH}u_N(0) \cdot e^{itH}v_M(0)\Vt^2_{L_t^2\mathcal H^1([0,T]\times\R^2)} \lesssim M^{2} N \Vt u_N\Vt^2_{L^2(\R^2)} \Vt v_M\Vt^2_{L^2(\R^2)}.
\end{equation}
By using Littlewood-Paley decomposition to $u_Nv_M$, one can get \eqref{fml-Intro-bilinear}  for $d=2$ and $\alpha=\beta=0$.
	
Our first contribution is extending \eqref{fml-Intro-bilinear} to $d\ge 3$, which removes the $\delta$-derivative loss of \cite{BPT}. Recall that in \cite{PTV-RMI}, they use Strichartz inequality and change of variables to get that
\begin{equation*}
		\mbox{(LHS) of }\eqref{fml-Intro-bilinear-xy} \lesssim \int_{\vt z\vt<M^{-1}} M^d \Vt u_N\Vt_{L_t^{p(d)}([0,T],L_x^{p(d)}(\R^d))}
		\Vt \nabla v_M\Vt_{L_t^{q(d)}([0,T],L_x^{q(d)}(\R^d))} \Ind{z}.
\end{equation*}
When $d=2$, we may take $p(2)=q(2)=4$ and notice that $(4,4)$ is an admissible pair, thus Strichartz inequality can be applied directly. However, for $d \ge 3$, $(p(d),p(d))$ and $(q(d),q(d))$ can not be admissible pair at the same time.  For example we take $p(3)=\frac{10}{3}$ and $q(3)=5$, it is clear that $(5,5)$ is not an admissible pair. Since $u_M$ or $v_N$ is spectrally localized, the operator $|\nabla|$ cannot commute with Littlewood-Paley projector $\Delta_N$ associated to $H$, which makes it difficult to use Bernstein's inequality directly. In the following, we are devoted to estimate
\begin{equation*}
\Vt P(\theta) e^{itH} v_M(0)\Vt_{L_t^{q(d)}([0,T],L_x^{q(d)}(\R^d))},
\end{equation*}
with order  $\mathrm{Ord}(\theta)=1$ and $q(d)>\frac{2(d+2)}{d}$. By Sobolev's embedding, we can obtain an admissible pair but with higher regularity:
\begin{align*}
&\Vt P(\theta) e^{itH} v_M(0)\Vt_{L_t^{q(d)}([0,T],L_x^{q(d)}(\R^d))} \\
\lesssim& \Vt P(\theta) e^{itH} v_M(0)\Vt_{L^{q(d)}_t([0,T], \dot{W}_x^{s(d),r(d)}(\R^d)) }\\
\lesssim &\Vt \vt \nabla\vt^{r(d)} P(\theta) e^{itH} v_M(0)\Vt_{L^{q(d)}_t([0,T], L_x^{s(d)}(\R^d)) },
\end{align*}
where $(q(d),r(d))$ is admissible and $r(d)>0$. Thanks to the interpolation, we may assume that $r(d)$ is an positive even integer, and we can write $\vt \nabla \vt^{r(d)}P(\theta) = P(\wt{\theta})$ for some multi-index $\wt{\theta}$. Observing that $P(\wt{\theta})e^{itH}v_M(0)$ satisfies
\begin{equation*}
(i\partial_t+H)(P(\wt{\theta})e^{itH}v_M(0)) = [P(\wt{\theta}),H] v_M.
\end{equation*}
Hence, using the inhomogeneous Strichartz inequality and the fact that $[P(\wt{\theta}),H]$ can be regarded as linear combination of $P(\theta^\prime)$ with $\mathrm{Ord}(\theta^\prime) = \mathrm{Ord}(\theta)$(even $H$ is order $2$), we finally get
\begin{equation*}
\big\Vt [P(\wt{\theta}),H] v_M\big\Vt_{\Tsnselfd{1}{2}{[0,T]}} \lesssim M^{\vt \wt{\theta}\vt} \Vt v_M\Vt_{\LpnRd{2}}.
\end{equation*}
The detailed proof will be postponed in Section \ref{Sec-bilinear}.
	
The above estimate shows that the commutator $[P(\wt{\theta}),H]$ is harmless, this fact can be adapted to extend \eqref{fml-Intro-bilinear} to non-trivial multi-index $\alpha,\beta$. Up to harmless commutators $[P(\alpha),H]$ and $[P(\beta),H]$, we can reduce \eqref{fml-Intro-bilinear} to estimate
\begin{equation*}
\Vt e^{itH} P(\alpha) u_N(0) \cdot e^{itH} P(\beta) v_M(0)\Vt_{\Tsnselfd{2}{2}{[0,T]}}.
\end{equation*} 
Strichartz inequality can be applied directly. See Corollary \ref{cor-bilinear-PDQD} in Section \ref{Sec-bilinear} for the proof of \eqref{fml-Intro-bilinear} for non-trivial multi-index $\alpha,\beta$.
	
After the preparation for the bilinear Strichartz estimate, we return back to the proof of our main theorem on the growth of Sobolev norm. Though for any $s\ge1$, problem \eqref{fml-NLS} is globally well-posed for $u_0 \in \CHsnRd{s}$. However, since the lack of conservation laws in higher-order regularity, we can not obtain uniformly bound for the norm $\Vt u(t,\cdot)\Vt_{\CHsnRd{s}}$. Here, we define a upside-down $I$-operator which can help us reduce the regularity to the energy level.

\begin{definition}
Let $N \in \mathbb{N}$, $u\in\LpnRd{2}$ and $\mu_k$ be the $k-$th eigenvalue associated with $\sqrt H$, we define operator $I_N$ as 
\begin{equation*}
I_N u := \sum_{k\in\Bbb N} m(N^{-1}\mu_k)\pi_{\mu_k} u,
\end{equation*}
with smooth symbol $m(\xi)$ satisfying
\begin{equation*}
	m(\xi) = \begin{cases}
		1,\,\, &\vt \xi\vt \le 1,\\
				\vt\xi\vt^{s-1},\,\, &\vt \xi\vt \ge 2.
		\end{cases}
\end{equation*}
\end{definition}

In the rest of this paper, we abbreviate $I_N$ to $I$. Thus, the $I$ operator can suppress the regularity of the initial data $u_0$ and gives that
\begin{equation*}
N^{1-s}\Vt u_0\Vt_{\CHsnRd{s}} \lesssim \Vt Iu_0\Vt_{\CHsnRd{1}} \lesssim \Vt u_0\Vt_{\CHsnRd{s}}.
\end{equation*}
Acting $I$ operator on both side of \eqref{fml-NLS}, we derive the $I$-system 
\begin{equation}\label{fml-Intro-NLS-I}
\begin{cases}
i\partial_tIu-HIu=I\big(\left|u\right|^2u\big), & (t,x)\in\Bbb{R}\times\Bbb{R}^d,\\
u(x,0)=Iu_0(x)\in\mathcal{H}^1(\Bbb{R}^d).
\end{cases}
\end{equation}
We can define the energy of \eqref{fml-Intro-NLS-I} as 
\begin{equation}\label{fml-Intro-NLS-I-EI}
E(Iu) := \frac{1}{2}\Vt H^{\frac{1}{2}}Iu\Vt^2_{\LpnRd{2}} + \frac{1}{4} \Vt Iu\Vt^4_{\LpnRd{4}}.
\end{equation}
Since $Iu$ is not a real solution to \eqref{fml-NLS}, $E(Iu)$ does not conserve under the evolution of \eqref{fml-Intro-NLS-I}. Hence, to control the increment of energy $\eqref{fml-Intro-NLS-I-EI}$, we are devoted to consider how the commutator $\vt Iu\vt^2 Iu - I\big(\vt u\vt^2 u\big)$ affect \eqref{fml-Intro-NLS-I}.(Note that if we replace the nonlinearity of \eqref{fml-Intro-NLS-I} with $-\vt Iu\vt^2 Iu$, energy \eqref{fml-Intro-NLS-I-EI} conserves). 
	
Thanks to \eqref{fml-Intro-NLS-I}, we have the rate of change of energy with respect with time $t$:
\begin{equation*}
\frac{d}{dt} E(Iu)(t) = \Re\int_{\R^d}\overline{Iu_t}(-I(|u|^2u)+|Iu|^2Iu)dx,
\end{equation*}
and by the fundamental theorem of calculus, we find that
\begin{equation}\label{fml-Intro-Edt}
\begin{aligned}
& E(Iu)(t) - E(Iu)(0) \\
=& \Re i\sum_{\mu_i}\int_{0}^t\int_{\R^d}\big(1-\frac{m(\mu_1)}{m(\mu_2)m(\mu_3)m(\mu_4)}\big)\overline{{\pi}_{\mu_1}(H Iu)}{\pi}_{\mu_2}(Iu)\overline{{\pi}_{\mu_3}(Iu)}{\pi}_{\mu_4}(Iu)dxdt\\
&+\Re i\sum_{\mu_i}\int_{0}^t\int_{\R^d}\big(1-\frac{m(\mu_1)}{m(\mu_2)m(\mu_3)m(\mu_4)}\big)\overline{{\pi}_{\mu_1}(I(|u|^2u))} ){\pi}_{\mu_2}(Iu)\overline{{\pi}_{\mu_3}(Iu)}{\pi}_{\mu_4}(Iu)dxdt\\
=&I_1+I_2,
\end{aligned}
\end{equation}
for all $t\in[0,\delta]$ and $\mu_1,\mu_2,\mu_3,\mu_4$ be the eigenvalue associated with $\sqrt H$.
	
In the sequel, we will estimate the two parts in \eqref{fml-Intro-Edt} separately. For each $u$, we decompose as $u=\sum_{N}\Delta_Nu$. Without loss of generality, we assume $\mu_j \sim N_j$ with  $N_j\in2^{\Bbb N}$. Therefore, we only need to prove the energy increment in a spectrally-localized version. More precisely, we will show that  there exists $\delta \sim 1$, such that for $t\in[0,\delta]$, the following holds
\begin{equation}\label{fml-Intro-ture-EI}
\Lf\vt E(Iu)(t) - E(Iu)(0)\Rt\vt \lesssim (N_1N_2N_3N_4)^{0-} N^{-\frac{5-d}{2}+},
\end{equation}
for each dyadic pieces.
	
It is worth to mention that the most hard case is the frequency-interaction of the type ``high-high-low-low", i.e. $N_1\sim N_2\gtrsim N\gg N_3,N_4$. In this case, the multiplier 
\begin{equation}\label{fml-Intro-multi}
\Lf\vt 1 - \frac{m(\mu_1)}{m(\mu_2)m(\mu_3)m(\mu_4)} \Rt\vt = \Lf\vt 1 - \Big(\frac{\mu_1}{\mu_2}\Big)^{s-1}\Rt\vt,
\end{equation}
ranges from $0$ and $1$ while $\mu_1,\mu_2$ changes.  Indeed, without using \eqref{fml-Intro-L0-Lk}, merely use $1$ as a trivial upper bound for $\vt 1-(\mu_1/\mu_2)^{s-1}\vt$ instead of $\frac{N_3}{N_2}$, the energy increment for $d=2$ reads
\begin{equation*}
\Lf\vt E(Iu)(t) - E(Iu)(0)\Rt\vt \lesssim N^{-1+},
\end{equation*}
as $t \in[0,\delta]$. It is worse than what we expect. 
	
Since \eqref{fml-Intro-multi} is very close to $0$ when $\mu_1\sim\mu_2$, taking $1$ as a trivial upper bound for \eqref{fml-Intro-multi} will waste this good property. Therefore, to make use of the benefit when $\mu_1$ and $\mu_2$ are sufficiently close, we introduce a new identity for product of eigenfunctions  established in Section \ref{Sec-Interfour}. The result is stated as follows. Let 
\begin{equation*}
L_0(e_1,e_2,e_3,e_4) = \int_{\R^d} e_1(x)e_2(x)e_3(x)e_4(x) \Ind{x},
\end{equation*}
then, we have
\begin{equation}\label{fml-Intro-L0-Lk}
L_0(e_1,e_2,e_3,e_4) = \frac{(-2)^k}{(\mu_1^2 - \mu_2^2- \mu_3^2- \mu_4^2)^k} L_k(e_1,e_2,e_3,e_4).
\end{equation}
Here, $L_k(e_1,e_2,e_3,e_4)$ denotes
\begin{equation*}
L_k(e_1,e_2,e_3,e_4)  = \int_{\R^d} e_1(x)\cdot \wt{R}_{2k}(e_2,e_3,e_4)(x) \Ind{x},
\end{equation*}
where $\wt{R}_{2k}(e_2,e_3,e_4)$ is linear combination of the form
\begin{equation}\label{fml-Intro-wtR}
\wt{R}_{2k}(e_2,e_3,e_4) = P(\alpha) e_2(x)P(\beta) e_3(x)P(\gamma) e_4(x),
\end{equation}
with $P\in\{\nabla,x\}$, $\mathrm{Ord}(\alpha) + \mathrm{Ord}(\beta) + \mathrm{Ord}(\gamma) = 2k$ and $0\le \mathrm{Ord}(\alpha),\mathrm{Ord}(\beta),\mathrm{Ord}(\gamma)\le k$. Here, we refer to Definition \ref{def-pre-Palp} for the precise definition of $P$ and $P(\alpha)$.
	
Assuming that $N_1 = N_2 + J N_3$ for some $J\ge0$, to describe multiplier \eqref{fml-Intro-multi} more precisely, we divide $[N_1,2N_1)$ and $[N_2,2N_2)$ into sub-intervals of length $N_3$:
\begin{gather*}
[N_1,2N_1) = \bigcup_{a} I_a := \bigcup_{a} \big[N_1+a N_3,N_1+(a+1) N_3\big),\\ 
[N_2,2N_2) = \bigcup_{b} I_b := \bigcup_{b} \big[N_1+b N_3,N_1+(b+1) N_3\big).
\end{gather*} 
According to the distance between $a$ and $b$, we make further decomposition:
\begin{gather*}
S_1:=\big\{ (I_a,I_b) : a-b > 8,\,\, I_a\subset [N_1,2N_1),\,\, I_b\subset  [N_2,2N_2) \big\},\\
S_2:=\big\{ (I_a,I_b) : a-b < -8,\,\, I_a\subset [N_1,2N_1),\,\, I_b\subset  [N_2,2N_2) \big\},\\
S_3:=\big\{ (I_a,I_b) : \vt a-b\vt \le 8,\,\, I_a\subset [N_1,2N_1),\,\, I_b\subset  [N_2,2N_2) \big\}.
\end{gather*}
Then, we treat $S_1$ only, since there are finite intervals contained in $S_3$ and $S_1$ is symmetry to that of $S_2$. Thanks to \eqref{fml-Intro-L0-Lk} with $k=3$, the first term in \eqref{fml-Intro-Edt}   becomes
\begin{align*}
&\sum_{S_1} \int_0^T\int \Lf[ 1 - \frac{m(\mu_1)}{m(\mu_2)} \Rt] \frac{(-2)^3}{(\mu_1^2-\mu_2^2-\mu_3^2-\mu_4^2)^3} P_{I_a} \ovr{\pi_{\mu_1}(u_1)} \times {\wt{R}_3(P_{I_b}\pi_{\mu_2}(u_2), \pi_{\mu_3}(u_3), \pi_{\mu_4}(u_4)) }\Ind{x} \Ind{t}\\
&=\sum_{S_1} \int_0^T\int \wt{m}(\mu_1,\mu_3,\mu_3,\mu_4) P_{I_a} \ovr{\pi_{\mu_1}(u_1)} \times {\wt{R}_3(P_{I_b}\pi_{\mu_2}(u_2),\pi_{\mu_3}(u_3), \pi_{\mu_4}(u_4))} \Ind{x} \Ind{t}.
\end{align*}
Meanwhile, we obtain a more refined multiplier expression on different intervals:
\begin{equation}\label{fml-Intro-wtm}
\wt{m}(\mu_1,\mu_3,\mu_3,\mu_4)=\Lf[ 1 - \frac{m(\mu_1)}{m(\mu_2)} \Rt] \frac{(-2)^3}{(\mu_1^2-\mu_2^2-\mu_3^2-\mu_4^2)^3}.
\end{equation}
	
Next, applying the Coifman-Meyer multiplier theorem associated to $H$ in Section \ref{Sec-pre}, we find that \eqref{fml-Intro-wtm} gains a coefficient of $\frac{N_3}{N_2} \ll 1$ while eliminating the derivatives loss which brought by $\wt{R}_3$. Together with the bilinear Strichartz estimate, we obtain the desired bound in $d=2$. The similar argument yields the estimate in $d=3$.
	
For term $I_2$, we will use the almost orthogonal property to treat the case where the maximal frequency is much larger than the second one. For the other cases, we will use the $L^p$ Strichartz estimate and the bound of the multiplier to obtain the bounds, which is much better than $I_1$ since it involves higher-order nonlinearity.  
	
Thus, together with simple argument in Section \ref{Sec-EI-Growth}, we can prove Theorem \ref{thm1}.
	
\textbf{Organization of this paper}
This paper is organized as follows. In Section 2, we  recall some basic definitions and properties of the Hermite eigenfunctions and the Littlewood-Paley theory associated to $H$. 	And, we prove some Bernstein-type estimate. In Section 3, we establish the sharp bilinear Strichartz estimates in higher-dimensions. As applications, we prove the local well-posedness for \eqref{fml-NLS}. Then, we introduce the upsider-down $I$-method and the modified $I$-system. In Section 4, we show the interaction of product of eigenfunctions which is crucial in the proof of main theorem. In Section 5, we prove Theorem \ref{thm1} by establishing the energy increment of modified energy.

\subsection{Notations}
In this paper, we use $A\lesssim B$ to mean that there exists a constant such that $A\leqslant CB$, where the constant is not depending on $B$. We will also use $s+$ or $s-$, which means that there exists a small positive number $\varepsilon$ such that it is equal to $s+\varepsilon$ or $s-\varepsilon$ respectively. Denote $\langle x\rangle$ by $(1+\left|x\right|^2)^\frac{1}{2}$.
	
\subsection*{Acknowledgements:}{The authors are grateful to Professor Nicolas Burq, Fabrice Planchon and Nikolay Tzvetkov
for their invaluable comments and suggestions.	  J. Zheng was supported by National key R\&D program of China: 2021YFA1002500 and NSF grant of China (No. 12271051).}



\section{Preliminaies}\label{Sec-pre}
	
In this section, we give some useful harmonic analysis tools associated with the harmonic  oscillator  operator $H$. 

\subsection{Eigen-problems of harmonic oscillator and function space}

Let $\lambda_m$ be the eigenvalue of $\sqrt H$ and $\{h_m\}$ be the family of $L^2$ orthogonal basis with multi-index $m=(m_1,\cdots,m_d)$. First we give the definition of eigenfunctions of Hermite operator when $d=1$. $h_m$ are Hermite functions($d=1$), which is defined as follows. For each $m\in \Bbb N$, the Hermite function of degree $m$ is defined by
\begin{align*}
h_m(t)= (2^m m!\sqrt{\pi})^{-\frac{1}{2}}H_m(t)e^{-\frac{t^2}{2}}, \quad m=0,1,2\cdots,
\end{align*}
where $H_m(t)$ denotes the Hermite polynomials which reads as
\begin{align*}
H_m(t)=(-1)^me^{t^2}\frac{d^m}{dt^m}(e^{-t^2}).
\end{align*}
Given the multi-index $m=(m_1,\cdots,m_d)\in\Bbb{N}^d$, the $d$-dimensional Hermite function is defined as
\begin{align*}
h_m(x)=\prod_{i=1}^{d}h_{m_i}(x_i),m=m_1+\cdots+m_d
\end{align*}
and the eigenvalues associated with $h_m$ is $2\left|m\right|+d$. More precisely, we have
\begin{align*}
Hh_m(x)=\big(2\left|m\right|+d\big)h_m(x)=\Big(2\sum_{i=1}^dm_i+d\Big)\prod_{j=1}^{d}h_{m_j}(x_j).
\end{align*}
For different choices of $m$, the eigenvalue may be same. We can reorder them in a non-decreasing sequence $\{\lambda_n^2\}_{n\geq0}$ (eigenvalue of $H$) repeated according to their multiplicities, and so that
\begin{align*}
Hh_n(x)=\lambda_n^2h_n(x),\quad \forall\, n\in\N.
\end{align*}
The eigenvalue satisfies the asymptotic formula: $\lambda_n\sim n^{\frac{1}{2d}}$ when $n\to\infty.$

We can define the Sobolev norm for every $u\in \mathcal{H}^s(\Bbb{R}^d)$ via the spectral decomposition,
\begin{align*}
\Vert u\Vert^2_{\mathcal{H}^s(\Bbb{R}^d)}=\sum_{n\in\Bbb{N}}\lambda_n^{2s}\left|c_n\right|^2,
\end{align*}
where 
\begin{align*}
u\stackrel{\triangle}{=}\sum_{n\in\Bbb{N}}\pi_nu=\sum_{n\in\Bbb{N}}c_nh_n=\sum_{n\in\Bbb{N}}\langle u,h_n\rangle h_n,
\end{align*} 
where $\langle\cdot,\cdot\rangle$ is the classical inner product in $L^2$ space.
		
We also have the Sobolev embedding in the frame of harmonic oscillator.

\begin{lemma}
Let $s\in[0,1]$, and $1<p\leqslant q\leqslant\infty$. Then we have
\begin{align*}
\Vert f\Vert_{L^q(\Bbb{R}^d)}\lesssim\Vert f\Vert_{\mathcal{W}^{s,p}}.
\end{align*}
\end{lemma}

Let us  define the dyadic localization operators. Let $\eta\in C_0^\infty(\Bbb{R})$ such that $\eta_{[0,1]}=1$ and $\eta_{[2,\infty)}=0$. Set $\psi(x)=\eta(x)-\eta(4x)$. For dyadic number $N\in 2^\Bbb{Z}$, we define the following Littlewood-Paley projector
\begin{align*}
\Delta_N(u)=\begin{cases}
\psi\big(\frac{H}{N^2}\big)u, & \mbox{if }N\geqslant1,\\
0,&\mbox{if }N<1.
\end{cases}
\end{align*}
We notice that if $\lambda_n\notin[\frac{N}{2},\sqrt{2}N]$, then $\Delta_N(h_n)=0$ and $\sum\limits_{N\in2^\Bbb{Z}}\Delta_N(u)=u$. For spectrally-localized function, we have the following Bernstein estimate.

\begin{lemma}[Bernstein's inequality, \cite{BPT}] 
For $d\geq1$ and $2\leq p\leq q\leq\infty$, we have
$$\|\Delta_{N}f\|_{L^q(\R^d)}\leq N^{d(\frac{1}{p}-\frac{1}{q})}\|\Delta_Nf\|_{L^p(\R^d)}.$$
\end{lemma}

For the Hermite-Sobolev space $\mathcal{H}^s$, we have the following equivalence:

\begin{lemma}
For $s\geq0$, we have the following equivalence of the Sobolev norms:
\begin{gather*}
\|f\|_{\mathcal{H}^s(\R^d)}^2\simeq\sum_{N\in2^{\Bbb N}}N^{2s}\big\|\Delta_{N} f\big\|_{L^2(\R^d)}^2.
\end{gather*}
\end{lemma}

Next, we introduce a equivalent characterization of Sobolev norms, that offer convenience in dealing with $\CHsnRd{1}$ norms. We refer to \cite{KVZ-CPDE,Jao-CPDE}.

\begin{lemma}[Equivalent of the Sobolev norms]\label{lem-Equi-H-Del}
Let $\gamma \in [0,1]$ and $1<p<\infty$. Then, for any function $f$
\begin{equation}
\Vt H^{\gamma}f\Vt_{\LpnRd{p}} \simeq \Vt (-\Delta)^\gamma f\Vt_{\LpnRd{p}} + \big\Vt \vt x\vt^{2\gamma} f\big\Vt_{\LpnRd{p}}.
\end{equation}
\end{lemma}

Next, we prove a Bernstein type inequality. We first introduce some definition  
\begin{definition}\label{def-pre-Palp}
Let $\alpha$ be a multi-index of order $k$ which is given by
\begin{equation*}
\alpha = (\alpha_1,\alpha_2,\cdots,\alpha_k),\,\,\text{with}\,\, \alpha_j\in\{1,2\},
\end{equation*}
and we denote the length of $\alpha$ by $\mathrm{Ord}(\alpha)$.
		
For operator $P\in \{x,\nabla\}$, $\mathrm{Ord}(\alpha)=1 ($i.e. $\alpha=\alpha_1)$, and scale-valued function $u$, we define
\begin{equation*}
P(\alpha)u = P(\alpha_1)u=
\begin{cases}
\nabla u, &\alpha_1 = 1,\\
xu, &\alpha_1 = 2.
\end{cases}
\end{equation*}
As for vector valued function $\vec{u}:\R^d\to\R^d$, we also define
\begin{equation*}
P(\alpha)\vec{u} = P(\alpha_1)\vec{u}=
\begin{cases}
{\rm div}(\vec{u}), &\alpha_1 = 1,\\
x\cdot \vec{u}, &\alpha_1 = 2.
\end{cases}
\end{equation*}
		
Then, for $\mathrm{Ord}(\alpha) = k\ge2$ and function $u$, we define
\begin{equation*}
P(\alpha)u = P(\alpha_1,\alpha_2,\cdots,\alpha_k)u = \prod_{j=1}^k P(\alpha_j) u.
\end{equation*}
		
For example, when $\mathrm{Ord}(\alpha)=2$, there are four possible forms of $P(\alpha)$:
\begin{equation*}
P(1,1)u = \Delta u,\,\, P(2,1)u = x\cdot\nabla u,\,\, P(1,2)u = div(xu),\,\, P(2,2)u = \vt x\vt^2 u.
\end{equation*}
\end{definition}

\begin{lemma}[Bernstein type inequality]\label{lem-Bernstein}
Let $N\in \mathbb{N}$ and frequency localized function $u_N \in L_x^2(\R^d)$ satisfies
\begin{equation*}
\Delta_N u_N = u_N.
\end{equation*}
Then, for $k\in\mathbb{N}$ and multi-index $\alpha$ with $\mathrm{Ord}(\alpha) = k$, we have
\begin{equation}\label{fml-pre-lem-res}
\Vt P(\alpha) u_N\Vt_{\LpnRd{2}} \lesssim N^k \Vt u_N\Vt_{\LpnRd{2}}.
\end{equation}
\end{lemma}

\begin{proof}
We prove by induction. For $k = 2$, on one hand we have
\begin{equation*}
\Vt \vt x\vt^2 u_N\Vt_{\LpnRd{2}} + \Vt \Delta u_N\Vt_{\LpnRd{2}} \lesssim \Vt H u_N\Vt_{\LpnRd{2}} \lesssim N^2 \Vt u_N\Vt_{\LpnRd{2}}.
\end{equation*}
On the other hand, from \cite{PTV-APDE} we also get
\begin{equation*}
\Vt x\cdot\nabla u_N\Vt_{\LpnRd{2}} \lesssim N^2 \Vt u_N\Vt_{\LpnRd{2}}.
\end{equation*}
		
For $k=1$, by the result for $k=2$ and interpolation, we have
\begin{equation*}
\Vt xu_N\Vt_{\LpnRd{2}} \lesssim \Vt \vt x\vt^2 u_N\Vt^{\frac{1}{2}}_{\LpnRd{2}} \Vt u_N\Vt^{\frac{1}{2}}_{\LpnRd{2}}\lesssim N \Vt u_N\Vt_{\LpnRd{2}},
\end{equation*}
and
\begin{equation*}
\Vt \nabla u_N\Vt_{\LpnRd{2}} \lesssim \Vt \Delta u_N\Vt^{\frac{1}{2}}_{\LpnRd{2}} \Vt u_N\Vt^{\frac{1}{2}}_{\LpnRd{2}}\lesssim N \Vt u_N\Vt_{\LpnRd{2}}.
\end{equation*}
Thus, \eqref{fml-pre-lem-res} holds for $k=1,2$.
		
Assuming that \eqref{fml-pre-lem-res} holds for all multi-index $\alpha$ of order no more than $k$ with $\alpha = (\alpha_1,\alpha_2,\cdots,\alpha_k)$, then we get
\begin{align*}
\Vt P(\alpha) u_N\Vt_{\LpnRd{2}} &= \Vt P(\alpha_1,\alpha_2\cdots,\alpha_{k-1}) P(\alpha_k) u_N\Vt_{\LpnRd{2}} \\
&\lesssim N^{k-1} \Vt P(\alpha_k) u_N\Vt_{\LpnRd{2}} \\
&\lesssim N^k \Vt u_N\Vt_{\LpnRd{2}},
\end{align*}
and \eqref{fml-pre-lem-res} follows for all $k\in\mathbb{N}$.
\end{proof}

By the definition of Sobolev norm and using the  functional calculus and the semi-classical defect measure \cite[Proposition 5.1]{BPT}, they show the following property
\begin{proposition}\label{Prop-pre-BPT}
	For all $s>0$, there exists two constants $C_1,C_2>0$ such that for all $n \in \N$ and eigenfunctions $h_n(x)$, there holds
	\begin{equation*}
	{C_1\lambda_n^{s} \le \Vt (-\Delta)^{s/2}h_n \Vt_{\LpnRd{2}} \le C_2 \lambda_n^s,}
	\end{equation*}
	and
	\begin{equation*}
		C_1\lambda_n^{s} \le \Vt \Lag x\Rag^{s/2}h_n \Vt_{\LpnRd{2}} \le C_2 \lambda_n^s.
	\end{equation*}
\end{proposition}
\begin{remark}
	It is clear that
	\begin{align*}
		&\qquad\lambda_n^2\Vt P(\alpha)h_n\Vt_{\LpnRd{2}} - \big\Vt [H,P(\alpha)]h_n\big\Vt_{\LpnRd{2}}\\&\le
		\Vt H P(\alpha)h_n\Vt_{\LpnRd{2}}
		\le \lambda_n^2\Vt P(\alpha)h_n\Vt_{\LpnRd{2}} + \big\Vt [H,P(\alpha)]h_n\big\Vt_{\LpnRd{2}},
	\end{align*}
	and note that $[H,P(\alpha)]$ has the same order with $P(\alpha)$. Then, using Lemma \ref{lem-Bernstein} we can get Proposition \ref{Prop-pre-BPT} for $s \in \N$.
\end{remark}

We can also define the Bourgain space associated with harmonic oscillator $H$.

\begin{definition}[Bourgain space]
We define the space $X_{H}^{s,b}(\Bbb{R}\times\Bbb{R}^d)$ as the closure of $C_0^\infty(\Bbb{R}\times\Bbb{R}^d)$ for the norm
\begin{align*}
\Vert u\Vert_{X_{H}^{s,b}(\Bbb{R}\times\Bbb{R}^d)}^2=\sum_{n\in\Bbb{N}}\big\Vert\lambda_n^s\widehat{\pi_nu}(t)\langle \tau+\lambda_n^2\rangle^b\big\Vert_{L_{\tau,x}^2}^2,
\end{align*}
where $\lambda_n^2$ denotes the $n$-th eigenvalue of $H$ and  $\widehat{\pi_nu}(t)$ denotes the Fourier transform w.r.t. time variable and $\pi_nu=\langle u,h_n\rangle h_n$ the spectral projector associated to $H$. 
		
For $T>0$, we can define the restricted Bourgain space equipped with the norm
\begin{align*}
\Vert u\Vert_{X_{H}^{s,b}([-T,T]\times \R^d)}=\inf\{\Vert w\Vert_{X_{H}^{s,b}(\Bbb{R}\times \R^d)},\hspace{1ex}w|_{[-T,T]\times \R^d}=u\}.
\end{align*}
\end{definition}

\begin{lemma}[Some basic properties of $X^{s,b}$]\label{lem:SobXTl4l2}
\begin{enumerate}
\item For $s\in\Bbb{R}$ and $b>\frac{1}{2}$, we have the Sobolev embedding $X^{s,b}(\Bbb{R}\times \R^d)\hookrightarrow C(\Bbb{R},\mathcal{H}^s(\R^d))$.
\item For $s'\leqslant s$ and $b'\leqslant b$, we have the including property
\begin{align*}
X_{T}^{s,b}\subset X_{T}^{s',b'}.
\end{align*}
\item We also need the following embedding property: $X_H^{0,\frac{1}{4}+}\hookrightarrow L_t^4L_x^2$.
\end{enumerate}
\end{lemma}
	
We also have the following embedding theorem. 

\begin{theorem}[Strichartz estimates,\cite{BPT}]
For any $b>\frac12$, there exists a constant $C>0$ such that for any  $s\in\R$ and $(q,r)\in\R^2$ satisfying
\begin{align*}
(q,r)\in[2,\infty)^2,\quad\tfrac{2}{q}=d\big(\tfrac{1}{2}-\tfrac{2}{r}\big),
\end{align*}
then we have
\begin{align}\label{equ:Striclassical}
\|e^{itH}f\|_{L_t^q([0,1], L_x^r(\R^d))}\lesssim&\|f\|_{L_x^2(\R^d)},\\\label{equ:StriBorgain}
\|u\|_{L_t^q\mathcal{W}^{s,r}(\R\times\R^d)}\lesssim& \|u\|_{X^{s,b}(\R\times\R^d)}.
\end{align}
\end{theorem}

Now we state the useful inequalities about the homogeneous and inhomogeneous estimates in Bourgain spaces.

\begin{lemma}[Homogeneous estimate]
Let $s,b>0$, $\psi\in C_0^\infty(\Bbb{R})$ and $f\in X_H^{s,b}(\Bbb{R}\times\Bbb{R}^d)$, then there exists a constant $C>0$ such that 
\begin{align*}
\Vert \psi(t)u\Vert_{X_H^{s,b}(\Bbb{R}\times \Bbb{R}^d) }\leqslant C\Vert u\Vert_{X_H^{s,b}(\Bbb{R}\times\Bbb{R}^d)}.
\end{align*}
\end{lemma}
	
\begin{lemma}[Inhomogeneous estimate]
Let $\frac{1}{4}<b'<\frac{1}{2}$ and $0<b<1-b'$. Then for all $F\in X_T^{s,-b'}([-T,T]\times \R^d)$, then we have
\begin{align*}
\bigg\Vert \int_{0}^{t}\psi(t/T)e^{i(t-s)H}F(s)ds\bigg\Vert_{X_T^{s,b}([-T,T]\times \Bbb{R}^d)}\leqslant CT^{1-b-b'}\Vert F\Vert_{X_T^{s,-b'}([-T,T]\times \Bbb{R}^d)}.
\end{align*}
\end{lemma}
	
\subsection{Almost orthogonal property for the eigenfunctions associated to $H$ and a spectral multiplier theorem}
		
In \cite{Burq2,Burq3},  Burq,  G\'{e}rard and Tzvetkov showed the almost orthogonality of the spectral projections associated to Laplace-Beltrami operator on compact manifold without boundary. Later, Akahori \cite{Akahori} gave the construction of the parametrix associated to the spectral projector associated to $H$ and then use this construction to show  the almost orthogonality of the interaction of eigenfunctions. 
		
\begin{theorem}[Almost orthogonal property for eigenfunctions, \cite{Akahori}]\label{thm-ortho}
Let $d\geq2$, $m\in\N$ and $C_0>0$ large enough, then for any $\ell>0$, there exists $C_\ell>0$ such that for $\lambda_{k_0}\geq C_0\lambda_{k_j}$ with $j\in\{1,\cdots, m\}$, we have
\begin{align}\label{fml-ortho}
\Big|\int_{\R^d}\prod_{j=0}^{m}\pi_{\lambda_{k_j}}u_j(x)\,dx\Big|\leqslant C_\ell\lambda_{k_0}^{-\ell}\prod_{j=0}^{m}\|u_j\|_{L^2(\R^d)}.
\end{align}
			
\end{theorem}

Next, we will show the Coifman-Meyer type multiplier theorem  adapting to the harmonic oscillator $H$. First, we introduce the following $k$-linear operator
\begin{align}\label{fml-spectral-multip}
\Lambda(f_1,\cdots,f_k)=\sum_{n_1,\cdots,n_k}\tilde{m}(n_1,\cdots,n_k)\int_{\R^d}\pi_{n_1}f_1(x)\cdots \pi_{n_k}f_k(x)dx,
\end{align}
where $n_k$ is the eigenvalue associated to $\sqrt H$ and $\pi_{n_k}$ is the  projection operator on the corresponding eigen-space. Before stating the multiplier theorem, we introduce the property of modulation stability.
		
\begin{definition}[Modulation stability property]\label{Def:Modsp}
Let $Y$ be a Banach space. We say that $Y$ satisfies the modulation stability property if the following estimate  
\begin{align*}
\big\|\tilde{f}\big\|_{Y}\lesssim \|f\|_Y, \quad \forall\; f\in Y,
\end{align*}
holds, where $\tilde{f}=\sum_{n_i}e^{i\theta_in_i}\pi_{n_i}f$ is a frequency modulation of $f$.
\end{definition}

Now, we state the boundedness of the $k$-linear spectral multiplier.
		
\begin{theorem}[Spectral multiplier]\label{thm-multi}
Let $\Lambda$ be a $k$-linear spectral multiplier similar to \eqref{fml-spectral-multip} with the multiplier $\tilde{m}$. Suppose that the Banach space $Y$ satisfies the modulation stability property and $\tilde{m}$ obeys the following symbol-type estimate
\begin{align*}
\left|\partial_{\xi_1}^{\alpha_1}\cdots\partial_{\xi_k}^{\alpha_k}\tilde{m}(\xi_1,\cdots,\xi_k)\right|\lesssim \langle\xi_1\rangle^{-\alpha_1}\cdots\langle\xi_k\rangle^{-\alpha_k}.
\end{align*} 
Assume that $f_k$ satisfies the space-time multilinear estimate  
\begin{align*}
\left|\int_{0}^t\int_{\R^d}f_1(t,x)f_2(t,x)\cdots f_k(t,x)dxdt\right|\leqslant A\prod_{j=1}^{k}\|f_j\|_{Y},
\end{align*}
where $f_j$ is spectrally-localized function at scales $N_j$. Then, there exists a constant $C>0$ such that 
\begin{align*}
\left|\int_{0}^t\Lambda(f_1,\cdots,f_k)dt\right|\leqslant CA\prod_{j=1}^{k}\|f_j\|_{Y}.
\end{align*}
\end{theorem} 

\begin{remark}
For compact manifold, Hani \cite{Hani2} first established this spectral multiplier theorem. Our theorem is an analogue of Hani's theorem in the harmonic oscillator setting. For the sake of completeness, we also give the detailed proof here.
\end{remark}

\begin{proof}
Since $f_i$ is spectrally localized around $N_i$, i.e. $n_i\sim N_i$, we can find $\tilde{n}_i\in(0,2]$ such that $n_i=N_i\tilde{n}_i$. Similarly, there exists a smooth function $\varphi\in C_0^\infty([-4,4])$ such that
\begin{align*}
\varphi(\tilde{n}_1,\cdots,\tilde{n}_k)=\tilde m(N_1\tilde{n}_1,\cdots,N_k\tilde{n}_k).
\end{align*}
By the assumption on the boundedness for the multiplier, the derivatives of $\varphi$ is bounded. Using the Fourier expansion, we can write
\begin{align*}
\varphi(\tilde{n}_1,\cdots,\tilde{n}_k)=\sum_{\theta_i\in\Bbb Z/4}M(\theta_1,\cdots,\theta_k)e^{i(\theta_1\tilde{n}_1+\cdots+\theta_k\tilde{n}_k)}.
\end{align*}
Using the definition of $\Lambda$, we obtain
\begin{align*}
\Lambda(f_1,\cdots,f_k)&=\sum_{n_i\sim N_i}\sum_{\theta_i\in\Bbb Z/4}M(\theta_1,\cdots,\theta_k)e^{i(\theta_1n_1N_1^{-1}+\cdots +\theta_kn_kN_k^{-1})}\int_{\R^d}\pi_{n_1}f_1(t)\cdots \pi_{n_k}f_k(t)dx\\
&=\sum_{\theta_i\in\Bbb Z/4}M(\theta_1,\cdots,\theta_k)\int_{\R^d}\prod_{j=1}^{k}\Big(\sum_{n_j\sim N_j}e^{i\theta_jn_jN_j^{-1}}\pi_{n_j}f_j(t)\Big)dx\\
&=\sum_{\theta_j\sim \Bbb Z/4}M(\theta_1,\cdots,\theta_k)\int_{\R^d}\prod_{j=1}^{k}\tilde{f}_j^{\theta_j}(t,x)dx,
\end{align*}
where the modulated function $\tilde{f}_j^{\theta_j}:=\sum_{n_j\sim N_j}e^{i\theta_jn_jN_j^{-1}}\pi_{n_j}f_j(t)$. Using the multilinear estimate for $f_j$ and modulated stability property, we can complete the proof. 
\end{proof}



\section{Bilinear Strichartz estimate for harmonic oscillator}\label{Sec-bilinear}
In this section, we will prove the sharp bilinear Strichartz estimate for harmonic oscillator in the sense that there is no loss of derivatives similarly to that in Euclidean space. Let $H := -\Delta + \vt x\vt^2$, we consider time-space bilinear estimate for the following initial data problem 
\begin{equation}\label{fml-linear-HNLS}
\begin{cases}
(i\partial_t + H)u = 0,\,\, (t,x) \in \R\times\R^d\\
u(0) = u_0 \in C_0^\infty(\R^d).
\end{cases}
\end{equation}
		
In  \cite{BPT}, Burq,  Poiret,  and Thomann established the following bilinear Strichartz estimate with $\delta$-loss of derivatives
\begin{align*}
\big\|e^{itH}\Delta_Mue^{itH}\Delta_Nv\big\|_{L_{t,x}^2([0,T]\times\R^d)}\leq C_T\min\{N,M\}^{\frac{d-2}{2}}\Big(\frac{\min(N,M)}{\max(N,M)}\Big)^{\frac12-\delta}\|\Delta_Mu\|_{L^2(\R^d)}\|\Delta_Nv\|_{L^2(\R^d)}.
\end{align*}
		
Let us briefly review the strategy of \cite{BPT}. First, they use the following physical-space decomposition to $u_M:=\Delta_Mu,$ 
\begin{align*}
u_M=\chi\big(\tfrac{4|x|^2}{M^2}\big)u_M+(1-\chi)\big(\tfrac{4|x|^2}{M^2}\big)u_M,
\end{align*}
with $\chi\in C_c^\infty(\R),\;0\leq\chi\leq1,$ satisfying
\begin{equation*}
\chi(x)= \begin{cases}
1\quad \text{if}\quad |x|\leq\frac{15}{32},\\
0\quad \text{if}\quad |x|\geq\frac12.
\end{cases}
\end{equation*}
Substituting into the bilinear form of Schr\"odinger flow, it suffices to show the following two estimate
\begin{align}\label{burq-1}
\big\|e^{itH}v_Ne^{itH}\chi\big(\tfrac{|x|^2}{M^2}\big)u_M\big\|_{L_{t,x}^2}
\end{align}
and 
\begin{align}\label{burq-2}
\big\|e^{itH}v_Ne^{itH}(1-\chi)\big(\tfrac{|x|^2}{M^2}\big)u_M\big\|_{L_{t,x}^2}.
\end{align}
Notice that in \eqref{burq-2}, they  use the almost orthogonality of the support of two functions and obtain the fast decay of \eqref{burq-2}. For \eqref{burq-1}, they want to use the lens transform to transfer the estimate to $e^{it\Delta}$. This requires the use of bilinear Strichartz estimate for functions $f,g$ which is not spectrally localized. To sum over all frequencies, the $\delta$-loss appears.  
		
Our main result in this section is to remove the $\delta$-loss by the argument of \cite{PTV-RMI} and some technique computations.

\begin{theorem}\label{thm-bilinear}
Let $1 \le M \le N$ be dyadic numbers and $u_N(0),v_M(0) \in \LpnRd{2}$ be two frequency localized functions with 
\begin{equation*}
\Delta_N u_N(0) = u_N(0)\,\, ,\Delta_M v_M(0) = v_M(0).
\end{equation*}
Then, for $T \in (0,\infty)$, $d\ge 2$, there exists constant $C = C(T)$ depend on $T$ only such that
\begin{equation}\label{fml-bilinear-res}
\Vt e^{itH}u_N(0) \cdot e^{itH}v_M(0)\Vt_{\Tsnselfd{2}{2}{[0,T]}} \le C \frac{M^{\frac{d-1}2}}{N^\frac12} \Vt u_N(0)\Vt_{\LpnRd{2}} \Vt v_M(0)\Vt_{\LpnRd{2}}.
\end{equation}
\end{theorem}
		
\begin{remark}
For $d=2$, the proof for this theorem can be found in \cite{PTV-RMI}. For higher dimension, the proof is more complex and need more techniques. However, presenting such generality would muddy the exposition, we only focus on the proof for $d=3$, and other dimension can be obtained similarly.
\end{remark}

\begin{proof} 
			
By paraproduct, we will show below that  the estimate \eqref{fml-bilinear-res}  with $d=3$ is equivalent to 
\begin{equation}\label{fml-thm-L2H1-Har}
\Vt e^{itH}u_N(0) \cdot e^{itH}v_M(0)\Vt_{\TsnCHself{2}{1}{[0,T]}} \le C MN^\frac12 \Vt u_N(0)\Vt_{\LpnR{2}} \Vt v_M(0)\Vt_{\LpnR{2}}.
\end{equation}
Writing
\begin{equation*}
u_N := e^{itH}u_N(0),\; v_M := e^{itH}v_M(0),
\end{equation*}
and 	by the equivalence of Sobolev norm \eqref{equivalence}, we could obtain \eqref{fml-thm-L2H1-Har} provided that we show
\begin{equation}\label{fml-thm-L2x2}
\big\| |x|u_N v_M\big\|_{L_{t,x}^2([0,T]\times\R^3)}
\leq C_T MN^\frac12 \Vt u_N(0)\Vt_{\LpnR{2}} \Vt v_M(0)\Vt_{\LpnR{2}},
\end{equation}
and
\begin{equation}\label{fml-thm-L2H1}
\big\| \nabla(u_N v_M)\big\|_{L_{t,x}^2([0,T]\times\R^3)} \leq C_T MN^\frac12 \Vt u_N(0)\Vt_{\LpnR{2}} \Vt v_M(0)\Vt_{\LpnR{2}}.
\end{equation}
We assert that the proof of \eqref{fml-thm-L2H1} can be reduced to that of
\begin{equation}\label{fml-thm-asser}
\begin{aligned}
&\int_0^T \Lf( \int_{\vt x-y\vt<M^{-1}} M\big\vt u_N(x)\nabla_y \ovr{v}_M(y) + \ovr{v}_M(y)\nabla_x u_N(x) \big\vt^2 \Ind{x}\Ind{y} \Rt) \Ind{t}\\
\le& C N \Vt u_N(0)\Vt^2_{\LpnR{2}} \Vt v_M(0)\Vt^2_{\LpnR{2}}.
\end{aligned}
\end{equation} 
			
Therefore, the proof of Theorem \ref{thm-bilinear} mainly contains the following four steps:
\begin{enumerate}
\item[$\circ$] Step 1: Proof of the assertion \eqref{fml-thm-asser},
\item[$\circ$] Step 2: Showing that \eqref{fml-thm-asser} implies \eqref{fml-thm-L2H1},
\item[$\circ$] Step 3: Proof of the weighted bilinear estimate \eqref{fml-thm-L2x2},
\item[$\circ$] Step 4: Showing that \eqref{fml-thm-L2H1-Har} implies \eqref{fml-bilinear-res}.
\end{enumerate}

\textbf{Step 1: Proof of the assertion \eqref{fml-thm-asser}:}
We consider a $C^1$ function $\rho:\R^3 \to \R$ such that $\partial^2_{kl}\rho$ are piece-wise continuous functions with $1\le k,l\le 3$. Then, we define a bilinear form $B_\rho$ associated to the Hessian of $\rho$ as follows
\begin{equation*}
B_{\rho(z)}(f(z),g(z)) := \sum_{k,l=1}^3 f_k(z)\partial^2_{kl}\rho(z)g_l(z),
\end{equation*}
where $z \in \R^3$ and $f,g : \R^3\to\Bbb C^3$. We claim that  there holds
\begin{equation}\label{fml-thm-asser-asser}
\begin{aligned}
&\int_0^T \Lf( \int_{\vt x-y\vt<M^{-1}} B_{\wt{\rho}(x-y)}\left( \ovr{v}_M(y)\nabla u_N(x) + u_N(x) \nabla \ovr{v}_M(y) , v_M(y) \nabla \ovr{u}_N(x) + \ovr{u}_N(x)\nabla v_M(y) \right) \Ind{x}\Ind{y} \Rt) \Ind{t}\\
\le& C {\Vt  \nabla\wt{\rho}\Vt_{\LpnR{\infty}} }\\
&\times  \big( \Vt v_M(0)\Vt^2_{\LpnR{2}} \Vt u_N(0)\Vt_{\LpnR{2}} \Vt u_N(0)\Vt_{\CHsnR{1}} + \Vt u_N(0)\Vt^2_{\LpnR{2}} \Vt v_M(0)\Vt_{\LpnR{2}} \Vt v_M(0)\Vt_{\CHsnR{1}} \big).
\end{aligned}
\end{equation}
Here $\wt{\rho}(z) := \rho_M(z_1)$ is determined by a convex function $\rho_M(z_1) : \R \to \R$, where
\begin{equation*}
\rho_M(z_1) = 
\begin{cases}
\frac{Mz_1^2}{2} + \frac{1}{2M},\, &\vt z_1\vt \le \frac{1}{M},\\
\operatorname{sgn}(z_1)z_1 ,\, &\vt z_1\vt > \frac{1}{M}.
\end{cases}
\end{equation*}
			
Denote by $a_{jk}=\big(\frac{\partial^2}{\partial z_j\partial z_k}\widetilde{\rho}\big)_{3\times 3}$.	We can check that when $\vt z_1\vt > \frac{1}{M}$, the Hessian of $\wt{\rho}(z_1)$ vanishes identically while when $\vt z_1\vt \le \frac{1}{M}$, only $a_{11}$ is non-zero and equals to $M$. Therefore, apply $\wt{\rho}(z)$ here and \eqref{fml-thm-asser-asser} becomes
\begin{align*}
&\int_0^T \Lf( \int_{\vt x_1-y_1\vt<M^{-1}} M\big\vt \ovr{v}_M(y)\partial_{x_1} u_N(x) + u_N(x) \partial_{x_1}\ovr{v}_M(y) \big\vt^2 \Ind{x}\Ind{y} \Rt) \Ind{t}\\
\le& C \big( \Vt v_M(0)\Vt^2_{\LpnR{2}} \Vt u_N(0)\Vt_{\LpnR{2}} \Vt u_N(0)\Vt_{\CHsnR{1}} + \Vt u_N(0)\Vt^2_{\LpnR{2}} \Vt v_M(0)\Vt_{\LpnR{2}} \Vt v_M(0)\Vt_{\CHsnR{1}} \big).
\end{align*} 
Since the region $\vt x-y\vt<M^{-1}$ is contained in $\vt x_1-y_1\vt<M^{-1}$, we can change the integration region in the above equality to $\vt x-y\vt<M^{-1}$. Next, we set $\wt{\rho}(z) := \rho_M(z_j)$ with $j=2,3$ and $(x_1,y_1)$ in the previous inequality can be replaced by $(x_j,y_j)$ with $j=2,3$. Finally, notice that $u_N$ and $v_M$ are spectral localized and we have proved that $\eqref{fml-thm-asser-asser}$ implies \eqref{fml-thm-asser}. 
			
Now, we turn to prove \eqref{fml-thm-asser-asser}. We consider the interaction energy as follows
\begin{equation}
I_\rho(t) := \int \vt u(x)\vt^2 \rho(x-y) \vt v(y)\vt^2 \Ind{x}\Ind{y},
\end{equation} 
where we have abbreviated $u_N(x),v_M(y)$ to $u(x),v(y)$ for short. Then, we are devoted to the second derivative of $I_\rho(t)$. Let $K(x,y) : \R^3\times\R^3 \to \R$ be a function which is at least $C^1$ smooth and $w(t,x)$ solves \eqref{fml-linear-HNLS}, by integrating by parts, we get
\begin{equation}\label{fml-thm-1-deri}
\frac{\Ind{}}{\Ind{t}} \int K(x,y) \vt w(x)\vt^2 \Ind{x} = 2 \int \nabla_x K(x,y) \cdot \Im (\nabla \ovr{w}(x)\cdot w(x)) \Ind{x},
\end{equation}
and
\begin{align}\label{fml-thm-2-deri}
&\frac{\Ind{}^2}{\Ind{t}^2} \int K(x,y) \vt w(x)\vt \Ind{x} \\\nonumber
=& 4 \int B_{K(x,y)}(\nabla w(x),\nabla\wt{w}(x)) \Ind{x} - \int \Delta_x K(x,y) \cdot \Delta_x (\vt w(x)\vt^2) \Ind{x} - 4 \int x \cdot \nabla_x K(x,y) \vt w(x)\vt^2 \Ind{x}.
\end{align}
Notice that
\begin{equation*}
\frac{\Ind{}}{\Ind{t}} \Big(\vt u(x)\vt^2 \rho(x-y) \vt v(y)\vt^2\Big) =\big (\rho(x-y) \vt v(y)\vt^2\big) \frac{\Ind{}}{\Ind{t}} \vt u(x)\vt^2 + \big(\vt u(x)\vt^2 \rho(x-y)\big) \frac{\Ind{}}{\Ind{t}} \vt v(y)\vt^2,
\end{equation*}
we can apply \eqref{fml-thm-1-deri} with $K(x,y) = \rho(x-y) \vt v(y)\vt^2$ and $K(x,y) = \vt u(x)\vt^2 \rho(x-y)$ respectively (we have exchanged the position of variable $x$ and $y$ in the second term above), we get the first derivative of $I_{\rho}(t)$:
\begin{equation}\label{fml-thm-1-deri-Irho}
\begin{aligned}
&\frac{\Ind{}}{\Ind{t}} I_{\rho}(t)=2\int \nabla \rho(x-y) \cdot \Im(\nabla\ovr{u}(x)u(x)) \vt v(y)\vt^2 \Ind{x}\Ind{y} - 2\int \nabla \rho(x-y) \cdot \Im(\nabla\ovr{v}(y)v(y)) \vt u(x)\vt^2 \Ind{x}\Ind{y}.
\end{aligned}
\end{equation}
Thanks to conservation laws and $\Vt \nabla w\Vt_{\LpnR{2}}\lesssim\Vt w\Vt_{\CHsnR{1}}$, \eqref{fml-thm-1-deri-Irho} implies that
\begin{equation}\label{fml-thm-1-deri-Irho-bound}
\begin{aligned}
&\Lf\vt \frac{\Ind{}}{\Ind{t}} I_{\rho}(t)\Rt\vt  \lesssim \Vt \nabla\rho\Vt_{\LpnR{\infty}} (\Vt v(t)\Vt_{\LpnR{2}}^2\Vt u(t)\Vt_{\LpnR{2}}\Vt u(t)\Vt_{\CHsnR{1}} + \Vt u(t)\Vt_{\LpnR{2}}^2\Vt v(t)\Vt_{\LpnR{2}}\Vt v(t)\Vt_{\CHsnR{1}})\\
&\lesssim \Vt \nabla\rho\Vt_{\LpnR{\infty}} (\Vt v(0)\Vt_{\LpnR{2}}^2\Vt u(0)\Vt_{\LpnR{2}}\Vt u(0)\Vt_{\CHsnR{1}} + \Vt u(0)\Vt_{\LpnR{2}}^2\Vt v(0)\Vt_{\LpnR{2}}\Vt v(0)\Vt_{\CHsnR{1}}).
\end{aligned}
\end{equation}
			
Now, we treat the second derivative of $I_\rho(t)$. Direct computation shows
\begin{align*}
&\frac{\Ind{}^2}{\Ind{t}^2} \big(\vt u(x)\vt^2 \rho(x-y) \vt v(y)\vt^2\big) \\ 
=&\big(\rho(x-y) \vt v(y)\vt^2\big) \frac{\Ind{}^2}{\Ind{t}^2} \vt u(x)\vt^2 + \big(\vt u(x)\vt^2 \rho(x-y)\big) \frac{\Ind{}^2}{\Ind{t}^2} \vt v(y)\vt^2 + 2 \rho(x-y) \frac{\Ind{}}{\Ind{t}} \vt u(x)\vt^2 \frac{\Ind{}}{\Ind{t}} \vt v(y)\vt^2.
\end{align*}
Since the first and second term can be calculated by \eqref{fml-thm-2-deri}, we can treat the last term via integration by parts and \eqref{fml-thm-1-deri}:
\begin{equation}\label{fml-thm-2-deri-mid}
\begin{aligned}
\int \Lf( \int \rho(x-y)\frac{\Ind{}}{\Ind{t}}\vt u(x)\vt^2 \Ind{x} \Rt)\frac{\Ind{}}{\Ind{t}}\vt v(y)\vt^2 \Ind{y} &= 2 \Lf( \int \nabla\rho(x-y)\cdot\Im(\nabla \ovr{u}(x)u(x)) \Ind{x} \Rt)\frac{\Ind{}}{\Ind{t}}\vt v(y)\vt^2 \Ind{y}\\
&=-4 \int B_{\rho}(\Im(\nabla \ovr{u}(x)u(x)),\Im(\nabla \ovr{v}(y)v(y))) \Ind{x}\Ind{y}.
\end{aligned}
\end{equation}
			
Combining \eqref{fml-thm-2-deri} and \eqref{fml-thm-2-deri-mid}, we can write the second derivative of $I_\rho$ as
\begin{equation}\label{fml-thm-2-deri-Irho}
\begin{aligned}
\frac{\Ind{}^2}{\Ind{t}^2} I_{\rho}(t) =& 4 \int \Lf( B_{\rho(x-y)}(\nabla u(x),\nabla \ovr{u}(x))\vt v(y)\vt^2 + B_{\rho(x-y)}(\nabla v(y),\nabla \ovr{v}(y))\vt u(x)\vt^2 \Rt) \Ind{x}\Ind{y}\\
&- \int \Delta\rho(x-y) \Big( \Delta(\vt u(x)\vt^2)\vt v(y)\vt^2 + \Delta(\vt v(y)\vt^2)\vt u(x)\vt^2 \Big) \Ind{x}\Ind{y}\\
&- 8\int B_{\rho(x-y)}(\Im(\nabla\ovr{u}(x)u(x)),\Im(\nabla\ovr{v}(y)v(y))) \Ind{x}\Ind{y}\\
& -4\Re \int (x-y)\cdot\nabla\rho(x-y) \vt u(x)\vt^2 \vt v(y)\vt^2 \Ind{x}\Ind{y}\\
&:= I_1 + I_2 + I_3 + I_4.
\end{aligned}
\end{equation}
			
Since $\Delta\rho(x-y) = -\nabla_x\cdot\nabla_y\rho(x-y) = -\nabla_y\cdot\nabla_x\rho(x-y)$, using integration by parts with variable $y$ and $x$, we obtain
\begin{align*}
\int -\Delta\rho(x-y) \Delta_x(\vt u(x)\vt^2)\vt v(y)\vt^2 \Ind{x}\Ind{y} &= \int \nabla_y\cdot\nabla_x\rho(x-y) \Delta_x(\vt u(x)\vt^2)\vt v(y)\vt^2 \Ind{x}\Ind{y}\\
&=-\int \nabla_y(\vt v(y)\vt^2)\cdot\nabla_x\rho(x-y)\Delta_x(\vt u(x)\vt^2) \Ind{x}\Ind{y}\\
&= \int -\nabla_x\nabla_y\rho(x-y) \nabla_y(\vt v(y)\vt^2)\cdot\nabla_x(\vt u(x)\vt^2) \Ind{x}\Ind{y}\\
&=4\int B_{\rho(x-y)}(\Re(\ovr{u}(x)\nabla u(x)),\Re(\ovr{v}(y)\nabla v(y)))\Ind{x}\Ind{y}.
\end{align*}
Similarly, we can deal with the second term in $I_2$ and arrive at
\begin{equation}\label{fml-thm-I2-eq}
I_2 = 8 \int B_{\rho(x-y)}\big(\Re(\ovr{u}(x)\nabla u(x)),\Re(\ovr{v}(y)\nabla v(y))\big)\Ind{x}\Ind{y}.
\end{equation}
For one direction derivatives, we easily get
\begin{align*}
&4|v|^2(y) |\partial u|^2(x)+ 4|u|^2(x) |\partial v|^2(y)+ 8\frac{(v\partial\bar v+\bar v \partial v)(y)}2 \frac{(u\partial\bar{u}+\bar u \partial u)(x)} 2\\
&- 8\frac{(v\partial\bar v-\bar v \partial v)(y)}2\frac{ (\bar{u}\partial u-  u \partial\bar u)(x)}2\\
=& 4|v|^2(y) |\partial u|^2(x)+ 4|u|^2(x) |\partial v|^2(y)+ 4 v\partial\bar v(y) u\partial\bar u(x) +4 \bar v \partial{v}(y)\bar u\partial u(x) \\
=&  4 |\bar v (y)\partial u(x)+u(x)\partial \bar v(y)|^2.
\end{align*}
Therefore, we can absorb $I_1 + I_2 + I_3$ into a bilinear form and rewrite that
\begin{align*}
\frac{\Ind{}^2}{\Ind{t}^2} I_\rho(t) = 4\int B_{\rho}(x-y)\Big( \ovr{v}(y)\nabla u(x)+u(x)\nabla\ovr{v}(y) , v(y)\nabla \ovr{u}(x)+\ovr{u}(x)\nabla v(y)\Big) \Ind{x}\Ind{y} + I_4.
\end{align*}
Integrating in time \and applying \eqref{fml-thm-1-deri-Irho-bound}, we obtain
\begin{align*}
&4\Big|\int_0^T \int B_{\rho}(x-y)\Big( \ovr{v}(y)\nabla u(x)+u(x)\nabla\ovr{v}(y) , v(y)\nabla \ovr{u}(x)+\ovr{u}(x)\nabla v(y)\Big) \Ind{x}\Ind{y} \Ind{t}\Big|\\ 
\lesssim& \Lf\vt\frac{\Ind{}}{\Ind{t}} I_\rho(0)\Rt\vt + \Lf\vt\frac{\Ind{}}{\Ind{t}} I_\rho(T)\Rt\vt + \Lf\vt \int_0^TI_4(t)\Ind{t}\Rt\vt\\
\lesssim& \Vt \nabla\rho\Vt_{\LpnR{\infty}} \\
&\times(\Vt v(0)\Vt_{\LpnR{2}}^2\Vt u(0)\Vt_{\LpnR{2}}\Vt u(0)\Vt_{\CHsnR{1}} + \Vt u(0)\Vt_{\LpnR{2}}^2\Vt v(0)\Vt_{\LpnR{2}}\Vt v(0)\Vt_{\CHsnR{1}})\\
&+ \Lf\vt \int_0^TI_4(t)\Ind{t}\Rt\vt.
\end{align*}
Noting that $\Vt yw\Vt_{\LpnR{2}}\lesssim \Vt w\Vt_{\CHsnR{1}}$, hence $\int_0^TI_4(t)\Ind{t}$ enjoys the bound
\begin{align*}
&\Lf\vt \int_0^TI_4(t)\Ind{t}\Rt\vt \\
\lesssim& \int_0^T \int \vt x-y\vt \vt\nabla\rho(x-y)\vt \vt u(x)\vt^2 \vt v(y)\vt^2 \Ind{x}\Ind{y} \Ind{t}\\
\lesssim& \Vt \nabla\rho\Vt_{\LpnR{\infty}} \\
&\times\sup_{t \in (0,T)} \Big( \Vt u(t)\Vt^2_{\LpnR{2}}\Vt v(t)\Vt_{\LpnR{2}}\Vt yv(t)\Vt_{\LpnR{2}} + \Vt v(t)\Vt^2_{\LpnR{2}}\Vt u(t)\Vt_{\LpnR{2}}\Vt xu(t)\Vt_{\LpnR{2}} \Big)\\
\lesssim& \Vt \nabla\rho\Vt_{\LpnR{\infty}}\big(\Vt v(0)\Vt_{\LpnR{2}}^2\Vt u(0)\Vt_{\LpnR{2}}\Vt u(0)\Vt_{\CHsnR{1}} + \Vt u(0)\Vt_{\LpnR{2}}^2\Vt v(0)\Vt_{\LpnR{2}}\Vt v(0)\Vt_{\CHsnR{1}}\big).
\end{align*}
			
Therefore, we have proved assertion \eqref{fml-thm-asser-asser}.
\end{proof}
		
Following the similar argument to the proof of \eqref{fml-thm-asser}, we can obtain the following estimate	
\begin{equation}\label{fml-thm-asser-Hvm}
\begin{aligned}
&\int_0^T \Lf( \int_{\vt x-y\vt<M^{-1}} M\big\vt u_N(x)\nabla_y (H\ovr{v}_M)(y) + (H\ovr{v}_M)(y)\nabla_x u_N(x) \big\vt^2 \Ind{x}\Ind{y} \Rt) \Ind{t}\\
\le& C NM^2 \Vt u_N(0)\Vt^2_{\LpnR{2}} \Vt v_M(0)\Vt^2_{\LpnR{2}},
\end{aligned}
\end{equation}
which will be used in the proof of Step 2.

Before the proof of Step 2, we first introduce a useful lemma.
		
\begin{lemma}\label{lem-point-wise}
There exists $\lambda_0>0$ such that for any smooth function $\phi : \R^3 \to \mathbb{C}$ and $\lambda\ge \lambda_0$, the following holds
\begin{equation}\label{fml-lem-res}
\vt \phi(x)\vt^2 \lesssim \lambda^{-1} \int_{\vt x-y\vt<\lambda^{-1}} \vt H\phi(y)\vt^2 \Ind{y} + \lambda^{3} \int_{\vt x-y\vt<\lambda^{-1}} \vt \phi(y)\vt^2 \Ind{y},\,\, \text{for all}\,\, x \in \R^3.
\end{equation}
\end{lemma}

\begin{proof}
Without of generality, we assume that $\phi$ is a real-valued function.  From \cite{Planchon-Jussieu}, we have 
\begin{equation*}
\vt \phi(x)\vt^2 \lesssim \lambda^{-1} \int_{\vt x-y\vt<(4\lambda)^{-1}} \vt -\Delta\phi(y)\vt^2 \Ind{y} + \lambda^3 \int_{\vt x-y\vt<(4\lambda)^{-1}} \vt \phi(y)\vt^2 \Ind{y}.
\end{equation*}
Hence, it is sufficient to prove that
\begin{equation}\label{fml-lem-point-deduce}
\begin{aligned}
&\lambda^{-1} \int_{\vt x-y\vt<(4\lambda)^{-1}} \vt -\Delta\phi(y)\vt^2 \Ind{y} \\
\lesssim& \lambda^{-1} \int_{\vt x-y\vt<(4\lambda)^{-1}} \vt H\phi(y)\vt^2 \Ind{y} + \lambda^3 \int_{\vt x-y\vt<(4\lambda)^{-1}} \vt \phi(y)\vt^2 \Ind{y}.
\end{aligned} 
\end{equation}
Since $-\Delta = H - \vt y\vt^2$, by integration by parts, for any $f \in C^\infty_0(\R^3)$ we get an equality
\begin{equation}\label{fml-lem-point-equality}
\int (\vt \Delta f\vt^2 + \vt y\vt^4 \vt f\vt^2 + 2\vt y\vt^2\vt \nabla f\vt^2) \Ind{y} = \int ( \vt Hf\vt^2 + 4\vt f\vt^2 ) \Ind{y}.
\end{equation}
			
Let $\chi_\lambda = \chi(\lambda(x-y))$ where $\chi\in C^\infty_0(\R^3)$ and
\begin{equation*}
\chi(\xi) = 
\begin{cases}
1, &\vt \xi\vt<\frac{1}{4},\\
0,&\vt \xi\vt > \frac{1}{2}.
\end{cases}
\end{equation*} 
Then, substituting $f(y) = \chi_\lambda(y)\phi(y)$ into \eqref{fml-lem-point-equality}, we arrive at
\begin{equation}\label{fml-lem-point-equality-chi}
\begin{aligned}
&\int (\vt \chi_\lambda\vt^2 \vt \Delta \phi\vt^2 + \vt y\vt^4 \vt \chi_\lambda\vt^2 \vt \phi\vt^2 + 2\vt y\vt^2\vt \nabla \chi_\lambda \phi\vt^2) \Ind{y}\\
=& \int ( \vt \chi_\lambda\vt^2\vt H\phi\vt^2 + 4\vt \chi_\lambda\vt^2 \vt \phi\vt^2 ) \Ind{y} - 2 \int \chi_\lambda \vt y\vt^2\phi (2\nabla\chi_\lambda\cdot \nabla\phi + \Delta\chi_\lambda\cdot \phi) \Ind{y}\\
\lesssim& \int ( \vt \chi_\lambda\vt^2\vt H\phi\vt^2 + 4\vt \chi_\lambda\vt^2 \vt \phi\vt^2 ) \Ind{y} + 2 \Lf\vt \int \chi_\lambda \vt y\vt^2\phi (2\nabla\chi_\lambda\cdot \nabla\phi + \Delta\chi_\lambda\cdot \phi) \Ind{y}\Rt\vt.
\end{aligned}
\end{equation}
Using Young's inequality, it gives that fo any $\mu > 0$
\begin{align*}
&\chi_\lambda \vt y\vt^2\phi (2\nabla\chi_\lambda\cdot \nabla\phi + \Delta\chi_\lambda\cdot \phi)
\lesssim \mu \vt \chi_\lambda\vt^2 \vt y\vt^4 \vt \phi\vt^2 + \mu^{-1}(2\nabla\chi_\lambda\cdot \nabla\phi + \Delta\chi_\lambda\cdot \phi)^2.
\end{align*}
We can choose $\mu$ properly, so that the term $\vt y\vt^4 \vt \chi_\lambda\vt^2 \vt \phi\vt^2$ in the l.h.s of \eqref{fml-lem-point-equality-chi} can be absorbed. Thus, we can control the term $\vt \chi_\lambda\vt^2 \vt \Delta \phi\vt^2$ in \eqref{fml-lem-point-equality-chi} as follows
\begin{equation}
\int \vt \chi_\lambda\vt^2 \vt \Delta \phi\vt^2 \Ind{y} \lesssim \int \vt \chi_\lambda\vt^2 \vt H\phi\vt^2 + 4 \vt \chi_\lambda\vt^2 \vt \chi\vt^2 + \vt \chi_\lambda\cdot \nabla\phi\vt^2 + \vt \Delta \chi_\lambda\vt^2 \vt \phi\vt^2 \Ind{y}.
\end{equation}
Basic calculation shows $\big\vt \vt \nabla\vt^k \chi_\lambda \big\vt \lesssim \lambda^k$ for $k =1,2$, it implies that
\begin{equation}
\int_{\vt y\vt<(4\lambda)^{-1}} \vt \Delta \phi\vt^2 \Ind{y} \lesssim \int_{\vt y\vt<(2\lambda)^{-1}} \vt H\phi\vt^2 \Ind{y} + \lambda^2 \int \wt{\chi}_\lambda \vt \nabla\phi\vt^2 \Ind{y} + (1+\lambda^4) \int_{\vt y\vt<(2\lambda)^{-1}} \vt \phi\vt \Ind{y},
\end{equation}
where $\wt{\chi}_\lambda(y) = \wt{\chi}(\lambda(y-x))$ with $\wt{\chi}$ is smooth and radial which satisfies 
\begin{equation*}
\wt{\chi}(\xi) = 
\begin{cases}
1,&\xi \in \mathrm{supp} \chi,\\
0,&\vt \xi\vt > 1.
\end{cases}
\end{equation*}
Hence, we deduce \eqref{fml-lem-point-deduce} to prove that
\begin{equation*}
\int \wt{\chi}_\lambda \vt \nabla\phi\vt^2 \Ind{y} \lesssim \lambda^{-2} \int_{\vt y\vt<(4\lambda)^{-1}} \vt \mathcal{H}\phi(y)\vt^2 \Ind{y} + \lambda^2 \int_{\vt y\vt<(4\lambda)^{-1}} \vt \phi(y)\vt^2 \Ind{y}.
\end{equation*}
Integrating by parts and using the fact that $H = -\Delta + \vt y\vt^2$, we can find
\begin{align*}
-2 \int (\wt{\chi}_\lambda \phi) \Delta\phi \Ind{y} &= 2 \int \wt{\chi}_\lambda \vt \nabla\phi\vt^2 \Ind{y} - \int \vt \phi\vt^2 \Delta\wt{\chi}_\lambda \Ind{y}\\
&= 2 \int (\wt{\chi}_\lambda \phi) H\phi \Ind{y} - 2 \int \vt y\vt^2 \wt{\chi}_\lambda \vt \phi\vt^2 \Ind{y}.
\end{align*}
By Young's inequality, for any $\lambda > 0$ we have
\begin{equation*}
\Lf\vt (\wt{\chi}_\lambda \phi) H\phi\Rt\vt \lesssim \lambda^{-2} \vt H\phi\vt^2 + \lambda^2 \vt \phi\vt^2, 
\end{equation*}
this implies
\begin{align*}
2 \int \wt{\chi}_\lambda \vt \nabla\phi\vt^2 \Ind{y} + 2 \int \vt y\vt^2 \wt{\chi}_\lambda \vt \phi\vt^2 \Ind{y} &= 2 \int (\wt{\chi}_\lambda \phi) H\phi \Ind{y} + \int \vt \phi\vt^2 \Delta\wt{\chi}_\lambda \Ind{y}\\
&\lesssim \lambda^{-2} \int_{\vt x-y\vt<\lambda^{-1}}\vt H\phi\vt^2 \Ind{y} + (1+\lambda^2)\int_{\vt x-y\vt<\lambda^{-1}}\vt \phi\vt^2 \Ind{y}.
\end{align*}
So far, we have proved \eqref{fml-lem-point-deduce} and this lemma follows immediately.

\textbf{Step 2: Showing that $\eqref{fml-thm-asser} \Rightarrow \eqref{fml-thm-L2H1}$:}
Recall that $u_N = e^{itH}u_N(0),v_M = e^{itH}v_M(0)$ and $N \ge M$, it is sufficient to control the term $\vt v_M\nabla u_N\vt^2$ in \eqref{fml-thm-L2H1}. Apply Lemma \ref{lem-point-wise} with $\lambda = M$, we have
\begin{equation}\label{fml-thm-asser-half}
\begin{aligned}
\int_0^T \Lf( \int \vt v_M\nabla u_N\vt^2(x) \Ind{x} \Rt) \Ind{t} \lesssim& M^{-1} \int_0^T \Lf( \int_{\vt x-y\vt<M^{-1}} \vt Hv_M(y) \nabla_x u_N(x)\vt^2 \Ind{x}\Ind{y}\Rt) \Ind{t} \\
& +M^3 \int_0^T \Lf(\int_{\vt x-y\vt<M^{-1}} \vt v_M(y) \nabla_x u_N(x)\vt^2 \Ind{x}\Ind{y}\Rt) \Ind{t}.
\end{aligned}
\end{equation} 
Using triangle inequality and \eqref{fml-thm-asser}, we arrive at
\begin{align*}
&M^3\int_0^T \Lf(\int_{\vt x-y\vt<M^{-1}} \vt v_M(y) \nabla_x u_N(x)\vt^2 \Ind{x}\Ind{y}\Rt) \Ind{t}\\
\lesssim& M^3\int_0^T \Lf(\int_{\vt x-y\vt<M^{-1}} \vt v_M(y) \nabla_x u_N(x) + u_N(x)\nabla_y \ovr{v}_M(y)\vt^2 \Ind{x}\Ind{y}\Rt) \Ind{t} \\
&+ M^3\int_0^T \Lf(\int_{\vt x-y\vt<M^{-1}} \vt u_N(x)\nabla_y \ovr{v}_M(y)\vt^2 \Ind{x}\Ind{y}\Rt) \Ind{t}\\
\lesssim& M^{2}N \Vt v_M(0)\Vt_{\LpnR{2}} \Vt u_N(0)\Vt_{\LpnR{2}} + \int_0^T \Lf(\int_{\vt x-y\vt<M^{-1}} M^3 \vt u_N(x)\nabla_y \ovr{v}_M(y)\vt^2 \Ind{x}\Ind{y}\Rt) \Ind{t}.
\end{align*}
We can treat the first term in l.h.s of \eqref{fml-thm-asser-half} by \eqref{fml-thm-asser-Hvm} and get that
\begin{align*}
&\int_0^T \Lf( \int M^3 \vt v_M\nabla u_N\vt^2(x) \Ind{x} \Rt) \Ind{t}\\
\lesssim& M^{2}N \Vt v_M(0)\Vt_{\LpnR{2}} \Vt u_N(0)\Vt_{\LpnR{2}} 
+ \int_0^T \Lf(\int_{\vt x-y\vt<M^{-1}} M^3\vt u_N(x)\nabla_y \ovr{v}_M(y)\vt^2 \Ind{x}\Ind{y}\Rt) \Ind{t}\\
& + \int_0^T \Lf(\int_{\vt x-y\vt<M^{-1}} M^3\vt u_N(x)\nabla_y (H\ovr{v}_M)(y)\vt^2 \Ind{x}\Ind{y}\Rt) \Ind{t}.
\end{align*}
			
On one hand, by change of variable, the Cauchy-Schwartz inequality and Strichartz estimate, there holds
\begin{equation}\label{fml-thm-asser-half-1}
\begin{aligned}
&\int_0^T \Lf(\int_{\vt x-y\vt<M^{-1}} M^3\vt u_N(x)\nabla_y \ovr{v}_M(y)\vt^2 \Ind{x}\Ind{y}\Rt) \Ind{t}\\
=& \int_0^T \Lf(\int_{\vt z\vt<M^{-1}} M^3\vt u_N(x)\nabla_y \ovr{v}_M(x-z)\vt^2 \Ind{x}\Ind{z}\Rt) \Ind{t}\\
\lesssim& \int_{\vt z\vt<M^{-1}} M^3\Vt u_N\Vt^2_{\Tsnself{\frac{10}{3}}{\frac{10}{3}}{[0,T]}} \Vt \nabla v_M\Vt^2_{\Tsnself{5}{5}{[0,T]}} \Ind{z}\\
\lesssim& M^3\Vt u_N(0)\Vt_{\LpnR{2}}^2 \Vt v_M(0)\Vt_{\LpnR{2}}^2.
\end{aligned}
\end{equation}
Here in the last inequality of \eqref{fml-thm-asser-half-1}, we have used the fact that
\begin{equation}\label{fml-thm-d+2d+2}
\Vt \nabla v_M\Vt^2_{\Tsnself{5}{5}{[0,T]}} \lesssim M^{3} \Vt v_M(0)\Vt^2_{\LpnR{2}}.
\end{equation}	
Indeed, by Sobolev embedding, we get
\begin{equation*}
\Vt \nabla v_M\Vt_{\Tsnself{5}{5}{[0,T]}} \lesssim \Vt \vt \nabla\vt^{\frac{3}{2}} v_M\Vt_{\Tsnself{5}{\frac{30}{11}}{[0,T]}}.
\end{equation*}
Since $-\Delta v_M$ solves
\begin{equation*}
(i\partial_t+H)(-\Delta v_M) = [ i\partial_t+H,-\Delta]v_M = 6 v_M+2x\cdot\nabla v_M,
\end{equation*}
and by Sobolev embedding and Strichartz estimates, we have
\begin{align}\nonumber
\Vt \Delta v_M\Vt_{\Tsnself{5}{\frac{30}{11}}{[0,T]}}&\lesssim \Vt \Delta v_M(0)\Vt_{\LpnR{2}} +\Vt v_M\Vt_{\Tsnself{1}{2}{[0,T]}}+\Vt x\cdot\nabla v_M\Vt_{\Tsnself{1}{2}{[0,T]}} \\\nonumber
&\lesssim \Vt v_M(0)\Vt_{\CHsnR{2}}\\\label{equ:vM2der}
&\lesssim M^{2} \Vt v_M(0)\Vt_{\LpnR{2}}.
\end{align}	
Interpolating between above and 
\begin{equation*}
\Vt v_M\Vt_{\Tsnself{5}{\frac{30}{11}}{[0,T]}} \lesssim \Vt v_M(0)\Vt_{\LpnR{2}},
\end{equation*}
we conclude that
\begin{equation*}
\Vt \vt \nabla\vt^{\frac{3}{2}} v_M\Vt_{\Tsnself{5}{\frac{30}{11}}{[0,T]}} \lesssim M^{\frac{3}{2}} \Vt v_M(0)\Vt_{\LpnR{2}},
\end{equation*}
and \eqref{fml-thm-d+2d+2} follows immediately.

On the other hand, we apply \eqref{fml-thm-asser-Hvm} and similar argument above to obtain that
\begin{equation}\label{fml-thm-asser-half-H}
\begin{aligned}
&\int_0^T \Lf(\int_{\vt x-y\vt<M^{-1}} M^{-1}\vt u_N(x)\nabla_y (H\ovr{v}_M)(y)\vt^2 \Ind{x}\Ind{y}\Rt) \Ind{t} \lesssim M^3 \Vt u_N(0)\Vt_{\LpnR{2}}^2 \Vt v_M(0)\Vt_{\LpnR{2}}^2.
\end{aligned}
\end{equation}
			
Consequently, \eqref{fml-thm-asser-half-1} and \eqref{fml-thm-asser-half-H} gives
\begin{equation*}
\begin{aligned}
\int_0^T \Lf( \int \vt v_M\nabla u_N\vt^2(x) \Ind{x} \Rt) \Ind{t} &\lesssim (NM + M^{2})\Vt u_N(0)\Vt_{\LpnR{2}}^2 \Vt v_M(0)\Vt_{\LpnR{2}}^2\\
&\lesssim MN \Vt u_N(0)\Vt_{\LpnR{2}}^2 \Vt v_M(0)\Vt_{\LpnR{2}}^2,
\end{aligned}
\end{equation*}
and we have finished the proof of Step 2.
			
\textbf{Step 3: Proof of bilinear estimate \eqref{fml-thm-L2x2}}:
Using H\"older's inequality, we find that the l.h.s of \eqref{fml-thm-L2x2} can be controlled by
\begin{equation*}
\Vt \vt x\vt v_M\Vt_{\Tsnself{5}{5}{[0,T]}} \Vt u_N\Vt_{\Tsnself{\frac{10}{3}}{\frac{10}{3}}{[0,T]}}.
\end{equation*}
It turns to estimate
\begin{equation}\label{fml-thm-d+2d+2-x2}
\Vt \vt x\vt v_M\Vt^2_{\Tsnself{5}{5}{[0,T]}} \lesssim \Vt \vt x\vt v_M\Vt^2_{L_t^{5}([0,T],\dot W_x^{\frac{1}{2},\frac{30}{11}}(\R^3))} \lesssim M^3 \Vt v_M(0)\Vt^2_{\LpnR{2}}.
\end{equation}
By interpolation, it can be reduced to prove
\begin{equation}\label{fml-thm-x2-deri1/2}
\big\Vt \vt x\vt^2 v_M\big\Vt^2_{L_t^{5}([0,T],\dot W_x^{\frac{1}{2},\frac{30}{11}}(\R^3))} \lesssim M^5 \Vt v_M(0)\Vt^2_{\LpnR{2}}, 
\end{equation}
and
\begin{equation}\label{fml-thm-deri1/2}
\Vt v_M\Vt^2_{L_t^{5}([0,T],\dot W_x^{\frac{1}{2},\frac{30}{11}}(\R^3))} \lesssim M \Vt v_M(0)\Vt^2_{\LpnR{2}}.
\end{equation}
Next, we prove \eqref{fml-thm-x2-deri1/2}, and \eqref{fml-thm-deri1/2} can be obtained similarly.
			
We can verify that $\vt x\vt^2 v_M$ solves
\begin{align*}
(i\partial_t + H)(\vt x\vt^2 v_M) = [i\partial_t+H,\vt x\vt^2]v_M = -6v_M-2x\cdot\nabla v_M.
\end{align*}
Then, by the Bernstein inequality, Strichartz estimates and Lemma \ref{lem-Bernstein}, we get that
\begin{align}\nonumber
\Vt \vt x\vt^2 v_M\Vt_{\Tsnself{5}{\frac{30}{11}}{[0,T]}} &\lesssim \Vt \vt x\vt^2v_M\Vt_{\LpnR{2}} + \Vt v_M(0)\Vt_{\Tsnself{1}{2}{[0,T]}} + \Vt x\cdot\nabla v_M\Vt_{\Tsnself{1}{2}{[0,T]}}\\\label{fml-thm-x2v_M}
&\lesssim \Vt v_M(0)\Vt_{\CHsnR{2}} \lesssim M^2\Vt v_M(0)\Vt_{\LpnR{2}}.
\end{align}
We also verify that $\nabla\cdot (\vt x\vt^2v_M)$ satisfies
\begin{align*}
(i\partial_t + H)(\nabla\cdot (\vt x\vt^2v_M)) &= [i\partial_t+H,\nabla\cdot\vt x\vt^2]v_M \\
&= -6\Delta v_M-2x v_M-6\nabla v_M-4x\cdot\nabla v_M-2\vt x\vt^2x v_M.
\end{align*}
Applying Bernstein inequality, Strichartz estimates and Lemma \ref{lem-Bernstein}, we find that
\begin{equation}\label{fml-thm-nabx2v_M}
\begin{aligned}
&\Vt x\cdot\nabla (\vt x\vt^2 v_M)\Vt_{\Tsnself{5}{\frac{30}{11}}{[0,T]}}\\ \lesssim& \Vt v_M(0)\Vt_{\CHsnR{3}} +\Vt \Delta v_M\Vt_{\Tsnself{1}{2}{[0,T]}}+\Vt x v_M\Vt_{\Tsnself{1}{2}{[0,T]}}\\
&+\Vt \nabla v_M\Vt_{\Tsnself{1}{2}{[0,T]}}+\Vt x\cdot \nabla v_M\Vt_{\Tsnself{1}{2}{[0,T]}}+\Vt \vt x\vt^2x\cdot v_M\Vt_{\Tsnself{1}{2}{[0,T]}}\\
\lesssim& \Vt v_M(0)\Vt_{\CHsnR{3}}\lesssim M^3 \Vt v_M(0)\Vt_{\LpnR{2}}.
\end{aligned}
\end{equation}
Hence, by \eqref{fml-thm-x2v_M}, \eqref{fml-thm-nabx2v_M} and interpolation, we arrive at
\begin{align*}
\Vt \vt x\vt^2 v_M\Vt^2_{\Tsnself{5}{5}{[0,T]}} &\lesssim \Vt \vt \nabla\vt^{\frac{1}{2}} v_M\Vt^2_{\Tsnself{5}{\frac{30}{11}}{[0,T]}}\\
&\lesssim \Vt v_M\Vt_{\Tsnself{5}{\frac{30}{11}}{[0,T]}} \Vt \vt \nabla\vt v_M\Vt_{\Tsnself{5}{\frac{30}{11}}{[0,T]}} \\
&\lesssim M^5 \Vt v_M(0)\Vt^2_{\LpnR{2}}.
\end{align*}
Hence, we  get \eqref{fml-thm-x2-deri1/2}. 
			
\textbf{Step4: Showing that \eqref{fml-thm-L2H1-Har} implies \eqref{fml-bilinear-res}}: By Littlewood-Paley decomposition, we get
\begin{equation*}
\Vt v_M u_N\Vt_{\Tsnself{2}{2}{[0.T]}} \leq \Big(\sum_{K >N} \Vt \Delta_K(v_M u_N)\Vt^2_{\Tsnself{2}{2}{[0,T]}}\Big)^\frac12 +\big\|S_N(v_M u_N)\big|_{\Tsnself{2}{2}{[0,T]}},
\end{equation*}
where $S_N := \sum\limits_{K\le N} \Delta_K.$ For the first term,  by Bernstein inequality and \eqref{fml-thm-L2H1-Har}, we get that
\begin{align*}
\sum_{K >N} \Vt \Delta_K(v_M u_N)\Vt^2_{\Tsnself{2}{2}{[0,T]}} &\lesssim \sum_{K > N} (1+K^2)^{-1} \Vt v_M u_N\Vt^2_{\TsnCHself{2}{1}{[0,T]}}\\
&\lesssim M^3 N^{-1}\Vt v_M u_N\Vt^2_{\Tsnself{2}{2}{[0,T]}},
\end{align*}
which is desired. For the second term, we write
\begin{equation*}
S_N(v_M u_N) = S_N(v_M N^{-2}H\wt{u}_N),
\end{equation*}
here, $\wt{u}_N = \wt{\Delta}_N u_N$ with spectral localized operator $\wt{\Delta}_N$ which satisfies $N^{-2}H\wt{\Delta}_N \circ \Delta_N = \Delta_N$. By definition, we have
\begin{equation*}
v_M H \wt{u}_N = H(v_M \wt{u}_N) + \wt{u}_N \Delta v_M + 2\nabla v_M\cdot\nabla \wt{u}_N.
\end{equation*}
Since $S_N$ and $N^{-1}H^{\frac{1}{2}}S_N$ are uniformly bounded operator on $\LpnR{2}$, from \eqref{fml-thm-L2H1-Har} we obtain
\begin{align*}
&\Vt S_N(v_M u_N)\Vt^2_{\LpnR{2}} \\
\lesssim& N^{-2} \Vt N^{-1}H^{\frac{1}{2}} S_N H^{\frac{1}{2}}(v_M u_N)\Vt^2_{\LpnR{2}} + N^{-4} \Lf( \Vt \wt{u}_N \Delta v_M\Vt_{\LpnR{2}}^2 + \Vt \nabla\wt{u}_N\cdot \nabla v_M\Vt_{\LpnR{2}}^2\Rt)\\
\lesssim& N^{-2}\Vt H^{\frac{1}{2}}(v_M u_N)\Vt^2_{\LpnR{2}} + N^{-4}\Lf( \Vt \wt{u}_N\Vt^2_{\LpnR{\frac{15}{2}}}\Vt \Delta v_M\Vt^2_{\LpnR{\frac{30}{11}}} + \Vt \nabla \wt{u}_N\Vt^2_{\LpnR{\frac{10}{3}}}\Vt \nabla v_M\Vt^2_{\LpnR{5}} \Rt).
\end{align*}
Integrating in time, then, use H\"older's inequality, Strichartz estimates and \eqref{fml-thm-L2H1-Har} we conclude that
\begin{align*}
&\Vt S_N(v_M u_N)\Vt^2_{\Tsnself{2}{2}{[0,T]}} \\
\lesssim& CN^{-2} \Vt v_M u_N\Vt_{\TsnCHself{2}{1}{[0,T]}}^2 + M^3N^{-4}(N^2 + M^2) \Vt v_M(0)\Vt_{\LpnR{2}}^2 \Vt \wt{u}_N(0)\Vt_{\LpnR{2}}^2\\
\lesssim& (M^{2}N^{-1} + M^3N^{-2}) \Vt v_M(0)\Vt_{\LpnR{2}}^2 \Vt \wt{u}_N(0)\Vt_{\LpnR{2}}^2\\
\lesssim& (M^{2}N^{-1} + M^3N^{-2}) \Vt v_M(0)\Vt_{\LpnR{2}}^2 \Vt u_N(0)\Vt_{\LpnR{2}}^2.
\end{align*}
Thus, we obtain \eqref{fml-bilinear-res}.
			
Therefore, we complete the proof of Theorem \ref{thm-bilinear}.
\end{proof}
		
However, merely apply Theorem \ref{thm-bilinear} can not establish desired energy increment estimate and we concern a bilinear estimate with derivatives and weight:

\begin{corollary}\label{cor-bilinear-PDQD}
Let $N,M,v_M(0),u_N(0)$ be the same as in Theorem \ref{thm-bilinear}. Let multi-index $\alpha,\beta$ with $\mathrm{Ord}(\alpha)=k_1,\mathrm{Ord}(\beta)=k_2$ and $P\in \{\nabla,x\}$. 
Then, for $T \in (0,\infty)$, there exists constant $C>0$ such that for $d \ge 2$
\begin{align}\label{fml-bilinear-PDQD-res}
&\Vt P(\alpha) e^{itH}u_N(0) \cdot P(\beta) e^{itH}v_M(0)\Vt_{\Tsnselfd{2}{2}{[0,T]}}\\\nonumber
\le& C {N^{k_1}M^{k_2}} \frac{M^{\frac{d-1}{2}}}{N^{\frac{1}{2}}} \Vt u_N(0)\Vt_{\LpnRd{2}} \Vt v_M(0)\Vt_{\LpnRd{2}},
\end{align}
where $P(\alpha)$ and $P(\beta)$ are as in Definition \ref{def-pre-Palp}.
\end{corollary}

\begin{proof}
Let $u_N := e^{itH}u_N(0)$ and note that $i\partial_t$  always commutes with $P$, we get that
\begin{align*}
(i\partial_t + H)(P(\alpha) u_N) &= [i\partial_t + H,P(\alpha)]u_N + P(\alpha) (i\partial_t + H)u_N\\
&= [H,P(\alpha)]u_N,
\end{align*} 
and by Duhamel's formula, we also get
\begin{align*}
P(\alpha) u_N &= e^{itH}P(\alpha) u_N(0) - i\int_0^T e^{i(t-s)H} \big[H,P(\alpha)\big] e^{isH}u_N(0) \Ind{s}\\
&=e^{itH}P(\alpha) u_N(0) - i e^{itH}\int_0^T e^{-isH} \big[H,P(\alpha)\big] e^{isH}u_N(0) \Ind{s}.
\end{align*}
Similar argument with that of $v_M$, we arrive at
\begin{equation*}
{P(\beta) v_M = e^{itH}P(\beta)  v_M(0) - i e^{itH}\int_0^T e^{-isH} \big[H,P(\beta) \big] e^{isH}v_M(0) \Ind{s}.}
\end{equation*}
			
By Theorem \ref{thm-bilinear} and Lemma \ref{lem-Bernstein}, one can verify that
\begin{equation}
\Vt  e^{itH}(P(\alpha) u_N(0)) \cdot e^{itH}(P(\beta) v_M(0))\Vt^2_{\Tsnselfd{2}{2}{[0,T]}} \le C M^{k_1}N^{k_2} \frac{M^{d-1}}{N} \Vt u_N(0)\Vt^2_{\LpnRd{2}} \Vt v_M(0)\Vt^2_{\LpnRd{2}}.
\end{equation}
Hence, we need to deal with $\Tsnselfd{2}{2}{[0,T]}$ norm of the form
\begin{align}\label{equ:eituNest1}
e^{itH}P(\alpha)u_N(0)\times e^{itH}\int_0^T e^{-isH}\big [H,P(\beta)\big] e^{isH}v_M(0) \Ind{s},\\\label{equ:eituNest2}
e^{itH}\int_0^T e^{-isH} [H,P(\alpha)] e^{isH}u_N(0) \Ind{s}\times e^{itH}P(\beta)v_M(0),\\\label{equ:eituNest3}
e^{itH}\int_0^T e^{-isH} \big[H,P(\alpha)\big] e^{isH}u_N(0) \Ind{s} \times e^{itH}\int_0^T e^{-isH} \big[H,P(\beta) \big] e^{isH}v_M(0) \Ind{s}.
\end{align}   
Using Theorem \ref{thm-bilinear}   and Lemma \ref{lem-Bernstein} again,  we conclude that
\begin{align*}
&\|\eqref{equ:eituNest1}+\eqref{equ:eituNest2}+\eqref{equ:eituNest3}\|_{\Tsnselfd{2}{2}{[0,T]}}\\
\leq&C_T\frac{M^{\frac{d-1}{2}}}{N^{\frac{1}{2}}}\Big[ N^{k_1}\|u_N(0)\|_{L_x^2} \sup_{s}\big\|[H,P(\beta)\big] e^{isH}v_M(0) \big\|_{L_x^2}
+\sup_s\big\|[H,P(\alpha)] e^{isH}u_N(0) \big\|_{L_x^2} M^{k_2}\|v_M(0)\|_{L_x^2}\\
&\qquad\qquad\quad+\sup_s\big\|[H,P(\alpha)] e^{isH}u_N(0) \big\|_{L_x^2} \big\|[H,P(\beta)\big] e^{isH}v_M(0) \big\|_{L_x^2}\Big]\\
\leq&C_T{N^{k_1}M^{k_2}} \frac{M^{\frac{d-1}{2}}}{N^{\frac{1}{2}}} \Vt u_N(0)\Vt_{\LpnRd{2}} \Vt v_M(0)\Vt_{\LpnRd{2}}
\end{align*}		
provided that we can show that
\begin{equation}\label{equ:NMk1k2}
\big\|[H,P(\alpha)]u_N(0) \big\|_{L_x^2}\leq CN^{k_1}\|u_N(0)\|_{L_x^2}.
\end{equation}
To prove \eqref{equ:NMk1k2}, we first assert that 
\begin{equation}\label{fml-BiPDQD-asser}
[P(\alpha),H] = \sum_{\vt\wt{\alpha}\vt \le k_1} P(\wt{\alpha}).
\end{equation}
Once we admit assertion \eqref{fml-BiPDQD-asser} and by Lemma \ref{lem-Bernstein}
\begin{equation*}
\Vt [H,P(\alpha)]u_N\Vt_{\LpnRd{2}} \lesssim \sum_{\vt\wt{\alpha}\vt \le k_1} \Vt P(\wt{\alpha})u_N(0)\Vt_{\LpnRd{2}} \lesssim N^{k_1} \Vt u_N\Vt_{\LpnRd{2}}.
\end{equation*}
And so we can obtain \eqref{equ:NMk1k2}.
			
Next, we turn to prove assertion \eqref{fml-BiPDQD-asser}. Observing that
\begin{equation*}
[P(\alpha),H] = \sum_{j=1}^{k_1} P(\alpha_1)\cdots[P(\alpha_j),H]\cdots P(\alpha_{k_1}),
\end{equation*}
meanwhile, it is clear
\begin{equation*}
[P(1),H] = [\nabla,H] = 2x,\,\, [P(2),H] = [x,H] = -2\nabla.
\end{equation*}
Consequently, assertion \eqref{fml-BiPDQD-asser} holds for all $k_1 \in \mathbb{N}$.
			
Therefore, we conclude the proof of Corollary \ref{cor-bilinear-PDQD}.
\end{proof}
		
As a consequence of Theorem \ref{thm-bilinear}, we also obtain the following bilinear estimate in Bourgain space. 

\begin{lemma}[Bilinear estimate in $X^{0,b}$]\label{lem:BilBourg}
Let  dyadic numbers $1\leq M\leq N$. Then, there holds  for $b>\frac12$
\begin{gather*}
\big\|{\Delta}_Nu{\Delta}_Mv\big\|_{L_{t,x}^2}\lesssim\frac{M^{\frac{d-1}{2}}}{N^{\frac{1}{2}}}\|{\Delta}_Nu\|_{X^{0,b}}\|{\Delta}_Mv\|_{X^{0,b}},\\
\Vt P(\alpha) \Delta_{N}u \cdot P(\beta) \Delta_{M}v\Vt_{L_{t,x}^2}
\le C {N^{k_1}M^{k_2}} \frac{M^{\frac{d-1}{2}}}{N^{\frac{1}{2}}} \Vt \Delta_{N}u\Vt_{X^{0,b}}  \Vt \Delta_{M}v\Vt_{X^{0,b}},
\end{gather*}
with ${\rm ord}(\alpha)=k_1,\,{\rm ord}(\beta)=k_2.$
\end{lemma}

\begin{proof}
	The proofs of two estimates stated in this lemma are similar. So we only give the detailed proof of the first one.
We denote $F(t)$ and $G(t)$ by
\begin{align*}
F(t):=e^{itH}{\Delta}_Nu(t),\quad G(t):=e^{itH}{\Delta}_Mv(t).
\end{align*}
By definition, we have 
\begin{gather*}
{\Delta}_Nu(t)=\frac{1}{2\pi}\int_{\R}e^{it\tau}e^{-itH}\widehat{F}(\tau)d\tau,\\
{\Delta}_Mv(t)=\frac{1}{2\pi}\int_{\R}e^{it\tau}e^{-itH}\widehat{G}(\tau)d\tau.
\end{gather*}
Furthermore, it implies 
\begin{align*}
{\Delta}_Nu{\Delta}_Mv=\frac{1}{4\pi^2}\int_{\R^2}e^{it(\tau_1+\tau_2)}\big(e^{-itH}\widehat{F}(\tau_1)e^{-itH}\widehat{G}(\tau_2)\big)d\tau_1d\tau_2.
\end{align*}
Using the Minkowski inequality and Theorem \ref{thm-bilinear}, we obtain
\begin{align*}
\big\|{\Delta}_Nu{\Delta}_Mv\big\|_{L_{t,x}^2([0,1]\times\R^d)}&\lesssim\Big\|\int_{\R^2}e^{it(\tau_1+\tau_2)}\big(e^{-itH}\widehat{F}(\tau_1)e^{-itH}\widehat{G}(\tau_2)\big)d\tau_1d\tau_2\Big\|_{L_{t,x}^2([0,1]\times\R^d)}\\
&\lesssim\int_{\R^2}\big\|e^{-itH}\widehat{F}(\tau_1)e^{-itH}\widehat{G}(\tau_2)\big\|_{L_{t,x}^2([0,1]\times\R^d)}d\tau_1d\tau_2\\
&\lesssim\frac{M^{\frac{d-1}{2}}}{N^{\frac{1}{2}}}\|\langle\tau_1\rangle^b\widehat{F}(\tau_1)\big\|_{L_{\tau_1}^2(\R,L^2(\R^d))}\|\langle\tau_2\rangle^b\widehat{G}(\tau_2)\big\|_{L_{\tau_2}^2(\R,L^2(\R^d))}\\
&\lesssim\frac{M^{\frac{d-1}{2}}}{N^{\frac{1}{2}}}\|{\Delta}_Nu\|_{X^{0,b}}\|{\Delta}_Mv\|_{X^{0,b}}.
\end{align*}
Hence, we complete the proof.
\end{proof}

Using Bernstein's inequality and Lemma \ref{lem:SobXTl4l2}(3), we get
\begin{align*}
\|\Delta_{N}u\Delta_Mv\|_{L_{t,x}^2}&\lesssim \|\Delta_{N}u\|_{L_t^4L_x^2}\|\Delta_{M}v\|_{L_t^4L_x^\infty}\\
&\lesssim M^{\frac{d}{2}}\|\Delta_{N}u\|_{L_t^4L_x^2}\|\Delta_{M}v\|_{L_t^4L_x^2}\\
&\lesssim M^{\frac{d}{2}}\|\Delta_{N}u\|_{X_H^{0,\frac14+}}\|\Delta_{M}v\|_{X_H^{0,\frac14+}}.
\end{align*}
Then interpolating with Lemma 3.2, we have the following bilinear estimates with $b^\prime\in(\frac14,\frac{1}{2})$.

\begin{lemma}\label{lem:bilinearbbes}
Let dyadic numbers $1\leq M\leq N$. There exists $b^\prime\in(\frac14,\frac12)$ and $\delta\in(0,\frac12)$ such that
\begin{align*}
\big\|\Delta_{N}u\Delta_{M}v\big\|_{L_{t,x}^2}\lesssim M^{\frac{d-1}{2}+\delta}N^{-\frac{1}{2}+\delta}\|\Delta_{M}u\|_{X_H^{0,b^\prime}}\|\Delta_{N}v\|_{X_H^{0,b^\prime}}.
\end{align*}
\end{lemma}
		
Similarly, we have the following two Strichartz estimates in Bourgain space.
\begin{align}\label{linear-strichartz}
\|u\|_{L_{t,x}^p}\lesssim \|u\|_{X^{\alpha(p),\frac{1}{2}+}},
\end{align}
where $\alpha(p)=\frac{d}{2}-\frac{d+2}{p}+$.
		
{Next, using Lemma \ref{lem:bilinearbbes} and the same argument as in \cite[Proposition 6.2]{BPT}, we get the following trilinear estimate, which is crucial in the proof of the local well-posedness. }
		
\begin{lemma}[Nonlinear estimate]\label{Non-Est}
Let $s_0=\frac{d-2}{2}$. For any $s>s_0$, there exists $(b,b^\prime)\in\R^2$ satisfying
\begin{align*}
0<b^\prime<\frac{1}{2}<b,\quad b+b^\prime<1.
\end{align*}
Then it holds
\begin{align*}
\big\||u|^2u\big\|_{X^{s,-b^\prime}}\lesssim \|u\|_{X^{s,b}}^3.
\end{align*}
\end{lemma}
	
\subsection{I-operator and the modified $I$ equation}

We first introduce the $I$-operator adapted to our setting. Using the functional calculus, we can define the spectral multiplier by
\begin{align*}
I_N(u)=I(\sqrt{H})u=\sum_{k\in\N}m({\nu_k})\pi_ku,
\end{align*}
where $\nu_k$ is the $k$-th eigenvalue associated to the operator $\sqrt{H}$. Here, we use the eigenvalues after rearrangement. for $s>1$, we define
\begin{align*}
m(k):=
\begin{cases}
1,&\mbox{if }k\leq N,\\
\big(\frac{N}{k}\big)^{1-s},&\mbox{if }k\geq 2N
\end{cases}
\end{align*}

\begin{remark}\label{Rem-local-wellposed}
It's easy to check that 
\begin{align*}
\|f\|_{\mathcal{H}^s(\R^d)}\lesssim N^{s-1}\|If\|_{\mathcal H^1(\R^d)}\lesssim N^{s-1}\|f\|_{\mathcal H^s(\R^d)}.
\end{align*}
\end{remark}

Next, we act the $I$ operator in both side of \eqref{fml-NLS} to get the following modified equation
\begin{align}\label{fml-I-system}
\begin{cases}
i\partial_tIu-H Iu=I(|u|^2u),&(t,x)\in \R\times\R^d,\\
Iu(0,x)=Iu_0\in\mathcal{H}^1(\R^d).
\end{cases}
\end{align}
First, we will establish the local-well-posedness for the modified equation $Iv$.
		
\begin{proposition}[Local well-posedness for $I$-equation]\label{prop-local}
Let $s>1$ and $u_0\in\mathcal{H}^s(\R^d)$. There exist a constant $c>0$ and  a time interval $[0,\delta]$ satisfying that $\delta\sim\|Iu_0\|_{\mathcal{H}^1(\R^d)}^{-c}$, the solution to \eqref{fml-I-system} obeys the following estimate
\begin{align*}
\|Iu\|_{X^{1,\frac12+}([0,\delta]\times\R^d)}\lesssim\|Iu_0\|_{\mathcal{H}^1(\R^d)}.
\end{align*}
\end{proposition}

\begin{proof}
By the Duhamel principle, we can write the solution $Iv$ as
\begin{align*}
Iu(t,x)=e^{-itH}Iu_0-\int_{0}^{t}e^{i(t-s)H}I(|u|^2u)(s)ds.
\end{align*}
For $0<b^\prime<\frac{1}{2}<b$ and $b+b^\prime<1$, we have
\begin{align*}
\|Iu\|_{X^{1,b}}\lesssim\|Iv_0\|_{\mathcal{H}^1(\R^d)}+\delta^{1-b-b^\prime}\|I(|v|^2v)\|_{X^{1,-b^\prime}}.
\end{align*}
On the other hand,  using the bound for multiplier $m(\xi)$ and following the proof of Lemma \ref{Non-Est}(see \cite[Proposition 6.2]{BPT}),  we  deduce that 
\begin{align*}
\|I(|v|^2v)\|_{X^{1,-b^\prime}([0,\delta]\times\R^d)}\lesssim\|Iv\|_{X^{1,b}([0,\delta]\times\R^d)}^3.
\end{align*}
Then using the continuity method, we conclude that for $\delta\sim \|Iu_0\|_{\mathcal H^1(\R^d)}^{-c}$,
\begin{align*}
\|Iv\|_{X^{1,b}([0,\delta]\times\R^d)}\lesssim\|Iu_0\|_{\mathcal H^1(\R^d)},\quad b>\tfrac{1}{2}.
\end{align*}
			
Therefore, we complete the proof of Proposition \ref{prop-local}.
\end{proof}
		


\section{Spectral localization of eigenfunctions associated to harmonic oscillator}\label{Sec-Interfour}
		
In this section, we study the spectral localization of products of eigenfunctions. Let $f,g$ be the eigenfunctions of $H$ spectrally localized around $m$ and $n$, a good question is that whether for $\pi_nf$ and $\pi_ng$, where is the location of  $\pi_nf\pi_mg$ in frequency space. In this paper, we will focus on the spectral localization of three product of eigenfunctions, which will be crucial in using the I-method. We refer the detail to Subsection \ref{subsec:modifyenergy}.
		
As presented in the previous section, we known that if $k$ is sufficiently large, we have the good property of spectral localization. For convenience, we recall the property here. 

\begin{lemma}[Almost orthogonal, \cite{Akahori,BPT}]\label{lem:orteigde}
Let $e_1,\cdots,e_4$ be $L^2$ normalized eigenfunction of the operator $H$ with respect to the eigenvalues $\mu_1^2,\mu_2^2,\mu_3^2,\mu_4^2$. Then, there exists $C>0$ such that for $\mu_1^2\geq C(\mu_2^2+\mu_3^2+\mu_4^2)$, it holds
\begin{align*}
\left|\int_{\R^d}e_1e_2e_3e_4dx\right|\leqslant C_p\mu_1^{-p},\quad \forall\;p>0.
\end{align*}
\end{lemma}  
		
However, in using the $I$-method, the main difficult part arise from the interaction of $\mu_1\sim \mu_2$. Thus, the above estimate is too crude to our application. We need to give the explicit formula for the left-hand side of the above estimate. For compact manifold, Hani \cite{Hani2} first establish such type formula for product of eigenfunctions	to Laplace-Beltrami operator. In the following, we will establish the spectral localization property for eigenfunctions to the Schr\"odinger operator $-\Delta+|x|^2$. This is the first result of spectral localization for Schr\"odinger operator with trapping potential in $\R^d$.
		
\begin{theorem}\label{thm-inter4}
Let $e_1,e_2,e_3$ and $e_4$ be eigenfunctions of $H$ corresponding eigenvalues $\mu_1^2,\mu_2^2,\mu_3^2$ and $\mu_4^2$. We set 
\begin{equation*}
L_0(e_1,e_2,e_3,e_4) = \int_{\R^d} e_1(x)e_2(x)e_3(x)e_4(x) \Ind{x}.
\end{equation*}
Then, for any $k\in\mathbb{N}$, we have
\begin{equation*}
L_0 = \frac{(-2)^k}{(\mu_1^2 - \mu_2^2- \mu_3^2- \mu_4^2)^k} L_k.
\end{equation*}
Here $L_k$ is the form
\begin{equation*}
L_k(e_1,e_2,e_3,e_4) = \int e_1(x)\cdot \wt{R}_{k}(e_2,e_3,e_4)(x) \Ind{x},
\end{equation*}
where $\wt{R}_{k}(e_2,e_3,e_4)$ is linear combination of the form
\begin{equation*}
P(\alpha) e_2(x)P(\beta) e_3(x)P(\gamma) e_4(x),
\end{equation*}
with $P\in\{\nabla,x\}$, $\mathrm{Ord}(\alpha) + \mathrm{Ord}(\beta) + \mathrm{Ord}(\gamma) \le 2k$ and $0\le \mathrm{Ord}(\alpha),\mathrm{Ord}(\beta),\mathrm{Ord}(\gamma)\le k$.
\end{theorem}
		
Before giving proof for Theorem \ref{thm-inter4} ,we introduce some notation that will be frequently used. Let $e_1,e_2,e_3,e_4$ be four functions, for a quadrilinear functional 
\begin{equation*}
L(e_1,e_2,e_3,e_4) = \int \mathcal{L}(e_1,e_2,e_3,e_4)(x) \Ind{x},
\end{equation*}
with quadrilinear form $\mathcal{L}$ and measurable function $\rho(x)$, we define
\begin{equation*}
L^{\rho(x)}(e_1,e_2,e_3,e_4) = \int \rho(x)\mathcal{L}(e_1,e_2,e_3,e_4)(x) \Ind{x}.
\end{equation*}
		
Let $k\in\mathbb{N}$ and $P(\alpha)$ be defined in Definition \ref{def-pre-Palp}. Then, we define $\wt{Q}_k(e_2,e_3,e_4)$ as the linear combination of the form
\begin{equation*}
P(\alpha)e_2 P(\beta)e_3 P(\gamma)e_4,
\end{equation*}
with $\mathrm{Ord}(\alpha)+\mathrm{Ord}(\beta)+\mathrm{Ord}(\gamma)\le 2k$, $\mathrm{Ord}(\alpha),\mathrm{Ord}(\beta),\mathrm{Ord}(\gamma) \le k$ and each component of $\alpha,\beta,\gamma$ is identically $1$. We also define $\wt{R}_k(e_2,e_3,e_4)$ as the linear combination of $P(\alpha)e_2 P(\beta)e_3 P(\gamma)e_4$
with $\mathrm{Ord}(\alpha)+\mathrm{Ord}(\beta)+\mathrm{Ord}(\gamma)\le 2k$ and $\mathrm{Ord}(\alpha),\mathrm{Ord}(\beta),\mathrm{Ord}(\gamma) \le k$.

\begin{remark}
Indeed, to avoid introducing too many notations, the linear combination here does not necessarily need multi-index $\alpha,\beta,\gamma$ take all possible values. To be more precisely, for some $E$ which related to $\wt{R}_k$ it-self, be the subset of
\begin{equation*}
\big\{ (\alpha,\beta,\gamma)| \mathrm{Ord}(\alpha)+\mathrm{Ord}(\beta)+\mathrm{Ord}(\gamma)\le 2k,\,\, \alpha,\beta,\gamma\le k\big\},
\end{equation*}
such that
\begin{equation*}
\wt{R}_k(e_2,e_3,e_4) = \sum_{(\alpha,\beta,\gamma)\in E} P(\alpha)e_2 P(\beta)e_3 P(\gamma)e_4,
\end{equation*}
and $\wt{Q}_k$ has similar properties.
			
Thus, there is no wonder when one meet equality of the form
\begin{gather*}
\wt{Q}_k(e_2,e_3,e_4) = \wt{R}_k(e_2,e_3,e_4),\,\, \wt{R}_k(e_2,e_3,e_4)+\wt{R}_{k+1}(e_2,e_3,e_4)=\wt{R}_{k+1}(e_2,e_3,e_4),\\
\,\, \wt{R}_k(\nabla e_2,e_3,e_4) = \wt{R}_{k+1}(e_2,e_3,e_4),\,\, \wt{R}_k(xe_2,e_3,e_4) = \wt{R}_{k+1}(e_2,e_3,e_4),\,\, \text{etc.}
\end{gather*}
\end{remark}
		
\begin{proof}[{\bf Proof of Theorem \ref{thm-inter4}:}]
First, we set 
\begin{equation}
L_0 (e_1,e_2,e_3,e_4) := \int e_1(x)e_2(x)e_3(x)e_4(x) \Ind{x},
\end{equation}
with $He_j = n_j^2 e_j$ for all $j=1,2,3,4$. Basic computation gives
\begin{equation*}
\mu_1^2L_0 (e_1,e_2,e_3,e_4) = L_0 (He_1,e_2,e_3,e_4) = L_0 ((-\Delta)e_1,e_2,e_3,e_4) + L^{\vt x\vt^2}_0 (e_1,e_2,e_3,e_4).
\end{equation*}
On one hand, integration by parts gives
\begin{align}\nonumber
&L_0 ((-\Delta)e_1,e_2,e_3,e_4) = \int e_1(x)\cdot (-\Delta)(e_2e_3e_4)(x) \Ind{x}\\\nonumber
=& \int (e_1e_3e_4)(x)\cdot (-\Delta e_2)(x) \Ind{x} + \int (e_1e_2e_4)(x)\cdot (-\Delta e_3)(x) \Ind{x} \\\label{fml-fourinter-Del}
&+ \int (e_1e_2e_3)(x)\cdot (-\Delta e_4)(x) \Ind{x}-2L_1(e_1,e_2,e_3,e_4),
\end{align}
where
\begin{align*}
L_1(e_1,e_2,e_3,e_4) =& \int \nabla e_2(x)\cdot\nabla e_3(x) (e_1e_4)(x) \Ind{x}\\
& + \int \nabla e_2(x)\cdot\nabla e_4(x) (e_1e_3)(x) \Ind{x} + \int \nabla e_3(x)\cdot\nabla e_4(x) (e_1e_2)(x) \Ind{x}.
\end{align*}
On the other hand, we obtain
\begin{align}\label{fml-fourinter-x2}
&L^{\vt x\vt^2}_0(e_1,e_2,e_3,e_4) = \int (e_1e_3e_4)(x)\cdot (\vt x\vt^2e_2)(x) \Ind{x}\\\nonumber
&+ \int (e_1e_2e_4)(x)\cdot (\vt x\vt^2e_3)(x) \Ind{x} + \int (e_1e_2e_3)(x)\cdot (\vt x\vt^2e_4)(x) \Ind{x}- 2L^{\vt x\vt^2}_0(e_1,e_2,e_3,e_4). 
\end{align}
Recall that $L^{\vt x\vt^2}_0(e_1,e_2,e_3,e_4)$ has the form
\begin{equation*}
L^{\vt x\vt^2}_0(e_1,e_2,e_3,e_4) = \int \vt x\vt^2 (e_1e_2e_3e_4)(x) \Ind{x}.
\end{equation*}
Combining \eqref{fml-fourinter-Del} with \eqref{fml-bilinear-res}, then $\mu_1^2L_0(e_1,e_2,e_3,e_4)$ can be rewritten as 
\begin{equation}\label{fml-fourinter-iter-1}
L_0(e_1,e_2,e_3,e_4) =  \frac{-2}{\mu_1^2-\mu_2^2-\mu_3^2-\mu_4^2}\Lf(L_1(e_1,e_2,e_3,e_4) + L^{\vt x\vt^2}_0(e_1,e_2,e_3,e_4)\Rt).
\end{equation}
Since both $L_1(e_1,e_2,e_3,e_4)$ and $L^{\vt x\vt^2}_0(e_1,e_2,e_3,e_4)$ can be absorbed in 
\begin{equation*}
\int e_1(x) \wt{R}_1(e_2,e_3,e_4)(x) \Ind{x},
\end{equation*}
we have proved this theorem for $k=1$.
			
As for $k\ge2$, things are quite different. We repeat the following process $L_1(e_1,e_2,e_3,e_4)$ first:
\begin{align*}
\mu_1^2L_1 (e_1,e_2,e_3,e_4) &= L_1 (He_1,e_2,e_3,e_4) \\
&= L_0(He_1,\nabla e_2,\nabla e_3,e_4) + L_0(He_1,\nabla e_2,e_3,\nabla e_4) + L_0(He_1,e_2,\nabla e_3,\nabla e_4).
\end{align*}
We focus on the first term of $L_1(He_1,e_2,e_3,e_4)$:
\begin{align*}
L_0(He_1,\nabla e_2,\nabla e_3,e_4) &=\int (e_4He_1)(x)(\nabla e_2\cdot\nabla e_3)(x) \Ind{x} \\
&= \int e_4(x)(-\Delta e_1)(x)(\nabla e_2\cdot\nabla e_3)(x) \Ind{x} + \int e_4(x)(\vt x\vt^2e_1)(x)(\nabla e_2\cdot\nabla e_3)(x) \Ind{x}.
\end{align*}
			
On one hand, we apply integration by parts to obtain
\begin{equation}\label{fml-fourinter-R2-mark}
\begin{aligned}
&\int (-\Delta)e_1(x)(\nabla e_2\cdot\nabla e_3)(x)e_4(x) \Ind{x} \\
=& L_0(e_1,-\Delta\nabla e_2,\nabla e_3,e_4)+L_0(e_1,\nabla e_2,-\Delta\nabla e_3,e_4)+L_0(e_1,\nabla e_2,\nabla e_3,-\Delta e_4) \\
&+ \int e_1(x) \wt{Q}_2(e_2,e_3,e_4)(x) \Ind{x},
\end{aligned}
\end{equation}
			
On the other hand, we also get
\begin{equation}\label{fml-fourinter-R2-mark-x2}
\begin{aligned}
L_0^{\vt x\vt^2}(e_1,\nabla e_2,\nabla e_3,e_4) =& L_0(e_1,\vt x\vt^2\nabla e_2,\nabla e_3,e_4)+L_0(e_1,\nabla e_2,\vt x\vt^2\nabla e_3,e_4)\\		&+L_0(e_1,\nabla e_2,\nabla e_3,\vt x\vt^2e_4)-2 L_0^{\vt x\vt^2}(e_1,\nabla e_2,\nabla e_3,e_4).
\end{aligned}
\end{equation}
From \eqref{fml-fourinter-R2-mark} and \eqref{fml-fourinter-R2-mark-x2}, we obtain
\begin{align}\label{fml-fourinter-L1-t1}
&L_0(He_1,\nabla e_2,\nabla e_3,e_4)\\\nonumber
 =& \frac{-2}{\mu_1^2-\mu_2^2-\mu_3^2-\mu_4^2}\Lf( -\frac{1}{2}\int e_1(x) \wt{Q}_2(e_2,e_3,e_4)(x) \Ind{x}+L_0^{\vt x\vt^2}(e_1,\nabla e_2,\nabla e_3,e_4) \Rt).
\end{align}
			
Then, we handle other two terms in $L_1(e_1,e_2,e_3,e_4)$ similarly and also obtain identities in the form of \eqref{fml-fourinter-L1-t1}. Therefore, we find that
\begin{equation}\label{fml-fourinter-L1toL2}
L_1(e_1,e_2,e_3,e_4) = \frac{-2}{\mu_1^2-\mu_2^2-\mu_3^2-\mu_4^2} \big( L_2(e_1,e_2,e_3,e_4) + L_1^{\vt x\vt^2}(e_1,e_2,e_3,e_4) \big),
\end{equation} 
where
\begin{equation*}
L_2(e_1,e_2,e_3,e_4) := -\frac{1}{2} \int e_1(x)\wt{Q}_2(e_2,e_3,e_4)(x) \Ind{x}.
\end{equation*}
			
Substituting the above result into \eqref{fml-fourinter-iter-1}, we conclude that
\begin{align*}
L_0(e_1,e_2,e_3,e_4) =& \frac{(-2)^2}{(\mu_1^2-\mu_2^2-\mu_3^2-\mu_4^2)^2} \Lf(L_2(e_1,e_2,e_3,e_4) + L_1^{\vt x\vt^2}(e_1,e_2,e_3,e_4)\Rt) \\
&+ \frac{-2}{\mu_1^2-\mu_2^2-\mu_3^2-\mu_4^2}L_0^{\vt x\vt^2}(e_1,e_2,e_3,e_4),
\end{align*}
			
Similar to the argument of obtaining \eqref{fml-fourinter-L1toL2}, we have for $s\in\mathbb{N}$
\begin{equation}\label{fml-fourinter-Ls-1toLs}
L_{s-1}(e_1,e_2,e_3,e_4) = \frac{-2}{\mu_1^2-\mu_2^2-\mu_3^2-\mu_4^2} \big( L_{s}(e_1,e_2,e_3,e_4) + L_{s}^{\vt x\vt^2}(e_1,e_2,e_3,e_4) \big),
\end{equation}
where $L_{s}(e_1,e_2,e_3,e_4)$ is the form
\begin{equation*}
L_s(e_1,e_2,e_3,e_4) = \int e_1(x) \wt{Q}_{s}(e_2,e_3,e_4)(x) \Ind{x}.
\end{equation*}
			
Thus, substituting \eqref{fml-fourinter-Ls-1toLs} into \eqref{fml-fourinter-iter-1}, we get for all $k\in\mathbb{N}$ 
\begin{equation}\label{fml-fourinter-iter-ell}
\begin{aligned}
L_0(e_1,e_2,e_3,e_4) &= \frac{(-2)^{k}}{(\mu_1^2-\mu_2^2-\mu_3^2-\mu_4^2)^{k}} \big(L_{k}(e_1,e_2,e_3,e_4)+L^{\vt x\vt^2}_{k-1}(e_1,e_2,e_3,e_4)\big)\\
&\hspace{2ex}+ \sum_{s = 0}^{k-2}\frac{(-2)^{s+1}}{(\mu_1^2-\mu_2^2-\mu_3^2-\mu_4^2)^{s+1}}L_s^{\vt x\vt^2}(e_1,e_2,e_3,e_4),
\end{aligned}
\end{equation}
by induction.
			
It remains to show that for $s,k\in\mathbb{N}$ and $s\ge k$ such that
\begin{equation}\label{fml-fourinter-Ls-Lk}
L_s^{\vt x\vt^2}(e_1,e_2,e_3,e_4) = \frac{(-2)^{k-s-1}}{(\mu_1^2-\mu_2^2-\mu_3^2-\mu_4^2)^{k-s-1}}\int e_1(x) \wt{R}_{k}(e_2,e_3,e_4)(x) \Ind{x}.
\end{equation}
			
We first prove the case when $s=0$ and $k=1$. Noting that
\begin{equation*}
\begin{aligned}
\mu_1^2 L_0^{\vt x\vt^2}(e_1,e_2,e_3,e_4) = L_0^{\vt x\vt^2}(He_1,e_2,e_3,e_4)= L_0^{\vt x\vt^2}((-\Delta)e_1,e_2,e_3,e_4) + L_0^{\vt x\vt^4}(e_1,e_2,e_3,e_4),
\end{aligned}
\end{equation*}
and
\begin{equation*}
-\Delta(\vt x\vt^2e_2e_3e_4) = \vt x\vt^2\cdot-\Delta(e_2e_3e_4) -2 \Lf( x\cdot\nabla e_2\cdot e_3e_4 + x\cdot\nabla e_3\cdot e_2e_4+x\cdot\nabla e_4\cdot e_2e_3 \Rt).
\end{equation*}
we arrive at
\begin{align}\label{fml-fourinter-x2-Del}
&L_0^{\vt x\vt^2}((-\Delta)e_1,e_2,e_3,e_4) \\\nonumber
=& L_0^{\vt x\vt^2}(e_1,(-\Delta)e_2,e_3,e_4) + L_0^{\vt x\vt^2}((e_1,e_2,(-\Delta)e_3,e_4) + L_0^{\vt x\vt^2}(e_1,e_2,e_3,(-\Delta)e_4) \\\nonumber
&- 2d L_0(e_1,e_2,e_3,e_4) - 2\big(L_0(e_1,x\cdot\nabla e_2,e_3,e_4) + L_0(e_1,e_2,x\cdot\nabla e_3,e_4) + L_0(e_1,e_2,e_3,x\cdot\nabla e_4)\big)\\\nonumber
& - 2L_1^{\vt x\vt^2}(e_1,e_2,e_3,e_4).
\end{align}
We also note that
\begin{align}\nonumber
L_0^{\vt x\vt^4}(e_1,e_2,e_3,e_4)=& L_0^{\vt x\vt^2}(e_1,\vt x\vt^2e_2,e_3,e_4) + L_0^{\vt x\vt^2}(e_1,e_2,\vt x\vt^2e_3,e_4) + L_0^{\vt x\vt^2}(e_1,e_2,e_3,\vt x\vt^2e_4)\\\label{fml-fourinter-x2-x2}
& -2L_0^{\vt x\vt^4}(e_1,e_2,e_3,e_4).
\end{align}
From \eqref{fml-fourinter-x2-Del} and \eqref{fml-fourinter-x2-x2}, we get that
\begin{equation*}
\begin{aligned}
&L_0^{\vt x\vt^2}(e_1,e_2,e_3,e_4) \\
=&\frac{-2}{\mu_1^2-\mu_2^2-\mu_3^2-\mu_4^2}\big(L_0(e_1,x\cdot\nabla e_2,e_3,e_4) + L_0(e_1,e_2,x\cdot\nabla e_3,e_4) + L_0(e_1,e_2,e_3,x\cdot\nabla e_4)\big)\\
&+\frac{-2}{\mu_1^2-\mu_2^2-\mu_3^2-\mu_4^2} \big(L_1^{\vt x\vt^2}(e_1,e_2,e_3,e_4) +L_0^{\vt x\vt^4}(e_1,e_2,e_3,e_4)\big).
\end{aligned}
\end{equation*} 
			
Since 
\begin{align*}
&L_0(e_1,x\cdot\nabla e_2,e_3,e_4) + L_0(e_1,e_2,x\cdot\nabla e_3,e_4) + L_0(e_1,e_2,e_3,x\cdot\nabla e_4)\\
&\hspace{8ex}+L_1^{\vt x\vt^2}(e_1,e_2,e_3,e_4) +L_0^{\vt x\vt^4}(e_1,e_2,e_3,e_4),
\end{align*}
can be absorbed into $\wt{R}_k(e_2,e_3,e_4)$. Thus have proved \eqref{fml-fourinter-Ls-Lk} for $s=0$ and $k=1$.
			
Next, we consider the case that $s=0$ and $k\in \mathbb{N}$. Indeed, we rewrite
\begin{align*}
\mu_1^2 \int e_1(x) \wt{R}_{k}(e_2,e_3,e_4)(x) \Ind{x} &= \int He_1(x) \wt{R}_{k}(e_2,e_3,e_4)(x) \Ind{x} \\
&= \int (-\Delta e_1(x)) \wt{R}_{k}(e_2,e_3,e_4)(x) \Ind{x} + \int \vt x\vt^2e_1(x) \wt{R}_{k}(e_2,e_3,e_4)(x) \Ind{x}.
\end{align*}
Similar with we treat $L_0^{\vt x\vt^2}(e_1,e_2,e_3,e_4)$ and $L_1(e_1,e_2,e_3,e_4)$ above and by Proposition \ref{Prop-pre-BPT}, the following holds
\begin{equation*}
\int e_1(x) \wt{R}_{k}(e_2,e_3,e_4)(x) \Ind{x} = \frac{-2}{\mu_1^2-\mu_2^2-\mu_3^2-\mu_4^2} \int e_1(x) \wt{R}_{k+1}(e_2,e_3,e_4)(x) \Ind{x}.
\end{equation*}
So far, we get \eqref{fml-fourinter-Ls-Lk} for $s=0$ and $k \in \mathbb{N}$.
			
It remains to prove for the case $s,k\in \mathbb{N}$ and $s\le k$. Recall that $L_s(e_1,e_2,e_3,e_4)$ has the form
\begin{equation*}
L_s(e_1,e_2,e_3,e_4) = \int e_1(x) \wt{Q}_s(e_2,e_3,e_4)\Ind{x}.
\end{equation*}
Similarly, we get that
\begin{align*}
\mu_1^2 \int e_1(x) \wt{Q}_{s}(e_2,e_3,e_4)(x) \Ind{x} &= \int He_1(x) \wt{Q}_{s}(e_2,e_3,e_4)(x) \Ind{x} \\
&= \int (-\Delta e_1(x)) \wt{Q}_{s}(e_2,e_3,e_4)(x) \Ind{x} + \int \vt x\vt^2e_1(x) \wt{Q}_{s}(e_2,e_3,e_4)(x) \Ind{x}.
\end{align*}
The structure of $\wt{Q}_{s}$ is simpler that that of $\wt{R}_{s}$, so we have
\begin{equation*}
\int e_1(x) \wt{Q}_{k}(e_2,e_3,e_4)(x) \Ind{x} = \frac{-2}{\mu_1^2-\mu_2^2-\mu_3^2-\mu_4^2} \int e_1(x) \wt{R}_{k+1}(e_2,e_3,e_4)(x) \Ind{x},
\end{equation*}
for all $k\in \mathbb{N}$. Repeating above operation for $k-s$ times, we obtain \eqref{fml-fourinter-Ls-Lk} for all $s,k\in \mathbb{N}$ and $s\le k$.
			
Thus, the proof of this theorem is complete.
\end{proof}



\section{Modified energy and polynomial growth of higher-order Sobolev norm}\label{Sec-EI-Growth}
		
In this section, we will first show the decay of modified energy and then obtain the polynomial growth of higher-order Sobolev norm.
		
Since $Iu$ is not the solution to \eqref{fml-NLS}, the following Hamilton energy
\begin{align*}
E(Iu)=\int_{\R^d}\Big(\frac{1}{2}|\nabla Iu(t,x)|^2+\frac12|x|^2|Iu(t,x)|^2+\frac{1}{4}|Iu(t,x)|^4\Big)dx
\end{align*}
is not conserved. Indeed, we can show the increment of $E(Iu)$ is almost conserved. 
		
\subsection{Decay of the modified energy}\label{subsec:modifyenergy}
We next prove the almost conservation of the modified energy:
		
\begin{proposition}\label{prop-EI}
Let $s_0=\frac{3}{2}$ when $d=2$ and $s_0=1$ when $d=3$. Suppose that  if $u(t):\;[0,\delta]\times\R^d\to\C$ is a solution to $\eqref{fml-NLS}$ and then we have
\begin{align*}
\left|E(Iu(t))-E(Iu_0)\right|\lesssim N^{-s_0+},
\end{align*}
for all $t\in[0,\delta]$, where $\delta$ is the same as in Proposition \ref{prop-local}.
\end{proposition}

\begin{proof}
First, we take the time derivative to the modified energy $E(Iu(t))$:
\begin{align*}
\frac{d}{dt}E(Iu(t))&=\Re\int_{\R^d}\overline{Iu_t}(HIu+|Iu|^2Iu)dx\\
&=\Re\int_{\R^d}\overline{Iu_t} \big(-i\partial_t(Iu)+H Iu+|Iu|^2Iu\big)dx\\
&=\Re\int_{\R^d}\overline{Iu_t} \big(-I(|u|^2u)+|Iu|^2Iu\big)dx.
\end{align*}
Decomposing the function $u(t,x)=\sum\limits_{\mu}{\pi}_{\mu}u(t,x)$, where ${\pi}_\mu$ is the projection operator on  eigen-space associated to $H$. By the Newton-Leibniz formula, we have
\begin{align*}
&E(Iu(t))-E(Iu(0))\\
=&\Re\sum_{\mu_i}\int_{0}^t\int_{\R^d}m(\mu_1)\overline{{\pi_{\mu_1}}u_t}\big(m(\mu_2)m(\mu_3)m(\mu_4)-I\big){\pi}_{\mu_2}u\overline{{\pi}_{\mu_3}u}{\pi}_{\mu_4}udxdt,
\end{align*}
where $\mu_k$ is the rearranged eigenvalue associated to $\sqrt H$. Using the fact that $I$ is a self-adjoint multiplier operator, we have
\begin{align*}
&E(Iu(t))-E(Iu(0))\\
=&\Re\sum_{\mu_i}\int_{0}^t\int_{\R^d}\Big(1-\frac{m(\mu_1)}{m(\mu_2)m(\mu_3)m(\mu_4)}\Big){\pi}_{\mu_1}(Iu_t){\pi}_{\mu_2}(Iu)\overline{{\pi}_{\mu_3}(Iu)}{\pi}_{\mu_4}(Iu)dxdt.
\end{align*}
Then substituting the equation of $Iu$, then we can rewrite 
\begin{align*}
&E(Iu(t))-E(Iu(0))\\
=&\Re i\sum_{\mu_i}\int_{0}^t\int_{\R^d}\Big(1-\frac{m(\mu_1)}{m(\mu_2)m(\mu_3)m(\mu_4)}\Big)\overline{{\pi}_{\mu_1}(H Iu)}{\pi}_{\mu_2}(Iu)\overline{{\pi}_{\mu_3}(Iu)}{\pi}_{\mu_4}(Iu)dxdt\\
&+\Re i\sum_{\mu_i}\int_{0}^t\int_{\R^d}\Big(1-\frac{m(\mu_1)}{m(\mu_2)m(\mu_3)m(\mu_4)}\Big)\overline{{\pi}_{\mu_1}(I(|u|^2u))}{\pi}_{\mu_2}(Iu)\overline{{\pi}_{\mu_3}(Iu)}{\pi}_{\mu_4}(Iu)dxdt\\
\stackrel{\triangle}{=}&\Re i(I_1+I_2).
\end{align*}
			
\noindent\textbf{The bound of term $\mathbf{I_1}$}. For the term $I_1$, we expect to show that
\begin{align*}
|I_1|\lesssim C\big(\|Iu\|_{X^{1,\frac{1}{2}+}}\big) N^{-s_0+}.
\end{align*}
To achieve this, we split $u$ into the sum of  dyadic pieces $u(t,x)=\sum\limits_{N\in2^{\Bbb N}}u_N$ with $u_N={\Delta}_Nu$. Substituting the decomposition in $I_1$, we can rewrite it as the sum of the frequency-localized terms at different scales $N$. From the Bernstein estimate, we have
\begin{align*}
\big\|H Iu\big\|_{X^{-1,\frac{1}{2}+}}\lesssim \|Iu\|_{X^{1,\frac{1}{2}+}}.
\end{align*}
It remains to show that
\begin{align}\nonumber
&\left|\sum_{\mu_i\sim N_i}\int_{0}^t\int_{\R^d}\big(1-\frac{m(\mu_1)}{m(\mu_2)m(\mu_3)m(\mu_4)}\big)\pi_{\mu_1}(v_1)\pi_{\mu_2}(v_2)\pi_{\mu_3}(v_3)\pi_{\mu_4}(v_4)dxdt\right|\\\label{fml-almost-reduction}
\lesssim& N^{-s_0+}(N_1N_2N_3N_4)^{0-}\|v_1\|_{X^{-1,\frac{1}{2}+}}\prod_{i=2}^{4}\|v_i\|_{X^{1,\frac{1}{2}+}},
\end{align}
where $v_j$ is frequency localized around $N_j$. Without loss of generality, we assume that 
\begin{align*}
1\leq N_2\leq N_3\leq N_4.
\end{align*}
Further, we can divide  into the following four cases:
\begin{enumerate}
\item[$\circ$]  Case One: All frequencies are low: $N_2\ll N$(trivial case).
\item[$\circ$]  Case Two: The first two frequencies are higher than $N$: $N_2\gtrsim N\gg N_3\geq N_4$.
\item[$\circ$]  Case Three: Only $N_4$ is lower than $N$: $N_2\sim N_3\gtrsim N$.
\item[$\circ$]  Case Four:  All frequencies are higher than $N$: $N_2\gg N_3\gtrsim N$.
\end{enumerate}
Indeed, by the following proposition, we can set $N_1\lesssim N_2$ without loss of generality. We will estimate the term $I_1$  via the case-by-case analysis.

\textbf{Case One}: Notice that the multiplier is vanished since all frequency are lower than $N$. Then we finish the proof of Case One.
			
\textbf{Case Two}: First, we assume that  $N_1$ and $N_2$ are comparable. Indeed, if $N_1\geq CN_2$ or $N_2\geq CN_1$ for $C\gg1$, we can establish the following fast decay estimate.

\begin{lemma}[Fast decay estimate]
There exists $C>0$ such that if $N_1\geq CN_2$, then for any $r>0$, there holds
\begin{align*}
\mbox{(LHS) of }\eqref{fml-almost-reduction}\lesssim N_1^{-r}\prod_{i=1}^{4}\|v_i\|_{X^{0,\frac{1}{2}+}([0,\delta]\times\R^d)}.
\end{align*}
\end{lemma}

\begin{proof}
Using  Minkowski's inequality and Lemma \ref{lem:orteigde}, we obtain
\begin{align*}
\mbox{(LHS) of }\eqref{fml-almost-reduction}&\leq\sum_{\mu_i\sim N_i}\left|\left(1-\frac{m(\mu_1)}{m(\mu_2)m(\mu_3)m(\mu_4)}\right)\int_0^t\int_{\R^d}\overline{\pi_{\mu_1}({v_1})}\pi_{\mu_2}({v_2})\overline{\pi_{\mu_3}({v_3})}\pi_{\mu_4}({v_4})dxdt\right|\\
&\lesssim\sum_{\mu_i\sim N_i}\mu_1^{-r}A(\mu_1,\mu_2,\mu_3,\mu_4)\int_0^t\prod_{j=1}^{4}\|\pi_{\mu_i}{v_i}(t)\|_{L^2(\R^d)}dt,
\end{align*}
where the multiplier $A$ enjoys the bound
\begin{align*}
A(\mu_1,\mu_2,\mu_3,\mu_4)=\left|1-\frac{m(\mu_1)}{m(\mu_2)m(\mu_3)m(\mu_4)}\right|\lesssim \mu_1^2.
\end{align*}
It's easy to check that
\begin{align*}
\#S:=\#\big\{\mu:N_i\leqslant\sqrt{\mu}\leqslant 2N_i,\; \mu\mbox{ is the eigenvalue of }H\big\}\lesssim N_i^10.
\end{align*}
Using Cauchy-Schwarz's inequality, we have
\begin{align*}
\sum_{\mu_i\sim N_i}\big\|\pi_{ \mu_i}{v_i}(t)\big\|_{L^2(\R^d)}\lesssim(\#S) ^\frac{1}{2}\|v_i\|_{L^2}\lesssim N_i^5\|v_i\|_{L^2}.
\end{align*}
Putting these estimates together,  we obtain
\begin{align*}
&\operatorname{(LHS)\, of }\eqref{fml-almost-reduction}\\
\lesssim&\frac{1}{N_1^{r-2}}\int_{0}^t\prod_{i=1}^{4}N_i^5\|v_i\|_{L^2(\R^d)}dt\lesssim N_1^{22-r}\prod_{i=1}^{4}\|v_i\|_{X^{0,\frac{1}{2}+}}.
\end{align*}
For sufficient large $r$, we have that 
\begin{align*}
\mbox{(LHS) of }\eqref{fml-almost-reduction}\lesssim N_1^{-r}\prod_{i=1}^4\|v_i\|_{X^{0,\frac{1}{2}+}}.	
\end{align*}
\end{proof}

With this lemma in hand, the Bernstein inequality yields the desired estimate if $N_1\gg N_2$ or $N_1\ll N_2$. Thus, we assume that  $N_1$ and $N_2$ are comparable.  We assume the order relation $N_1\ge N_2\gtrsim N \gg N_3,N_4$.   We will prove that
\begin{equation}\label{fml-EI-maingoal}
\begin{aligned}
&\Lf\vt \sum_{\mu_i\sim N_i}\int_0^T\int \Lf[ 1-\frac{m(\mu_1)}{m(\mu_2)m(\mu_3)m_N(\mu_4)}\Rt] \ovr{\pi_{\mu_1}(v_1)}\pi_{\mu_2}(v_2)\ovr{\pi_{\mu_3}(v_3)}\pi_{\mu_4}(v_4) \Ind{x}\Ind{t} \Rt\vt\\
\lesssim& N^{-\frac{3}{2}+} (N_1N_2N_3N_4)^{0-} \Vt v_1\Vt_{\TsnCHself{2}{-1}{[0,T]}} \prod_{i=2}^4\Vt v_i\Vt_{\TsnCHself{2}{1}{[0,T]}}.
\end{aligned}
\end{equation}
			
Then there exists $N_1 = N_2 + R N_3$ for some $R\in\mathbb{N}$. Let $a,b\in\mathbb{N}$ and we set 
\begin{equation*}
I_a := [N_2+a N_3,N_2+(a+1)N_3),\,\, I_b := [N_2+b N_3,N_2+(b+1)N_3).
\end{equation*}
Then, we divide interval $[N_1,2N_1)$ and $[N_2,2N_2)$ as follows 
\begin{equation}
[N_1,2N_1) = \bigcup_{a=R}^{R+J-1} I_a,\,\, [N_2,2N_2) = \bigcup_{b=0}^{K} I_b,
\end{equation}
where $J \le N_1N_3^{-1}$ and $K \le N_2N_3^{-1}$. Now, we classify them according to the relationship of the position of $a,b$:
\begin{align*}
&S_1 := \{ (I_a,I_b):\; a-b > 8,\,\, {I_a \subset [N_1,2N_1),\,\,I_b \subset [N_2,2N_2)} \},\\
&S_2 := \{ (I_a,I_b):\; a-b < -8,\,\, I_a \subset  [N_1,2N_1),\,\,I_b \subset  [N_2,2N_2) \},\\
&S_3 := \{ (I_a,I_b):\; \vt a-b\vt \le 8,\,\, I_a \subset  [N_1,2N_1),\,\,I_b \subset  [N_2,2N_2) \}.
\end{align*}
			
We consider $S_1$ first. Applying Theorem \ref{thm-inter4}, we get that 
\begin{align*}
&\sum_{S_1}\int_0^T\int \Lf[ 1-\frac{m(\mu_1)}{m(\mu_2)m(\mu_3)m(\mu_4)}\Rt] \ovr{\pi_{\mu_1}(v_1)}\pi_{\mu_2}(v_2)\ovr{\pi_{\mu_3}(v_3)}\pi_{\mu_4}(v_4) \Ind{x}\Ind{t}\\
=& \sum_{S_1}  \Lf[ 1-\frac{m(\mu_1)}{m(\mu_2)m(\mu_3)m(\mu_4)}\Rt] \frac{(-2)^{\ell}}{(\mu_1^2-\mu_2^2-\mu_3^2-\mu_4^2)^{\ell}}\\
&\times \int_0^T\int \ovr{\pi_{\mu_1}v_1}\times\wt{R}_{\ell}(\pi_{\mu_2}v_2,\ovr{\pi_{\mu_3}v_3},\pi_{\mu_4}v_4) \Ind{x}\Ind{t}.
\end{align*}
Recall that $\wt{R}_{\ell}(\pi_{\mu_2}v_2,\ovr{\pi_{\mu_3}v_3},\pi_{\mu_4}v_4)$ is the linear combination of the form
\begin{equation*}
P(\alpha)\pi_{\mu_2}v_2P(\beta)\ovr{\pi_{\mu_3}v_3}P(\gamma)\pi_{\mu_4}v_4,
\end{equation*}
where $P \in \{x,\nabla\}$ and $\mathrm{Ord}(\alpha)+\mathrm{Ord}(\beta)+\mathrm{Ord}(\gamma)\le 2\ell$. It is sufficient to estimate 
\begin{equation}\label{fml-EI-Case2-PPP}
\begin{aligned}
&\Lf[ 1-\frac{m(\mu_1)}{m(\mu_2)m(\mu_3)m(\mu_4)}\Rt]\frac{(-2)^{\ell}}{(\mu_1^2-\mu_2^2-\mu_3^2-\mu_4^2)^{\ell}} \\
&\times\sum_{S_1} \int_0^T\int \ovr{\pi_{\mu_1}(v_1)}\big[P(\alpha)\pi_{\mu_2}v_2P(\beta)\ovr{\pi_{\mu_3}v_3}P(\gamma)\pi_{\mu_4}v_4\big] \Ind{x}\Ind{t}.
\end{aligned}
\end{equation}
			
Then, we denote the multiplier as
\begin{equation*}
\wt{m}(\nof) := \Lf[ 1-\frac{m(\mu_1)}{m(\mu_2)m(\mu_3)m(\mu_4)}\Rt]\frac{(-2)^{\ell}}{(\mu_1^2-\mu_2^2-\mu_3^2-\mu_4^2)^{\ell}}.
\end{equation*}
Since the linear combination of \eqref{fml-EI-Case2-PPP} contains finite elements, we only need to estimate
\begin{equation}\label{fml-EI-Case2-piece}
\sum_{S_1}\wt{m}(\mu_1,\mu_2,\mu_3,\mu_4) \int_0^T\int \ovr{\pi_{\mu_1}(v_1)}\big[P(\alpha)\pi_{\mu_2}v_2P(\beta)\ovr{\pi_{\mu_3}v_3}P(\gamma)\pi_{\mu_4}v_4\big] \Ind{x}\Ind{t}.
\end{equation}
			
By change of variable
\begin{equation}\label{fml-EI-variablechange}
\begin{gathered}
\wt{\mu}_1 := { \frac{\mu_1-N_2-a N_3}{N_3}},\,\,
\wt{\mu}_2 :=  {\frac{\mu_2-N_2-b N_3}{N_3}},\,\,
\wt{\mu}_3 := \frac{\mu_3-N_3}{N_3},\,\,
\wt{\mu}_4 := \frac{\mu_4-N_4}{N_4}.
\end{gathered}
\end{equation}
Since $\wt{\mu}_j \in [0,1]$, then there exists a function $\Psi : [0,1]^4 \to \R$ such that
\begin{equation*}
\Psi(\ntof) = \wt{m}(\nof).
\end{equation*}
			
We extend $\Psi$ to smooth and compactly supported function on $[-2,2]^4$ and it can be also regarded as $4$-periodic function on $\R^4$. Therefore, we can expand $\Psi$ as Fourier series
\begin{equation*}
\Psi(\ntof) = \sum_{\theta_i\in \Z/4} e^{i(\theta_1\wt{\mu}_1+\theta_2\wt{\mu}_2+\theta_3\wt{\mu}_3+\theta_4\wt{\mu}_4)} B(\thetaof),
\end{equation*}
and the coefficient can be bounded by \cite[Theorem 3.2.16]{Grafacus}
\begin{equation}\label{equ:AthetaestGra}
\sum_{\theta_i\in \Z/4} \vt B(\thetaof)\vt \lesssim \Vt \Psi\Vt_{C^2([0,1]^4)}.
\end{equation}
			
With this in hand, \eqref{fml-EI-Case2-piece} can be rewritten as
\begin{equation*}
{\sum_{I_a,I_b\subset S_1}\sum_{\theta_i\in \Z/4}B(\theta_1,\theta_2,\theta_3,\theta_4) \int_0^T\int P_{I_a}v^{\theta_1}_1\big[P(\alpha)(P_{I_b}v^{\theta_2}_2)\ovr{P(\beta)(\pi_{\mu_3}v^{\theta_3}_3)}P(\gamma)(\pi_{\mu_4}v^{\theta_4}_4)\big] \Ind{x}\Ind{t},}
\end{equation*}
with $v_j^{\theta_j} = e^{i\theta_j \wt{\mu}_j}v_j$ and 
\begin{equation*}
	P_{I_a} := \sum_{\mu_a \in I_a} \pi_{\mu_a},\,\,\, P_{I_b} := \sum_{\mu_b \in I_b} \pi_{\mu_b}.
\end{equation*}
			
We assert that
\begin{equation}\label{fml-EI-asser-A}
\sum_{\theta_i\in \Z/4}\vt B(\theta_1,\theta_2,\theta_3,\theta_4)\vt \lesssim (a-b)\frac{N_3}{N_2} \frac{1}{(N_2N_3(a-b))^{\ell}}.
\end{equation}
For fixed $(\theta_1,\theta_2,\theta_3,\theta_4)\in [\Z/4]^4$, by Lemma \ref{lem:BilBourg}, we get 
\begin{equation*}
\begin{aligned}
&\Big|\int_0^T\int P_{I_a}v^{\theta_1}_1\big[P(\alpha)(P_{I_b}v^{\theta_2}_2)\ovr{P(\beta)(\pi_{\mu_3}v^{\theta_3}_3)}P(\gamma)(\pi_{\mu_4}v^{\theta_4}_4)\big] \Ind{x}\Ind{t}\Big|\\
\lesssim& \Vt P_{I_a}v^{\theta_1}_1 P(\gamma)(\pi_{\mu_4}v^{\theta_4}_4)\Vt_{L_t^2([0,T],L_x^2(\R^d))} \Vt P(\alpha)(P_{I_b}v^{\theta_2}_2) P(\beta)(\pi_{\mu_3}v^{\theta_3}_3)\Vt_{L_t^2([0,T],L_x^2(\R^d))}\\
\lesssim& (N_2N_3)^{\ell} \frac{N_3^{\frac{d-1}{2}}N_4^{\frac{d-1}{2}}}{N_1^{\frac{1}{2}}N_2^{\frac{1}{2}}} \Vt P_{I_a}v_1\Vt_{X^{0,\frac{1}{2}+}} \Vt P_{I_b}v_2\Vt_{X^{0,\frac{1}{2}+}} \Vt v_3\Vt_{X^{0,\frac{1}{2}+}} \Vt v_4\Vt_{X^{0,\frac{1}{2}+}}.
\end{aligned}
\end{equation*}
Combining the above with \eqref{fml-EI-asser-A}, \eqref{fml-EI-Case2-PPP} enjoys the bound
\begin{equation*}
\begin{aligned}
&\sum_{I_a,I_b\subset S_1} \Big( \sum_{\theta_i\in[\Z/4]^4} \vt B(\theta_1,\theta_2,\theta_3,\theta_4)\vt\Big) (N_2N_3)^{\ell} \frac{N_3^{\frac{d-1}{2}}N_4^{\frac{d-1}{2}}}{N_1^{\frac{1}{2}}N_2^{\frac{1}{2}}}\\
&\hspace{36ex}\times\Vt P_{I_a}v_1\Vt_{X^{0,\frac{1}{2}+}} \Vt P_{I_b}v_2\Vt_{X^{0,\frac{1}{2}+}} \Vt v_3\Vt_{X^{0,\frac{1}{2}+}} \Vt v_4\Vt_{X^{0,\frac{1}{2}+}}\\
&\lesssim \sum_{I_a,I_b\subset S_1} \vt a-b\vt \frac{N_3}{N_2} \frac{(N_2N_3)^{\ell}}{(N_2N_3(a-b))^{\ell}} \frac{N_3^{\frac{d-1}{2}}N_4^{\frac{d-1}{2}}}{N_1^{\frac{1}{2}}N_2^{\frac{1}{2}}}\\
&\hspace{36ex}\times\Vt P_{I_a}v_1\Vt_{X^{0,\frac{1}{2}+}} \Vt P_{I_b}v_2\Vt_{X^{0,\frac{1}{2}+}} \Vt v_3\Vt_{X^{0,\frac{1}{2}+}} \Vt v_4\Vt_{X^{0,\frac{1}{2}+}}\\
&\lesssim \frac{N_3}{N_2} \frac{N_3^{\frac{d-1}{2}}N_4^{\frac{d-1}{2}}}{N_1^{\frac{1}{2}}N_2^{\frac{1}{2}}} \sum_{a-b>8}\frac{1}{(a-b)^{\ell-1}} \Vt P_{I_a}v_1\Vt_{X^{0,\frac{1}{2}+}} \Vt P_{I_b}v_2\Vt_{X^{0,\frac{1}{2}+}} \Vt v_3\Vt_{X^{0,\frac{1}{2}+}} \Vt v_4\Vt_{X^{0,\frac{1}{2}+}}. 
\end{aligned}
\end{equation*}
Taking $\ell = 3$ guarantees summability over $a,b$, finally \eqref{fml-EI-Case2-PPP} can be controlled by
\begin{equation}\label{fml-EI-Case2-Incre}
\begin{aligned}
&\frac{N_3}{N_2} \frac{N_3^{\frac{d-1}{2}}N_4^{\frac{d-1}{2}}}{N_1^{\frac{1}{2}}N_2^{\frac{1}{2}}} \frac{N_1}{N_2N_3N_4} \Vt v_1\Vt_{X^{-1,\frac{1}{2}+}}\prod_{i=2}^4 \Vt v_i\Vt_{X^{1,\frac{1}{2}+}}.
\end{aligned}
\end{equation}
For $d=2$
\begin{equation*}
\eqref{fml-EI-Case2-Incre} \lesssim \frac{N_1^{\frac{1}{2}}N_3^{\frac{1}{2}}}{N_2^{\frac{5}{2}} N_4^{\frac{1}{2}}}\Vt v_1\Vt_{X^{-1,\frac{1}{2}+}}\prod_{i=2}^4 \Vt v_i\Vt_{X^{1,\frac{1}{2}+}} \lesssim N^{-\frac{3}{2}+}(N_1N_2N_3N_4)^{0+}\Vt v_1\Vt_{X^{-1,\frac{1}{2}+}}\prod_{i=2}^4 \Vt v_i\Vt_{X^{1,\frac{1}{2}+}},
\end{equation*}
while for $d=3$ 
\begin{equation*}
\eqref{fml-EI-Case2-Incre} \lesssim \frac{N_1^{\frac{1}{2}}N_3^1}{N_2^{\frac{5}{2}}}\Vt v_1\Vt_{X^{-1,\frac{1}{2}+}}\prod_{i=2}^4 \Vt v_i\Vt_{X^{1,\frac{1}{2}+}} \lesssim N^{-1+}(N_1N_2N_3N_4)^{0+}\Vt v_1\Vt_{X^{-1,\frac{1}{2}+}}\prod_{i=2}^4 \Vt v_i\Vt_{X^{1,\frac{1}{2}+}}.
\end{equation*}
Therefore, we have finished this case under the assumption of \eqref{fml-EI-asser-A}. 
			
Next, we show \eqref{fml-EI-asser-A}. Using \eqref{equ:AthetaestGra}, we are reduced to show
\begin{equation}\label{fml-EI-C2-bound}
\Vt\Psi\Vt_{C^2([0,1]^4)} \lesssim (a-b) \frac{N_3}{N_2} \frac{1}{(N_2N_3(a-b))^{\ell}}.
\end{equation}
Since $N \gg N_3,N_4$ here, it is sufficient to estimate point-wise bound for
\begin{equation}\label{fml-EI-n1234}
\Lf\vt \frac{\partial^k}{\partial(\ntof)^k}\Lf( \frac{2^\ell}{(\mu_1^2-\mu_2^2-\mu_3^2-\mu_4^2)^\ell}\Rt) \Rt\vt,
\end{equation}
and 
\begin{equation}\label{fml-EI-mn1mn2}
\Lf\vt \frac{\partial^k}{\partial(\wt{\mu}_1,\wt{\mu}_2)^k}\Lf( 1-\frac{m(N_2+\alpha N_3+N_3\wt{\mu}_1)}{m(N_2+\beta N_3+N_3\wt{\mu}_2)} \Rt) \Rt\vt,
\end{equation}
with $\vt k\vt \le 2$ respectively. 
			
We begin with \eqref{fml-EI-mn1mn2}. For all $\ell\in\N$, one has
\begin{equation}\label{fml-EI-m-dere}
{\Lf\vt\frac{\partial^\ell m(\xi)}{m(\xi)}\Rt\vt \lesssim \frac{1}{\vt \xi\vt^{\ell}}.}
\end{equation}
Combining \eqref{fml-EI-m-dere} and the mean value theorem, we get that
\begin{equation*}
\Lf\vt 1 - \frac{m(\mu_1)}{m(\mu_2)}\Rt\vt \sim \vt \mu_1-\mu_2\vt \frac{m^{\prime}(\mu_2)}{m(\mu_2)} \sim \vt a-b\vt\frac{N_3}{N_2},
\end{equation*}
and
\begin{equation*}
\Lf\vt \frac{\partial^k}{\partial(\ntof)^k}\Lf( 1 - \frac{m(N_2+\alpha N_3+N_3\wt{\mu}_1)}{m(N_2+\beta N_3+N_3\wt{\mu}_2)} \Rt) \Rt\vt \lesssim \Lf(\frac{N_3}{N_2}\Rt)^{\vt k\vt} \lesssim \vt a-b\vt \frac{N_3}{N_2}.
\end{equation*}
Hence, we arrive at
\begin{equation}\label{fml-EI-n1234-bound}
\eqref{fml-EI-mn1mn2} \lesssim \vt a-b\vt \frac{N_3}{N_2}.
\end{equation}
			
Next, we turn to \eqref{fml-EI-n1234}. Substituting \eqref{fml-EI-variablechange} into $\mu_1^2 - \mu_2^2-\mu_3^2-\mu_4^2$ we have
\begin{equation*}
\begin{aligned}
\mu_1^2 - \mu_2^2-\mu_3^2-\mu_4^2 &= 2N_2N_3(a-b+(\wt{\mu}_1+\wt{\mu}_2)) + N_3^2(a^2-b^2)+ N_3^2(\wt{\mu}_1^2-\wt{\mu}_2^2) + 2N_3^2 (a\wt{\mu}_1-\beta\wt{\mu}_2)\\
&\hspace{2ex}-N^2_3 (1+\wt{\mu}_3)^2 -N_4^2(1+\wt{\mu}_4)^2\\
&=\big( 2N_2N_3 (a-b) + (a^2-b^2)N_3^2 \big) G(\ntof),
\end{aligned}
\end{equation*}
where $G(\ntof) : \R^4 \to (0,1)$. Since $a-b>8$, $G(\ntof)$ enjoys lower and upper bound
\begin{equation*}
\frac{1}{2} \le G(\ntof) \lesssim 1.
\end{equation*} 
Moreover, for all multi-index $k$, there holds
\begin{equation*}
\Lf\vt\frac{\partial^kG(\ntof)}{\partial(\ntof)^k}\Rt\vt \lesssim_{\vt k\vt} 1.
\end{equation*}
We conclude that
\begin{equation}\label{fml-EI-mn1mn2-bound}
\eqref{fml-EI-n1234} \lesssim \frac{1}{\big( 2N_2N_3 (a-b) + (a^2-b^2)N_3^2 \big)^{\ell}} \lesssim \frac{1}{\big(N_2N_3(a-b)\big)^{\ell}}.
\end{equation}
			
Combining \eqref{fml-EI-n1234-bound} with \eqref{fml-EI-mn1mn2-bound}, we obtain \eqref{fml-EI-C2-bound}. So far, we have finished summation on $S_1$. We also obtain the same bound for summation on $S_2$ by similar argument with that of $S_1$.
			
Finally, we deal with summation on $S_3$. We write that
\begin{align}\label{fml-EI-S3-eq}
&\sum_{S_3}\int_0^T\int \Lf[ 1-\frac{m(\mu_1)}{m(\mu_2)m(\mu_3)m(\mu_4)}\Rt] \ovr{\pi_{\mu_1}(v_1)}\pi_{\mu_2}(v_2)\ovr{\pi_{\mu_3}(v_3)}\pi_{\mu_4}(v_4) \Ind{x}\Ind{t}\\\nonumber
=&\frac{N_3}{N_2} \sum_{I_b} \sum_{I_a} \sum_{\substack{\mu_1\in I_a,\mu_2\in I_b\\ \mu_3\sim N_3, \mu_4\sim N_4}} \frac{N_3}{N_2} \int_0^T\int \frac{N_2}{N_3}\Lf[ 1-\frac{m(\mu_1)}{m(\mu_2)m(\mu_3)m(\mu_4)}\Rt] \ovr{\pi_{\mu_1}(v_1)}\pi_{\mu_2}(v_2)\ovr{\pi_{\mu_3}(v_3)}\pi_{\mu_4}(v_4) \Ind{x}\Ind{t}.
\end{align}
			
It is clear that the multiplier
\begin{equation*}
\wt{m}(\nof) = \frac{N_2}{N_3} \Lf[ 1-\frac{m(\mu_1)}{m(\mu_2)m(\mu_3)m(\mu_4)}\Rt],
\end{equation*}
meets the condition in Theorem \ref{thm-multi}. Then, applying Lemma \ref{lem:BilBourg}, we obtain
\begin{align*}
&\sum_{\substack{\mu_1\in I_a,\mu_2\in I_b\\ \mu_3\sim N_3, \mu_4\sim N_4}} \Big| \int_0^T\int \frac{N_2}{N_3}\Lf[ 1-\frac{m(\mu_1)}{m(\mu_2)m(\mu_3)m(\mu_4)}\Rt] \ovr{\pi_{\mu_1}(v_1)}\pi_{\mu_2}(v_2)\ovr{\pi_{\mu_3}(v_3)}\pi_{\mu_4}(v_4) \Ind{x}\Ind{t}\Big|\\
&\lesssim \Vt \pi_{\mu_1}(v_1)\pi_{\mu_4}(v_4)\Vt_{L_t^2([0,T],L_x^2(\R^d))}\Vt \pi_{\mu_2}(v_2)\pi_{\mu_3}(v_3)\Vt_{L_t^2([0,T],L_x^2(\R^d))}\\
&\lesssim \frac{N_3^{\frac{d-1}{2}}N_4^{\frac{d-1}{2}}}{N_1^{\frac{1}{2}}N_2^{\frac{1}{2}}} \Vt \pi_{\mu_1}v_1\Vt_{X^{-1,b}} \prod_{i=2}^4\Vt \pi_{n_i}v_i\Vt_{X^{1,b}}.
\end{align*}
Hence, \eqref{fml-EI-S3-eq} can be controlled by
\begin{equation*}
\begin{aligned}
&\sum_{I_b} \sum_{I_a} \frac{N_3}{N_2} \frac{N_3^{\frac{d-1}{2}}N_4^{\frac{d-1}{2}}}{N_1^{\frac{1}{2}}N_2^{\frac{1}{2}}}\frac{N_1}{N_2N_3N_4} \Vt P_{I_a}v_1\Vt_{X^{-1,b}}\Vt P_{I_b}v_2\Vt_{X^{1,b}}\prod_{i=2}^4 \Vt v_i\Vt_{X^{1,b}}\\
\lesssim& \frac{N_3}{N_2} \frac{N_3^{\frac{d-1}{2}}N_4^{\frac{d-1}{2}}}{N_1^{\frac{1}{2}}N_2^{\frac{1}{2}}}\frac{N_1}{N_2N_3N_4} \Vt v_1\Vt_{X^{-1,b}}\prod_{i=2}^4 \Vt v_i\Vt_{X^{1,b}},
\end{aligned}
\end{equation*}
where in the last inequality, note that there are finite number of $I_b$ once we fix $I_a$. 
			
We conclude that
\begin{equation*}
\mbox{(LHS) of } \eqref{fml-EI-S3-eq} \lesssim N^{-s_0+}(N_1N_2N_3N_4)^{0+}\Vt v_1\Vt_{X^{-1,\frac{1}{2}+}}\prod_{i=2}^4 \Vt v_i\Vt_{X^{1,\frac{1}{2}+}},
\end{equation*}
Therefore, we have finished the summation on $S_1,S_2$ and $S_3$.

\textbf{Case 3: }$N_2\sim N_3\gtrsim N$.
			
In this case, we need to compare $N$ and $N_1$ which will impact the bound of the multiplier $A$. Using the spectral multiplier theorem, bilinear Strichartz estimate, Bernstein's inequality and Cauchy-Schwarz, we have
\begin{align*}
\mbox{(LHS) of }\eqref{fml-almost-reduction}&\lesssim \frac{m(N_1)}{m(N_2)m(N_3)m(N_4)}\frac{N_1}{N_2N_3N_4}\frac{(N_1N_4)^\frac{d-1}{2}}{(N_2N_3)^\frac{1}{2}}\|v_1\|_{X^{-1,\frac{1}{2}+}}\prod_{i=2}^{4}\|v_i\|_{X^{1,\frac{1}{2}+}}.
\end{align*}
Then, we divide the frequencies into the following two subcases.
			
\noindent\textbf{Subcase 3(a)}: $N_2\sim N_3$, $N_1\gg N$. By definition, we have
\begin{align*}
\frac{m(N_1)}{m(N_2)m(N_3)m(N_4)}\lesssim\frac{N^{1-s}N_1^{-(1-s)}}{N^{2(1-s)}(N_2N_3)^{s-1}m(N_4)}.
\end{align*}
Then we obtain
\begin{align*}
\mbox{(LHS) of }\eqref{fml-almost-reduction}&\lesssim\frac{N_1^{s+\frac{d-1}{2}}}{N^{1-s}(N_2N_3)^{s+\frac{1}{2}}m(N_4)N_4^{\frac{3-d}{2}}}\|v_1\|_{X^{-1,\frac{1}{2}+}}\prod_{i=2}^{4}\|v_i\|_{X^{1,\frac{1}{2}+}}\\
&\lesssim \frac{1}{N^{1-s}N_2^{2s+1-s-\frac{d-1}{2}-}m(N_4)N_4^{0+}}\|v_1\|_{X^{-1,\frac{1}{2}+}}\prod_{i=2}^{4}\|v_i\|_{X^{1,\frac{1}{2}+}}.
\end{align*}
We notice that if $N_4 \ge N$, for any $k>0$ and $s\ge 1-k$
\begin{equation*}
m(N_4)N_4^k = N^{1-s}N_4^{s+k-1} \ge N^k, 
\end{equation*}
and it gives that $m(N_4)N_4^{\frac{1}{2}}\gtrsim N^{\frac{1}{2}}$ and $m(N_4)N_4^{0+}\gtrsim N_4^{0+}$. Therefore, we get 
\begin{align*}
\mbox{(LHS) of }\eqref{fml-almost-reduction}\lesssim N^{-s_0+}\|v_1\|_{X^{-1,\frac{1}{2}+}}\prod_{i=2}^{4}\|v_i\|_{X^{1,\frac{1}{2}+}}.
\end{align*} 
			
\noindent\textbf{Subcase 3(b)}: $N_2\sim N_2\geq N$ and $N_1\lesssim N$. Using the bound for the multiplier
\begin{align*}
\frac{m(N_1)}{m(N_2)m(N_3)m(N_4)}\lesssim\frac{1}{N^{2(1-s)}(N_2N_3)^{s-1}m(N_4)},
\end{align*}
we have
\begin{align*}
\mbox{(LHS) of }\eqref{fml-almost-reduction}&\lesssim\frac{1}{N^{2(1-s)}(N_2N_3)^{s-1}m(N_4)}\frac{N_1}{N_2N_3N_4}\frac{(N_1N_4)^\frac{d-1}{2}}{(N_2N_3)^\frac{1}{2}}\|v_1\|_{X^{-1,\frac{1}{2}+}}\prod_{i=2}^{4}\|v_i\|_{X^{1,\frac{1}{2}+}}\\
&\lesssim  N^{-s_0+}\|v_1\|_{X^{-1,\frac{1}{2}+}}\prod_{i=2}^{4}\|v_i\|_{X^{1,\frac{1}{2}+}}.
\end{align*}
			
\textbf{Case 4:} $N_2\gg N_3\gtrsim N$. By the spectral multiplier theorem, bilinear Strichartz estimate, we have
\begin{align}\label{fml-1-Case4}
\mbox{(LHS) of }\eqref{fml-almost-reduction}&\lesssim\frac{m(N_1)}{m(N_2)m(N_3)m(N_4)}\frac{N_1}{N_2N_3N_4}\frac{(N_3N_4)^\frac{d-1}{2}}{(N_1N_2)^\frac{1}{2}}\|v_1\|_{X^{-1,\frac{1}{2}+}}\prod_{i=2}^{4}\|v_i\|_{X^{1,\frac{1}{2}+}}.
\end{align}
For $d=2$, we bound the above term by
\begin{align*}
\eqref{fml-1-Case4}&\lesssim\frac{1}{m(N_3)N_3^\frac{1}{2}m(N_4)N_4^\frac{1}{2}(N_1N_2)^\frac12}\|v_1\|_{X^{-1,\frac{1}{2}+}}\prod_{i=2}^{4}\|v_i\|_{X^{1,\frac{1}{2}+}}\\
&\lesssim\frac{1}{NN^{1-s}N_3^{s-\frac{1}{2}}}\|v_1\|_{X^{-1,\frac{1}{2}+}}\prod_{i=2}^{4}\|v_i\|_{X^{1,\frac{1}{2}+}}\\
&\lesssim N^{-\frac{3}{2}+}\|v_1\|_{X^{-1,\frac{1}{2}+}}\prod_{i=2}^{4}\|v_i\|_{X^{1,\frac{1}{2}+}}.
\end{align*}
For $d=3$, we have
\begin{align*}
\eqref{fml-1-Case4}&\lesssim \frac{1}{m(N_3)m(N_4)N_2}\|v_1\|_{X^{-1,\frac{1}{2}+}}\prod_{i=2}^{4}\|v_i\|_{X^{1,\frac{1}{2}+}}\\
&\lesssim N^ {-1+}\|v_1\|_{X^{-1,\frac{1}{2}+}}\prod_{i=2}^{4}\|v_i\|_{X^{1,\frac{1}{2}+}}.
\end{align*}
Combining the estimate together, we finish the proof of term $I_1$.
			
\subsection{Bound on term $I_2$}
For the term $I_2$, we expect to show the better bound
\begin{align*}
I_2\lesssim N^{-s_0+}\|Iu\|_{X^{1,\frac{1}{2}+}}^6.
\end{align*}
Since $I$ is  a self-adjoint operator, we can rewrite term $I_2$ as
\begin{align*}
I_2&=\sum_{\mu_i\sim N_i}\int_{0}^t\int_{\R^d}\big(1-\frac{m(\mu_1)}{m(\mu_2)m(\mu_3)m(\mu_4)}\big)\overline{\pi_{\mu_1}I(|u|^2u)}\pi_{\mu_2}(Iu)\overline{\pi_{\mu_3}(Iu)}\pi_{\mu_4}(Iu)dxdt\\
&=\int_{0}^t\int_{\R^d}I(|u|^2u)|Iu|^2Iudxdt-\int_{0}^t\int_{\R^d}|I^2(|u|^2u)u\overline{u}udxdt.
\end{align*}
Decomposing $u=\sum_{N}u_N$ into the sum of  Littlewood-Paley pieces as before and writing $I_2$ as
\begin{align}\label{term 2 dyadic}
I_2&=\sum_{N_i \in 2^{\Bbb{N}}}\Big(\int_0^t\int_{\R^d}
I\big(u_{N_1}\overline{u_{N_2}}u_{N_3}\big)Iu_{N_4}\overline{Iu_{N_5}}Iu_{N_6}dxdt \notag\\
&\hspace{10ex}-\int_0^t\int_{\R^d}
I^2\left(u _{N_1}\overline{u_{N_2}}u_{N_3}\right)u_{N_4}\overline{u_{N_5}}u_{N_6}dxdt\Big).
\end{align}
We denote by $N_{max}$ the largest among all $N_1,\ldots,N_6$ and $N_{med}$ the second one. We renumber $N_{j}$ for which delete $N_{max},N_{med}$ from it as $\tilde{N}_{j}$($\{\tilde{N}_1,\cdots,\tilde{N}_4\}=\{N_1,\cdots,N_6\}\backslash\{N_{max},N_{med}\}$). We will treat two cases below:
\begin{enumerate}
\item \textbf{Case one}: $N\gg N_{max}$ or $N_{max}\gg N_{med}$,
\item \textbf{Case two}: $N_{max}\gtrsim N,\quad N_{max}\sim N_{med}$.
\end{enumerate}
			
\textbf{Case One} : $N\gg N_{max}$ or $N_{max}\gg N_{med}$. We only give the details  for the case $N\gg N_{max}$ and remaining case $N_{max}\gg N_{med}$ is similar. Splitting $u_{N_1}u_{N_2}u_{N_3}$ by
\begin{equation}\label{high-low-separation}
u_{N_1}\overline{u_{N_2}}u_{N_3}=P_{\leq N}\left(u_{N_1}\overline{u_{N_2}}u_{N_3}\right)+P_{>N}\left(u_{N_1}\overline{u_{N_2}}u_{N_3}\right)
\end{equation}
and	substituting $\eqref{high-low-separation}$ for $\eqref{term 2 dyadic}$, we obtain
\begin{align}\label{high-low-term2}
I_2&=\sum_{N_i \in 2^{\Bbb N}}\left(\int_0^t\int_{\R^d}
I(P_{\le N}+P_{>N})\left(u_{N_1}\overline{u_{N_2}}u_{N_3}\right)Iu_{N_4}\overline{Iu_{N_5}}Iu_{N_6}\,dx\,dt \right.\nonumber\\
&\hspace{10ex}\left .-\int_0^t\int_{\R^d}
I^2(P_{\le N}+P_{>N})\left(u_{N_1}\overline{u_{N_2}}u_{N_3}\right)u_{N_4}\overline{u_{N_5}}u_{N_6}\,dx\,dt\right).\nonumber\\		
&=\sum_{N_i \in 2^{\N}}\bigg(\int_0^t\int_{\R^d}[(IP_{\le N}-I^{2}P_{\le N})+IP_{>N}-I^{2}P_{>N}]\left(u_{N_1}\overline{u_{N_2}}u_{N_3}\right)\notag\\
&\hspace{10ex}\times Iu_{N_4}\overline{Iu_{N_5}}Iu_{N_6}\,dx\,dt \bigg),
\end{align}
under the assumption that $N_{max}\ll N$. Noticing that $I_N,P_{>N}$ is self-adjoint and $L^{2}$ bounded, the almost orthogonal estimate $\eqref{fml-ortho}$ implies the following for every $r\in\Bbb N$,
\begin{align*}
\|IP_{>N}(u_{N_1}\overline{u_{N_2}}u_{N_3})\|_{L^{2}(\R^d)}=&\sup\limits_{\|v\|_{L^{2}}=1} \Big|\int_{\R^d}\overline{v}IP_{>N}(u_{N_1}\overline{u_{N_2}}u_{N_3})dx\Big| \\
=&\sup\limits_{\|v\|_{L^{2}_{x}}=1} \Big|\int_{\R^d}I(\overline{v})P_{>N}(u_{N_1}\overline{u_{N_2}}u_{N_3})dx\Big|\\
\lesssim&\sup\limits_{\|v\|_{L^{2}_{x}}=1}\left|\int_{\R^d}P_{>N}(I\overline{v})(u_{N_1}\overline{u_{N_2}}u_{N_3})dx\right|\\
\lesssim&N^{-r}\prod_{i=1}^{3}\|u_{N_{i}}\|_{L^{2}(\R^d)}.
\end{align*}
Using the fact that $I^2 u_{N_i}=I u_{N_i}=u_{N_i}$ if ${N_i}\ll N$, H\"{o}lder's inequality and the   Strichartz estimate, then the contribution of the term for which $N_{max}\ll N$ to high frequency part of $\eqref{high-low-term2}$ is
\begin{align*}
&\sum_{N_{i}\ll N}\Big|\int_0^t\int_{\R^d}IP_{>N}\left(u_{N_1}\overline{u_{N_2}}u_{N_3}\right) Iu_{N_4}\overline{Iu_{N_5}}Iu_{N_6}\,dx\,dt\Big|\\
\lesssim&\sum_{N_{i}\ll N}\int_0^t\|IP_{>N}(u_{N_1}\overline{u_{N_2}}u_{N_3})\|_{L_{x}^{2}(\R^d)}\|Iu_{N_4}\overline{Iu_{N_5}}Iu_{N_6}\|_{L^{2}_{x}(\R^d)}\,dt\\
\lesssim&\|IP_{>N}(u_{N_1}\overline{u_{N_2}}u_{N_3})\|_{L_{t}^\infty L_x^2([0,\delta]\times \R^d)}\prod_{i=4}^{6}{\|u_{N_{i}}\|_{L_{t}^{3}L_x^6([0,\delta]\times \R^d)}}\\
\lesssim& N^{-r}\prod_{i=1}^{3}\Vert u_{N_i}\Vert_{X^{0,\frac{1}{2}+}([0,\delta]\times \R^d)}\prod_{i=4}^{6}{\|u_{N_{i}}\|_{L_{t}^{3}L_x^6([0,\delta]\times \R^d)}}\\
\lesssim& N^{-(r-r_0)}\prod_{i=1}^{6}\|u_{N_{i}}\|_{X^{0,\frac{1}{2}+}([0,\delta]\times \R^d)},
\end{align*}
where $r_0=0$ if $d=2$ and $r_0=1$ if $d=3$. In the last inequality, we use the following Strichartz estimate, 
\begin{align*}
\Vert e^{itH}\Delta_Nf\Vert_{L_t^{3}L_x^6([0,\delta]\times \R^3)}\lesssim N^{\frac{1}{3}}\Vert e^{itH}\Delta_Nf\Vert_{L_t^3L_x^{\frac{18}{5}}}\lesssim N^{\frac{1}{3}}\Vert \Delta_Nf\Vert_{L^2(\R^3)}.
\end{align*}
			
Similar argument gives the same bound with $N^ {-r}\prod\limits_{i=1}^{6}\|v_{N_{i}}\|_{L^{2}_{x}}$ for the term $$\sum_{N_{i}\ll N}\int_0^t\int_{\R^d}I^{2}P_{>N}\left(u_{N_1}\overline{u_{N_2}}u_{N_3}\right) Iu_{N_4}\overline{Iu_{N_5}}Iu_{N_6}\,dx\,dt.$$
Meanwhile the contribution of the case for which $N_{max}\ll N$ to low frequency part of (\ref{high-low-term2}) vanishes due to $I\left(P_{\leq N}u_{N_1}\overline{u_{N_2}}u_{N_3}\right)=I^2\left(P_{\leq N}u_{N_1}\overline{u_{N_2}}u_{N_3}\right)$.
			
A similar argument using Theorem \ref{thm-ortho} shows that the contribution of $N_{max}\gg N_{med}$ is also harmless so we restrict attention to the case when $N_{max}\gtrsim N$ and $N_{max}\sim N_{med}$.
			
\textbf{Case two}: $N_{max}\gtrsim N,\hspace{1ex} N_{max}\sim N_{med}$.
			
We notice that $I$ is bounded as an $L^p$ multiplier, one can estimate the case of $N_{max}\gtrsim N$ in $d=2$ via H\"older's inequality as follows:
\begin{align*}
I_2\lesssim& \sum_{\substack{N_{max}, N_{med}\gtrsim N,\\  \tilde N_i\leq N_{max}}}||u_{N_{max}}||_{L_{t,x}^{4}}||u_{N_{med}}||_{L^{4}_{t,x}}\prod_{i=1}^4||u_{\tilde N_i}||_{L_{t,x}^{8}}.
\end{align*}

By using the Strichartz estimate
\begin{align*}
I_2&\lesssim \sum_{\substack{N_{max}, N_{med}\gtrsim N,\\ \tilde N_i\leq N_{max}}} ||u_{N_{max}}||_{X^{0,1/2+}}||u_{N_{med}}||_{X^{0,1/2+}}\prod_{i=1}^4 \tilde N_i^{\frac{1}{2}}||u_{\tilde N_i}||_{X^{0,1/2+}}\\
&\lesssim \frac{1}{m(N_{max})N_{max}m(N_{med})N_{med}}\Vert Iu_{max}\Vert_{X^{1,\frac{1}{2}+}}\Vert Iu_{med}\Vert_{X^{1,\frac{1}{2}+}}\prod_{j=1}^{4}\frac{1}{\tilde N_i^{\frac{1}{2}}m(\tilde{N}_i)}\Vert Iu_{\tilde{N}_j}\Vert_{X^{1,\frac{1}{2}+}}.
\end{align*}
For any $k>0$ and $\alpha>1-s$ for $s>1$, we have  that
\begin{equation}\label{index}
m(k)k^{\alpha}\gtrsim\left\{
\begin{aligned}
&1,\quad k\le N,\\
&N^{\alpha},\quad k\ge 2N.\\
\end{aligned}
\right	.
\end{equation}
Then we have
\begin{align*}
I_2&\lesssim \frac{1}{N^{2}}\sum_{N_{max}\in2^{\Bbb{N},N_{med},\tilde{N}_i\lesssim N_{max}}}N_{max}^{0-}\Vert Iu_{max}\Vert_{X^{1,\frac{1}{2}+}}\Vert Iu_{med}\Vert_{X^{1,\frac{1}{2}+}}\prod_{i=1}^{4}\Vert Iu_{\tilde{N}_j}\Vert_{X^{1,\frac{1}{2}+}}\\
&\lesssim N^{-2+}\Vert Iu\Vert_{X^{1,\frac{1}{2}+}}^6. 
\end{align*}
For the case of $d=3$, we have
\begin{align*}
I_2\lesssim& \sum_{\substack{N_{max}, N_{med}\gtrsim N,\\  \tilde N_i\leq N_{max}}} ||u_{N_{max}}||_{L_{t,x}^{\frac{30}{7}}}||u_{N_{med}}||_{L^{\frac{30}{7}}_{t,x}}\prod_{i=1}^4||u_{\tilde N_i}||_{L_{t,x}^{\frac{15}{2}}}\\
&\sum_{\substack{N_{max}, N_{med}\gtrsim N,\\  \tilde N_i\leq N_{max}}}  N_{max}^{0-}N_{max}^{-\frac{2}{3}+}N_{med}^{-\frac{2}{3}+} ||u_{N_{max}}||_{X^{1,\frac{1}{2}+}}||u_{N_{med}}||_{X^{1,\frac{1}{2}+}}\prod_{i=1}^4N_j^{-\frac{1}{6}}||u_{\tilde N_i}||_{X^{1,\frac{1}{2}+}}\\
\lesssim&\sum_{\substack{N_{max}, N_{med}\gtrsim N,\\  \tilde N_i\leq N_{max}}}N_{max}^{0-}N_{max}^{-\frac{2}{3}+}m(N_{max})^{-1}N_{med}^{-\frac{2}{3}+}m(N_{med})^{-1}\|Iu_{max}\|_{X^{1,\frac{1}{2}+}}\|Iu_{med}\|_{X^{1,\frac{1}{2}+}}\\
&\hspace{4ex}\times\prod_{j=1}^{4}m(N_j)^{-1}N_j^{-\frac{1}{6}}\|Iu_{\tilde{N}_i}\|_{X^{1,\frac{1}{2}+}}.
\end{align*}
It is clear that
\begin{equation*}
m(N_{max})^{-1}N_{max}^{-\frac{2}{3}+} m(N_{med})^{-1}N_{med}^{-\frac{2}{3}+}\times\prod_{j=1}^{4}m(N_j)^{-1}N_j^{\frac{1}{6}} \lesssim N^{-\frac{4}{3}+}.
\end{equation*}
Combining the estimate of $I_1$ and $I_2$ above, we conclude that the energy increment is at most $O(N^{-\frac{3}{2}+})$ in $d=2$ and $O(N^{-1})$ in $d=3$.

Therefore, we conclude the proof of Proposition \ref{prop-EI}.

\end{proof}

\subsection{Proof of Theorem \ref{thm1}}

Now, we are in position to prove Theorem \ref{thm1}. Let $u(t)$ be global solution to $\eqref{fml-NLS}$. Taking arbitrary $T>0$ and splitting the time interval $[0,T]$ into $\lfloor{\frac{T}{\delta}}\rfloor$ pieces of small slab $[0,\delta]$, where $\delta\gtrsim1$ be determined in the local theory. By the local well-posedness, decay of modified energy, we have that for $TN^{-s_0+}\lesssim1$, there holds
\begin{align*}
E(Iu)(t)\lesssim1.
\end{align*} 
Using the property of $I$-operator, we have
\begin{align*}
\|u(T)\|_{\mathcal H^{s}(\R^d)}\lesssim N^{s-1}\|Iu(T)\|_{\mathcal H^1(\R^d)}\lesssim N^{s-1}.
\end{align*}
In the last inequality, we use the fact that the $\mathcal H^s$ can be bounded by the modified energy.  In conclusion, we have proved that
\begin{align*}
\|u(T)\|_{\mathcal{H}^s(\R^d)}\lesssim (1+|t|)^{\tilde s_0(s-1)+},
\end{align*} 
where $\tilde{s}_0=\frac{2}{3}$ in $d=2$ and $\tilde{ s}_0=1$ in $d=3$.
		
Therefore, we conclude the proof of Theorem \ref{thm1}.



\end{document}